\documentclass[11pt]{amsart}
\usepackage{fancyhdr}
\usepackage{amsthm, amscd, mathdots, longtable, young}
\usepackage{makecell}
\usepackage{verbatim}  

\usepackage{setspace,amsthm,amsmath,amssymb,url}
\usepackage{geometry}        
\usepackage{graphicx}
\usepackage{amssymb}
\usepackage{color}
\usepackage{wrapfig}
 
\usepackage{epstopdf}
\usepackage{hyperref}
\usepackage{enumerate}
\usepackage{blkarray}
\usepackage{tikz}
\usetikzlibrary{arrows,matrix,positioning}
\usetikzlibrary{calc}
\newtheorem{theorem}{Theorem}[section]
\newtheorem{lemma}[theorem]{Lemma}

\newtheorem{corollary}[theorem]{Corollary}
\newtheorem{proposition}[theorem]{Proposition}
\newtheorem{claim}[theorem]{Claim}
\newtheorem{question}[theorem]{Question}

\theoremstyle{definition}     
\newtheorem{definition}[theorem]{Definition}
\newtheorem{example}[theorem]{Example}
\newtheorem{remark}[theorem]{Remark}

\mathchardef\mhyphen="2D    

\newcommand{\cE}{\mathcal{E}}
\newcommand{\ZZ}{\mathbb{Z}}
\newcommand{\RR}{\mathbb{R}}

\newcommand{\F}{\mathbb{F}}

\newcommand{\fatness}{\operatorname{span}}   

\newcommand{\VB}{{\operatorname{\it VB}}}
\newcommand{\ep}{\varepsilon}

\usepackage{lipsum}

\makeatletter
\renewcommand{\@biblabel}[1]{[#1]\hfill}
\makeatother

\newcounter{commentcounter}

\newcommand{\tikzmark}[1]{\tikz[overlay,remember picture] \node (#1) {};}
\newcommand{\DrawBox}[1][]{%
    \tikz[overlay,remember picture]{
    \draw[blue,#1]
      ($(left)+(-0.2em,0.9em)$) rectangle
      ($(right)+(0.2em,-0.3em)$);}
}

\input xy
\xyoption{all}
\UseComputerModernTips

\begin{document}

\title[Alexander invariants of periodic virtual knots]{Alexander invariants of periodic virtual knots}

\author[H. Boden]{Hans U. Boden}
\address{Mathematics \& Statistics, McMaster University, Hamilton, Ontario}
\email{boden@mcmaster.ca}

\author[A. Nicas]{Andrew J. Nicas}
\address{Mathematics \& Statistics, McMaster University, Hamilton, Ontario}
\email{nicas@mcmaster.ca}

\author[L. White]{Lindsay White}
\address{Mathematics \& Statistics, McMaster University, Hamilton, Ontario}
\email{whitela3@mcmaster.ca}

\thanks{
}


\date{\today}

\begin{abstract}
We show that every periodic virtual knot can be realized as the closure of a periodic virtual braid and use this to  study the Alexander invariants of periodic virtual knots.
If $K$ is a $q$-periodic and almost classical knot, we show that its quotient knot $K_*$ is also almost classical, and in the case $q=p^r$ is a prime power, we establish an
analogue of Murasugi's congruence relating the Alexander polynomials of $K$ and $K_*$ 
over the integers modulo $p$.
This result is applied to the problem of determining the possible periods of a virtual knot $K$.
One consequence is that
if $K$ is an almost classical knot with a nontrivial Alexander polynomial, then it is $p$-periodic for only finitely many primes $p$.
Combined with parity and Manturov projection, our methods provide conditions that a general virtual knot must satisfy in order to be $q$-periodic. 
\end{abstract} 

\maketitle


\section{Introduction}

An oriented knot or link $K$ in $S^3$ is called \emph{periodic} of  {\it period} $q > 1$ if there is an orientation preserving diffeomorphism $\varphi \colon (S^3, K) \to (S^3,K)$ of finite order $q$  whose fixed point set, $C$, is non-empty and disjoint from $K$.
The solution of the Smith Conjecture implies that $C$ is an unknotted circle and furthermore that $\varphi$ is 
conjugate via a self-diffeomorphism of $S^3$ to a rotation about an axis in $S^3$.
The quotient knot or link, denoted $K_*$, is the image of $K$ under the orbit map
$S^3 \to S^3/\langle \varphi \rangle ~\cong~ S^3$.

In \cite{Flapan-1985}, Flapan proved that a nontrivial classical knot admits only finitely many periods,
and her result was extended to links by Hillman in  \cite{Hillman-1984}.
In \cite{Edmonds-1984}, Edmonds used minimal surface theory to establish a strong upper bound on the period of a given knot $K$ in terms of its Seifert genus.

There are many known conditions on the invariants of classical knot $K$ for it to be periodic.
These include Murasugi's conditions on its Alexander polynomial $\Delta_K(t)$ \cite{Murasugi-1971},
conditions on its Jones polynomial $V_K(t)$ due to Murasugi \cite{Murasugi-1988} and Yokota \cite{Yokota-1991},
as well as Traczyk's conditions on its Kauffman bracket \cite{Traczyk-1990}. 

We recall Murasugi's theorem, \cite{Murasugi-1971}, for classical knots as it motivates our main results.
Given two Laurent polynomials $f(t), \, g(t) \in \ZZ[t^{\pm 1}]$, we write $f(t) \doteq g(t)$ if $f(t) = \pm t^k g(t)$ for some $k \in \ZZ.$

\begin{theorem}[Murasugi, \cite{Murasugi-1971}] \label{classicalmur}
Let $p$ be a prime and $q=p^r$ a prime power. If $K$ is a $q$-periodic knot with quotient knot $K_*$
and linking number $k$ with the rotation axis then
\begin{enumerate}
\item $~\Delta_{K_*}(t)$ divides $\Delta_{K}(t)$ in $\ZZ[t^{\pm 1}],$ and  
\item $~\Delta_K(t) \doteq (\Delta_{K_*}(t))^{q}(1+t+t^2+ \cdots +t^{k-1})^{q-1} \mod p$.
\end{enumerate}
\end{theorem}

Various extensions and alternative proofs of Murasugi's theorem for classical knots and links have been considered in 
\cite{MR643289,  
MR706531,  
MR1133872, 
MR1115739, 
MR2199457}. 

One way to construct examples of periodic knots is to realize them as the closure of a proper power of a braid.
Indeed, if the quotient knot $K_*$ can be written as the closure of a braid $\beta$ of index $k$, and if $q$ is a positive integer  
relatively prime to $k$, then we take $K$ to be the closure of $\beta^q$. It follows that $K$ is a periodic knot with period $q$
and $k$ is equal to the linking number of $K$ with the rotation axis.
Relaxing the assumption that $k$ and $q$ are relatively prime,  the closure of $\beta^q$ will be a link $L$ with $n=\gcd(k,q)$ components \cite[\S 8.2]{Livingston}. 

We say that the periodic knot or link $L$ {\it admits a periodic braid representative} if $L$ is isotopic to the closure of $\beta^q$ for some braid $\beta$.
In \cite{Lee-Park-1997}, Lee and Park establish necessary conditions on a link $L$ for it to admit a periodic braid representative and  
show that not all periodic knots and links admit periodic braid representatives.

In this paper we study periodic {\it virtual knots} (virtual knots are discussed in Section~\ref{virtualknots}),  defined as follows.
\begin{definition}
A virtual knot $K$ is called {\it periodic with period $q$} if it admits a virtual knot diagram which misses the origin and is invariant under a rotation in the plane by an angle of $2 \pi / q$ about the origin. 
\end{definition}

One useful result that we establish is the following.
\begin{theorem} \label{braid-rep}
Any periodic virtual knot or link $L$ admits a periodic virtual braid representative.
\end{theorem}
Although a periodic classical knot or link cannot always be represented by a periodic classical braid,
Theorem 1.2 guarantees the existence of a periodic {\it virtual} braid representative;
see Examples \ref{ex-pvb} and  \ref{two_bridge_example}.
This provides an alternative approach to establishing conditions that classical knot and link invariants must satisfy in the periodic case
and a means of applying virtual knot theory to classical knots.

In his study of periodic virtual knots,  \cite{Lee-2012} S.\,Y.~Lee  posed the following interesting questions:

\begin{question} \label{question-1}
Does a non-trivial virtual knot admit only finitely many periods?
\end{question}

\begin{question}\label{question-2}
Given a classical knot $K$, can it admit a $q$-periodic virtual knot diagram without admitting any $q$-periodic classical knot diagrams?
\end{question}
 
So far, these basic questions have not been resolved.
There is an assortment of known constraints on certain invariants of a virtual knot or link for it to be periodic,
including conditions on the arrow and index polynomials \cite{Im-Lee-2012},
the Miyazawa polynomial \cite{Kim-Lee-Seo-2009},
the {\it VA}-polynomial \cite{Kim-Lee-Seo-2013},
the writhe and odd writhe polynomials \cite{Bae-Lee},
and the virtual Alexander polynomial \cite{Kim-Lee-Seo-2014}.

For various reasons, Murasugi's theorem
has not been extended to the virtual category.
One obstacle is that the Alexander polynomial does not generalize in an entirely straightforward manner to virtual knots and links,
see \cite{Sawollek} for a discussion of the difficulties involved. 

The main goal of this paper is to establish a generalization of  Murasugi's theorem for virtual knots.  
Although the Alexander polynomial does not give a well-behaved invariant for all virtual knots and links, it
does extend nicely to the subcategory of ``almost classical'' knots, defined below.


\begin{definition}[Almost Classical Knots and  Links] 
$  $\newline       
\vspace{-15pt}   
\begin{enumerate}
\item[(i)] A virtual knot diagram is {\it Alexander numberable} if there exists an integer-valued function $\lambda$ on the set of short arcs satisfying the relations in Figure~\ref{Alexander-Numbering-Def-2}.  
\item[(ii)] Given a virtual knot or link $K$, we say $K$ is \emph{almost classical} if it admits a virtual knot diagram that is Alexander numberable.
\end{enumerate}
\end{definition}

While almost classical knots are defined in terms of Alexander numberings,  
it is helpful to keep in mind that a virtual knot is almost classical if and only if it admits a representative knot in a thickened surface which is homologically trivial,
see \cite{AC} for a discussion. 
Another useful observation is that Alexander numberability and periodicity are compatible in the sense that if a knot $K$ is almost classical and periodic then it admits a periodic virtual knot diagram which is also Alexander numberable (see Theorem~\ref{ACrep}).

Recall that if $K$ is a $q$-periodic virtual knot, then by Theorem~\ref{braid-rep}, we can write $K = \widehat{{\beta}^{q}}$ for some $k$-strand virtual braid $\beta$. 
\begin{theorem}\label{virtmur-0}
Let $K = \widehat{{\beta}^{q}}$ be a $q$-periodic almost classical knot diagram with $q=p^r$ a prime power. Then $K_* = \widehat{\beta}$ and 
\begin{enumerate}
\item $~\Delta_{K_*}(t)$ divides $\Delta_{K}(t)$ in $\ZZ[t^{\pm 1}],$ and  
\item $~\Delta_{K}(t) ~\doteq~ \left(\Delta_{K_*}(t)\right)^{q} \left(f(t) \right)^{q - 1} \mod p,$ 
where $f(t) = {\textstyle \sum_{i=1}^{k}} t^{\lambda_i}$ and $\lambda_i$ is the Alexander number on the $i$-th strand of $\beta$.
\end{enumerate}
\end{theorem}

The main difference between this result and Theorem~\ref{classicalmur} is that the polynomial term $1+t+t^2+ \cdots +t^{k-1}$ in the original statement of Murasugi's theorem has been replaced with the general polynomial $f(t) = \sum_{i=1}^{k}t^{\lambda_i}$.
This factor $f(t)$ can be read from the Alexander numbering on the braid strands once one has realized the periodic virtual knot $K$ as the closure of a periodic virtual braid (Theorem~\ref{braid-rep}). 
In the classical case, if $K$ is the closure of a classical braid $\beta$, then it follows easily that $f(t) = 1+t+t^2+ \cdots +t^{k-1}$. 

This theorem allows us to eliminate certain periods for almost classical knots by testing the Alexander polynomial to see if it can be factored in the desired form after reduction modulo $p$.
For example, in Theorem~\ref{thm:finitely-many-periods} we show that any almost classical knot $K$ with non-trivial $\Delta_K(t)$  is $p$-periodic for at most finitely many primes $p$.
In the case $\Delta_K(t)$ mod $p$ is non-trivial for all primes $p$, Corollary~\ref{cor:finitely-many-periods} provides the stronger conclusion that $K$ can be $q$-periodic for at most finitely many $q$. 
This applies to classical fibered knots, giving a positive answer to Question~\ref{question-1} for such knots.

It also allows us to show that many classical knots do not exhibit additional periods in their virtual knot diagrams.

We extend the techniques described above to eliminate composite periods in many cases.
In Example~\ref{threesix}, we show how to eliminate
 $6= 2 \cdot 3$ as a possible virtual period for the trefoil. 
 More generally, we give criteria that can be used to eliminate periods of the form $2p$ where $p$ is an odd prime, see Propositions~\ref{M_1},~\ref{M_2}.
 
Further applications of Theorem \ref{braid-rep} are given in \S\ref{Alexander_ideals}.
Let  $p$ be a prime and $K$ a $p^r$-periodic virtual knot with quotient knot $K_*$.
Corollary \ref{divisible-four} asserts that
$\left( \Delta^\ell_{K_*}\right)^{p^r} \text{\rm mod } p$  divides 
$\Delta^\ell_K \text{ \rm mod } p$
where $\Delta^\ell_{K}$ is $\ell$-th elementary divisor of the Alexander module of $K$ (and likewise for $\Delta^\ell_{K_*}$).
We show in Proposition \ref{prop:KLS-2014} that
\[
{\widehat H}_K(s,t,q)  = \left[{\widehat H}_{K_*}(s,t,q)\right]^{p^r}  \text{ mod $p$, up to multiplication by a power of $st$,}
\]
where ${\widehat H}_K(s,t,q)$ is the ``normalized virtual Alexander polynomial'' of $K$ introduced in \cite{BDGGHN}.
Under suitable non-vanishing hypotheses on ${\widehat H}_K(s,t,q)$,
we obtain an upper bound on the possible periods of $K$ in terms
of $v(K),$ the virtual crossing number of $K$ (see Proposition \ref{vcrossing_bound}
and Corollary \ref{strong_vcrossing_bound}). This gives an answer to Question \ref{question-1} when ${\widehat H}_K(s,t,q)$ is nontrivial mod $p$.
Note that if $K$ is almost classical, then ${\widehat H}_K(s,t,q) =0$, and so it does not provide any restrictions on the possible periods of $K$.

Although our Theorem~\ref{virtmur-0} applies only to almost classical knots, 
we now explain how to apply parity projection  to obtain constraints on any virtual knot. In the following, let $f$ denote the total Gaussian parity (see Equation \eqref{totalgaussian} on page \pageref{totalgaussian} for its definition).

In general, the parity $f(c_i)$ of a chord in a virtual knot diagram is either even or odd, and it must satisfy Manturov's parity axioms.
The projection $P_f(D)$ of a diagram is obtained by eliminating all the odd chords, and the parity axioms are defined to ensure that if two diagrams $D_1$ and $D_2$ are equivalent through Reidemeister moves, then so are the diagrams $P_f(D_1)$ and $P_f(D_2)$ obtained by projection. It follows that the knot type of $P_f(K)$ is well-defined and independent of the representative diagram for $K$.

Since any virtual knot diagram can have only finitely many chords, for any $D$, there is a positive integer $n$ such that $P_f^{n+1}(D) = P_f^n(D)$ (and therefore, for any $m>n$, the diagram will remain the same, so we have $P_f^{m}(D) = P_f^n(D)$).
Stable Manturov projection is defined as $P_f^\infty(K) = \lim_{n \to \infty} P_f^n$, that is, if $\ell$ is the smallest $n$ such that $P_f^{n+1}(D) = P_f^n(D)$,
then $P_f^\infty(K) =  P_f^{\ell}(D)$.
It is a general fact that the image $P_f^\infty(K)$ is an almost classical diagram for any virtual knot diagram $K$, where $f$ is the total Gaussian parity.

\begin{theorem}
If  $K$ is a $q$-periodic virtual knot diagram, then $\bar{K}= P_f^\infty(K)$ is a $q$-periodic almost classical  diagram.
\end{theorem}

Hence we can apply Theorem~\ref{virtmur-0} to $\bar{K}$  to give conditions that must be satisfied in order for $K$ to be periodic.  For instance, any virtual knot $K$ whose projection $\bar{K}= P_{f}^\infty(K)$ has nontrivial Alexander polynomial is $p$-periodic for only finitely many primes $p$.

Table~\ref{tab:periods} lists the known periods and excluded periods for almost classical knots up to 6 crossings. 
This table was obtained by applying Theorem~\ref{virtmur-0} to their Alexander polynomials, which were computed in \cite{AC} and are listed here in Table~\ref{acks2}.
A table of Gauss diagrams of almost classical knots with up to six crossings can be found in \cite{AC}.
In Table~\ref{tab:periodic-braids} we list the known periodic almost classical knots as closures of periodic virtual braids.

The main results in this paper are derived from the Ph.D.~thesis of Lindsay White, \cite{White},  
written under the supervision of Hans Boden and Andrew Nicas at McMaster University.

\bigskip
\noindent
{\it Acknowledgements.}  H.~Boden and A.~Nicas were supported by grants from the Natural Sciences and Engineering Research Council of Canada.


\section{Preliminaries} \label{sec:Preliminaries}

In this section, we recall some basic notions from knot theory and virtual knot theory.

\subsection{Classical knots and Gauss diagrams}
A \emph{knot} $K$ is a smooth embedding $S^1 \to S^3$ of the circle into 3-space. Two knots are considered \emph{equivalent} if there is an ambient isotopy of $S^3$ taking one to the other.
More generally, a \emph{link} is a smooth embedding $S^1 \cup \cdots \cup S^1  
\to S^3$ of a disjoint union of circles, up to ambient isotopy of $S^3$ taking one to another.

\begin{wrapfigure}{r}{0.42\textwidth}
\centering
\includegraphics[scale=01.00]{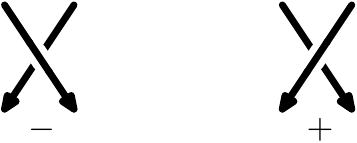} 
\vspace{-1mm}
\caption{Left and right handed crossings.}\label{local-writhe}
\end{wrapfigure}

We will usually work with oriented knots and links, and when necessary we indicate the choice of orientation using an arrow.  



A \emph{knot diagram} is the regular projection of a knot $K$ to $\RR^2$ with only double-points, which are drawn to indicate the overcrossing and undercrossing strands. Two knot diagrams are \emph{equivalent} if one can be transformed into the other through a series of Reidemeister moves and planar isotopies.

The \emph{Gauss code} of a knot diagram is a word that records the crossings and their signs as one traverses the knot. To begin, number the crossings ${\tt 1, 2, \ldots, n}$ arbitrarily and pick a basepoint on the knot. Then traverse the knot and record each crossing as it is encountered along with its sign,  which is positive if the crossing is right-handed and negative  if it is left-handed (see Figure~\ref{local-writhe}). Each crossing will be recorded twice, once as an over-crossing (written {\tt Oi}) and then as an under-crossing (written {\tt Ui}). For example, the trefoil in Figure~\ref{Fig:trefoil} has Gauss code {\tt O1+U2+O3+U1+O2+U3+}.

\begin{wrapfigure}{r}{0.42\textwidth}
\centering
\includegraphics[scale=0.85]{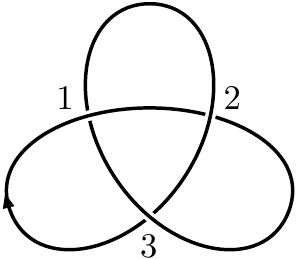}
\vspace{-1mm}
\caption{The trefoil knot $3_1$.}\label{Fig:trefoil}
\end{wrapfigure}

The Gauss code is determined by the oriented knot up to relabeling of the crossings and altering the choice of basepoint. A relabeling of the crossings amounts to permuting the numbers ${\tt 1, 2, \ldots, n}$ within the Gauss code, and altering the choice of basepoint amounts to a cyclic permutation of the Gauss code. 

A \emph{Gauss diagram} is a trivalent graph consisting of a base circle, which represents the underlying knot, along with directed chords $c_1,\ldots, c_n$, one for each crossing. The $i$-th chord $c_i$ points from the over-crossing arc to the under-crossing arc, and its writhe, $\ep_i = \pm 1$, is given by the sign of the $i$-th crossing. 

The Reidemeister moves can be translated into moves between Gauss diagrams, 
and in this way one can regard a classical knot as an equivalence class of Gauss diagrams. Every classical knot diagram is uniquely determined by its associated Gauss diagram, but not all Gauss diagrams correspond to classical knots. 


 
\subsection{Virtual knots}\label{virtualknots}

Virtual knot theory was invented by Kauffman \cite{KVKT}, and virtual knots represent the complete set of all Gauss diagrams modulo Reidemeister moves. As with classical knots, virtual knots can be represented in terms of virtual knot diagrams, which are described next.

\begin{wrapfigure}{r}{0.60\textwidth}
\centering
\includegraphics[scale=0.80]{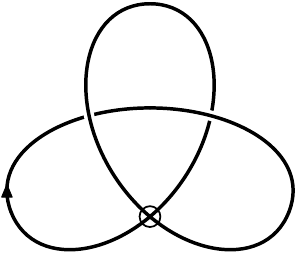} \qquad \includegraphics[scale=0.70]{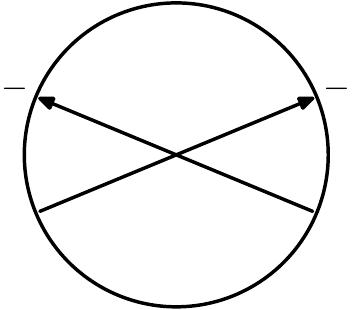}
\caption{The virtual trefoil and its Gauss diagram.}
\label{2-1}
\end{wrapfigure}

A \emph{virtual knot diagram} is an immersion of a circle in the plane with only double points, such that each double point is either classical (indicated by over- and under-crossings) or virtual (indicated by a circle). \emph{Virtual link diagrams} are defined similarly. Such a diagram is \emph{oriented} if every component has an orientation.

Two oriented virtual link diagrams are \emph{virtually isotopic} (or \emph{equivalent}) if they can be related by planar isotopies and a series of \emph{Reidemeister moves} 
and the \emph{detour move} in Figure~\ref{fig:detour}.

Virtual isotopy defines an equivalence relation on virtual link diagrams, and a \emph{virtual knot or link} is defined to be an equivalence class of virtual knot or link diagrams under virtual isotopy.


\begin{wrapfigure}{r}{0.54\textwidth} 
\centering
\includegraphics[width=0.50\textwidth]{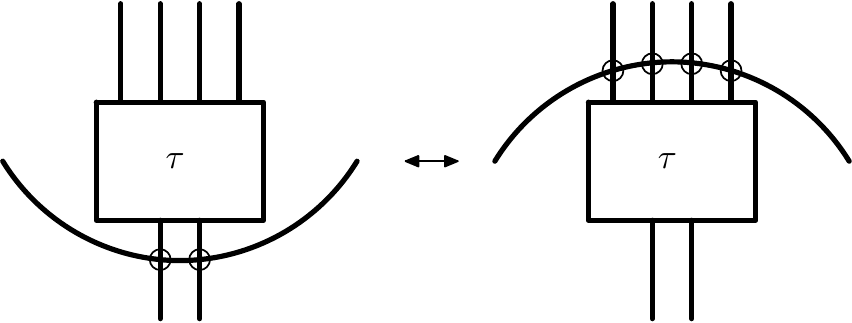}
\caption{Detour  move. \newline }\label{fig:detour}
\end{wrapfigure}

In this paper, we refer to various knots (up to virtual equivalence) by labelling them with a decimal number (for example, $6.90099$), which comes from the enumeration by \cite{Green}. The number before the decimal refers to the real crossing number of the knot (that is, the minimum number of real crossings in an equivalent diagram of the knot).

Just as with classical knot diagrams, every virtual knot diagram determines a Gauss code and Gauss diagram, either of which uniquely determines the virtual knot diagram. Indeed, an alternative but equivalent way to define virtual knots is as equivalence classes of  Gauss diagrams by Reidemeister moves, as proved by Goussarov, Polyak, and Viro in \cite{GPV}.

\begin{figure}[ht]
\centering
\includegraphics[scale=0.80]{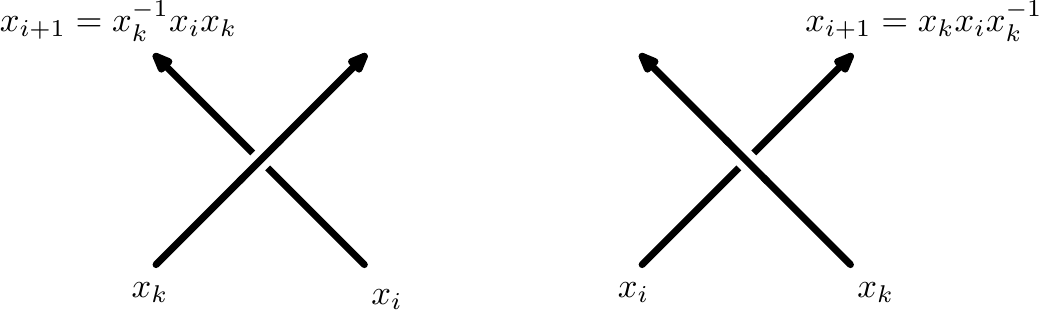}
\caption{The relations in $G_K$ from the $i$-th crossing of $K$.}
\label{Knot-Group-Wirtinger}
\end{figure}

\subsection{The knot group and Alexander invariants}
\label{knot_and_alex}

\begin{definition}
\label{knot_group}
Suppose $K$ is an oriented virtual knot with $n$ classical crossings, and choose a basepoint on $K$. Starting at the basepoint, we label the arcs $x_1, x_2, \ldots, x_n$ so that at each under-crossing, $x_i$ is the incoming arc and $x_{i+1}$ is the outgoing arc. (If $i=n$, then we set $i+1 := 1$; that is, we take $i$ modulo $n$.) We use a consistent labeling of the crossings so that the $i$-th crossing is as shown in 
Figure~\ref{Knot-Group-Wirtinger}.
For $i=1,\ldots, ,n$ let $\ep_i=\pm 1$ be according to the sign of the $i$-th crossing.
Then the \emph{knot group} of $K$ is the finitely presented group given by

\begin{equation} \label{gp-pres}
G_K = \langle x_1, \ldots, x_n \mid x_{i+1} = x_k^{-\ep_i} x_i x_k^{\ep_i}, i=1, \ldots, n \rangle.
\end{equation}
Note that virtual crossings are ignored in this construction.
\end{definition}

\begin{wrapfigure}{r}{0.5\textwidth}
\centering
\includegraphics[scale=0.90]{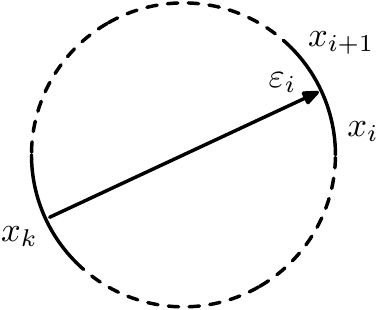} 
\caption{The chord $c_i$ has Wirtinger relation $r_i=  x_k^{-\ep_i} x_i x_{k}^{\ep_i} x_{i+1}^{-1}$.}
\label{GD-crossing}
\end{wrapfigure}

As explained by S.G. Kim \cite{Kim}, the Wirtinger presentation of the knot group $G_K$ can also be easily read from a Gauss diagram $D$ for $K$ as follows. 
Pick a basepoint on $D$ and number the chords $c_1, \ldots, c_n$ sequentially in the order in which one encounters their arrowheads when going around $D$ counterclockwise. The long arcs of $D$ are the subarcs of $D$ from one arrowhead to the next, and we label them $x_1, \ldots, x_n$ sequentially so that $x_i$ and $x_{i+1}$ are separated by the arrowhead of $c_i$ (here $i$ is taken modulo $n$). Then the knot group admits the presentation
\begin{equation} \label{gp-pres-2}
G_K = \langle x_1,\ldots, x_n \mid r_1, \ldots, r_n \rangle,
\end{equation}  
where the relation $r_i$ arises from the chord $c_i$ as follows. If the arrowtail of $c_i$ lies on the arc labeled by $x_k$, then $r_i$ is the relation $ x_k^{-\ep_i} x_i x_{k}^{\ep_i} x_{i+1}^{-1},$ where $\ep_i$ is the writhe of $c_i$ and $i$ is taken modulo $n$
(cf.~Figure~\ref{GD-crossing}).

The knot group $G_K$ is an invariant of virtual isotopy. (In fact, it is an invariant of the underlying welded knot type of $K$). In case $K$ is classical, we have  $G_K \cong \pi_1(S^3 - N(K)),$ the fundamental group of the complement of $K$.


We recall the construction of the Alexander module associated to the knot group $G_K$ of a virtual knot $K$. Let $G_K' =[G_K,G_K]$ and $G_K'' = [G_K',G_K']$ be the first and second commutator subgroups. The Alexander module is then the quotient $G_K' / G_K''.$ It is a finitely generated module over $\ZZ[t^{\pm 1}]$, the ring of Laurent polynomials, and it is determined by the $n \times n$ Jacobian matrix obtained by Fox differentiating the relations $r_i$ (appearing in the presentation \eqref{gp-pres-2} of $G_K$) with respect to the generators $x_j$ and applying the abelianization map $x_\ell \mapsto t$ for $\ell = 1,\ldots, n$.
While the matrix $A$ will depend on the choice of presentation for $G_K$,  the associated sequence of {\it elementary ideals}
\begin{equation}\label{chain}
 (0)= \cE_0 \subset \cE_1 \subset \cdots \subset \cE_n  = \ZZ[t^{\pm 1}]
 \end{equation}
does not. Here, the $k$-th elementary ideal $\cE_k$ is defined as the ideal of $\ZZ[t^{\pm 1}]$ generated by all $(n-k) \times (n-k)$ minors of $A$. 
The ideals in the  chain \eqref{chain}, also known as the {\it Alexander ideals} of $K$, are knot invariants.

We now describe the standard method for deriving a presentation matrix for the Alexander module from the Wirtinger presentation \eqref{gp-pres} of the virtual knot.  
As before, we assume that $K$ is a virtual knot with $n$ real crossings $c_1, \ldots, c_n$ and long arcs $x_1, \ldots, x_n$ such that $x_i$ starts at the under-crossing of $c_{i-1}$ and ends at the under-crossing of $c_i$. 
 
 The Fox derivatives of the relations $r_i$ are given by
$$
\frac{\partial r_i}{\partial x_{j}} = \left \{
\begin{array}{ll}
x_k^{-\ep_i} & \text{if $j = i,$} \\
-x_k^{-\ep_i} x_i x_k^{\ep_i} x_{i+1}^{-1} & \text{if $j = i+1,$} \\
1 - x_k x_i x_k^{-1} & \text{if $j =k$ and $\ep_i = -1,$} \\
-x_k^{-1}+ x_k^{-1} x_i & \text{if $j = k$ and $\ep_i = 1,$} \\
0 & \text{otherwise.}
\end{array}
\right .
$$

\begin{definition} \label{method-1} 
The \emph{Jacobian matrix} $A=A(D)$ is the $n \times n$ matrix with
$$A_{i,j} = \left.\frac{\partial r_i}{\partial x_j} \right|_{x_1,\ldots, x_n=t}$$ given by Fox differentiation and applying the abelianization map $G_K \to \ZZ$ sending $x_\ell \mapsto t$ for $\ell =1,\ldots, n.$ 
More concretely, if the $i$-th crossing is as in Figure~\ref{GD-crossing}, then

\begin{equation} \label{eq-method1}
A_{i,j} = \begin{cases} 
t^{-\ep_i} & \text{if $j=i$,} \\
-1 & \text{if $j=i+1$,} \\
1-t^{-\ep_i}& \text{if $j=k$,} \\
0 & \text{otherwise}.
\end{cases}
\end{equation}
\end{definition}
For both classical and virtual knots, the zeroth elementary ideal $\cE_0$ is always trivial. This follows from the observation that $\cE_0 = (\det(A)) = 0$, since
the sum of the columns of $A$ is zero by the fundamental identity of Fox derivatives. 

For a classical knot $K$, the first elementary ideal $\cE_1$ is always principal and generated by the \emph{Alexander polynomial} $\Delta_K$, which is well-defined up to multiplication by $\pm t^k$ and satisfies $\cE_1 = (\Delta_K(t))$. It is obtained by taking the determinant of the Alexander matrix, which is the $(n-1) \times (n-1)$ matrix obtained by removing a row and column from $A$.

The first elementary ideal $\cE_1$ of a virtual knot $K$ is not always principal.
Nevertheless, one can define  the  Alexander polynomial $\Delta_K(t)$  to be the generator of the smallest principal ideal containing $\cE_1$,
obtained by taking the greatest common divisor of all the $(n-1) \times (n-1)$ minors of $A$.


%



\subsection{Almost classical knots and parity}

\begin{definition}[Almost Classical Knots and  Links] \label{AN-Def}
$  $\newline       
\vspace{-15pt}   
\begin{enumerate}
\item[(i)] A virtual knot diagram is  \emph{almost classical} if there exists an integer-valued function $\lambda$ on the set of short arcs satisfying the relations in Figure~\ref{Alexander-Numbering-Def-2} at each crossing. (This definition extends naturally to virtual links.) 
Such diagrams are also called \emph{Alexander numberable}.
\item[(ii)] Given a virtual knot or link $K$, we say $K$ is \emph{almost classical} if it is represented by an almost classical diagram.
\end{enumerate}
\end{definition}

\begin{figure}[ht]
\centering
\includegraphics[scale=0.80]{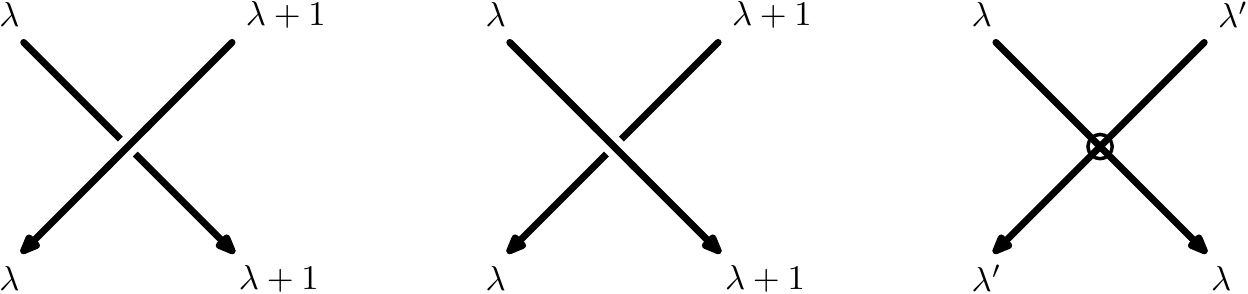}
\caption{The Alexander numbering conditions.}
\label{Alexander-Numbering-Def-2}
\end{figure}

\begin{definition} \label{chord-index}
If $D$ is a Gauss diagram with chords $c_1, \ldots, c_n$, we define
the \emph{index} of $c_i$ by counting the chords $c_j$ that intersect $c_i$ with sign and keeping track of their direction.  Specifically, position the diagram so that $c_i$ points vertically, and set 
$$I(c_i)=r_{+} - r_{-} + \ell_{-} -\ell_{+},$$ where
\begin{eqnarray*}
&r_\pm = \#\{ c_j \mid \text{$c_j$ intersects $c_i$ with $\ep_j = \pm 1$ and  arrowhead to the right}\},& \\
&\ell_\pm = \#\{ c_j \mid \text{$c_j$ intersects $c_i$ with $\ep_j = \pm 1$ and arrowhead to the left} \}. \\
\end{eqnarray*}
\end{definition}

For example, the Gauss diagram in Figure~\ref{2-1} has one chord with index $1$ and another with index $-1$.


One can verify that a Gauss diagram $D$ represents an almost classical virtual knot diagram if and only if every chord $c_i$ of $D$ has index $I(c_i) = 0$. 

Almost classicality in virtual knot theory is closely related to Gaussian parity, and here we give a brief account. Parity is an important topic in virtual knot theory, and here we will only scratch the surface. For more information, we refer the reader to Manturov's original article \cite{Manturov-13}, his book \cite{State}, and the monograph \cite{IMN-11}.

Given a virtual knot diagram, a parity is a function that assigns to each classical crossing a value in $\{0,1\}$ (or ``even" and ``odd") such that the following axioms hold: 

\begin{itemize}
\item[1.] In a Reidemeister one move, the parity of the crossing is even.  
\item[2.] In a Reidemeister two move, the parities of the two crossings are either both even or both odd. 
\item[3.] In a Reidemeister three move, the parities of the three crossings are unchanged. Further, the three crossings can be all even, all odd, or one even and two odd. (In other words, we exclude the case that one crossing is odd and two are even.) 
\end{itemize}

Note that this is the definition of ``parity in the weak sense," cf. Manturov \cite{Manturov-13} and Nikonov \cite{Nikonov-2016}.

For example, taking
\begin{equation}\label{totalgaussian}
f(c_i)= \begin{cases} 0 & \text{if $I(c_i) = 0$}\\
1 & \text{if $I(c_i) \neq 0$}
\end{cases}
\end{equation}
gives a parity that we call the {\it total Gaussian parity}. 

One can easily verify that Gaussian parity satisfies the parity axioms, and we leave the details to the reader.

Notice that a diagram $D$ has only even chords  if and only if $I(c_i) = 0$ for all chords (by definition of $f$). This is equivalent to the condition that $D$ admit an Alexander numbering.

There is a map  $$P_f \colon \{ \text{Gauss diagrams} \} \longrightarrow \{\text{Gauss diagrams}\}$$
called \emph{Manturov projection} which is defined by removing the odd chords of $D$.
Thus, if all chords of $D$ are even 
then $P_f(D)=D.$ Otherwise, if $P_f(D) \neq D$ then $D$ contains one or more odd chords and does not admit an Alexander numbering. Its projection $P_f(D)$ will then be a diagram with fewer chords, but because removal of chords may alter the parity of the surviving chords, the new diagram $P_f(D)$ may contain odd chords. 

Repeated application of $P_f$ will eventually give a diagram without odd chords. In fact, 
for some some $n \geq 0$ the projection $P_f^{n}$ stabilizes in that $P_f^{n+1}(D) = P_f^{n}(D)$. The resulting diagram $\bar{D} = P_f^{n}(D)$ is has only even chords and hence is almost classical. We define $P_f^{\infty} = {\displaystyle \lim_{n \to \infty}} P_f^n$ and call it
the stable Manturov projection.

In summary, we have shown that for any virtual knot diagram $D$, its image $P_f^{\infty}(D)$ under stable Manturov projection admits an Alexander numbering and therefore is an almost classical virtual knot.

Although Manturov projection $P_f$ is defined at the level of diagrams,
the next proposition implies that it is well-defined on virtual knots. The proof is an immediate consequence of the parity axioms, and for details we refer to either \cite{Manturov-13} or \cite{Nikonov-2016}.

\begin{proposition}\label{Prop-Projection}
If two virtual knot diagrams $K$ and $K'$ are virtually isotopic, then so are their images $P_f(K)$ and $P_f(K')$ under Manturov projection.
\end{proposition}

Next, we recall that if $K$ is an almost classical knot, then its first elementary ideal $\cE_1$ is principal. This result was proved by Nakamura, Nakanishi, Satoh, and Tomiyama in \cite[Theorem 1.2]{Nakamura-et-al} by using an Alexander numbering to determine a linear combination of the rows of the Jacobian matrix that sum to zero. Because it is central to our later results, we will go through the argument carefully.

\begin{proposition}\label{linearcombo}
If $K$ is an almost classical knot or link, then its first elementary ideal $\cE_1$ is principal.
\end{proposition}

\begin{proof} 
Let $A= \left( \left. { \frac{\partial r_i}{\partial x_j}} \right|_{x_1,\ldots, x_n=t}\right)$ be the Jacobian matrix of Definition~\ref{method-1} and $\cE_1$ the ideal generated by $(n-1) \times (n-1)$ submatrices of $A$.
Then equation  \eqref{eq-method1} implies that    
\begin{equation}\label{Alexander-Matrix}
A_{i,j} = \left. \frac{\partial r_i}{\partial x_{j}} \right|_{x_1,\ldots, x_n=t}   = \left \{
\begin{array}{ll}
t^{-\ep_i} & \text{if $j = i,$} \\
-1 & \text{if $j = i+1,$} \\
1 - t^{-\ep_i} & \text{if $j = k,$} \\
0 & \text{otherwise.} 
\end{array}
\right.
\end{equation}

Let $A_{i,*}$ denote the $i$-th row of $A$, 
and set
$$ \vartheta_i = 
\begin{cases}
t^{\lambda_i} & \text{if $\ep_i = -1,$} \\
t^{\lambda_i+1} & \text{if  $\ep_i = +1.$}
\end{cases}
$$
where $\lambda_i$ and $\lambda_{i+1}$ are the two Alexander numbers showing up at the crossing $c_i$ (see Figure~\ref{AN-at-i}).
Notice that $\vartheta_i$ is a unit in $\ZZ[t^{\pm 1}]$ for $i=1,\ldots, n.$

\begin{claim} \label{LC-method-1}
We have ${\displaystyle \sum_{i=1}^{n}} \vartheta_i A_{i,*} = 0$.
\end{claim}

To prove the claim, we compute  
\begin{eqnarray*}
\vartheta_i{A_{i,*}} &=& 
\begin{cases}
(0, \ldots, 0, t^{\lambda_i},- t^{\lambda_i+1},0,\ldots, 0,t^{\lambda_i+1}- t^{\lambda_i},0,\ldots,0)  & \text{if $\ep_i = +1$,}\\
(0, \ldots, 0, t^{\lambda_i+1}, - t^{\lambda_i},0,\ldots, 0, t^{\lambda_i}-t^{\lambda_i+1},0,\ldots, 0)  & \text{if $\ep_i = -1$.} 
\end{cases}
\end{eqnarray*}

Recall that the row $A_{i,*}$ corresponds to the crossing $c_i$ of $K$ with incoming underarc $x_i$, outgoing underarc $x_{i+1}$, and overar $x_k$. 
In the case $\ep_i = +1$, the $t^{\lambda_i}$ term in row $i$ corresponds to the incoming underarc $x_i$; $-t^{\lambda_i + 1}$ corresponds to the outgoing underarc $x_{i+1}$, and $t^{\lambda_i + 1}-t^{\lambda_i}$ corresponds to the overarc $x_k$. Similar considerations apply in the case $\ep_i = -1$ case.

\begin{figure}[h]
\begin{center}
\includegraphics[scale=0.80]{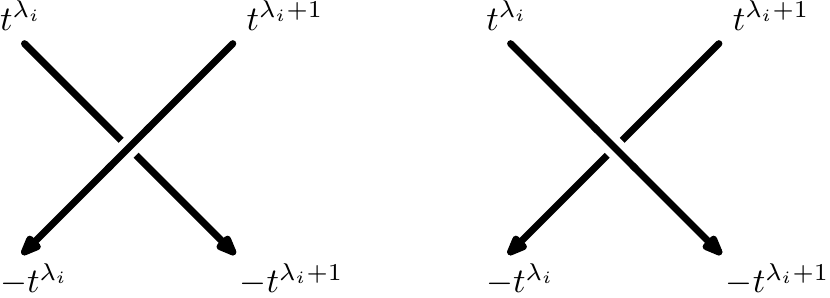}
\end{center}
\caption{Arc labels at the $i$-th crossing in terms of Alexander numbers.} \label{AN-at-i}
\end{figure}

Notice that the incoming arcs have labels with positive signs, and the outgoing arcs have labels with negative signs.
In the linear combination $\sum_{i=1}^{n} \vartheta_i A_{i,*}$, the $j$-th entry is given as the sum of all terms in the $j$-th column (each multiplied by a $\vartheta_i$), namely all the terms as above contributed by the arc $x_j$. This includes terms for which $x_j$ is the outgoing underarc, incoming underarc, or overarc, and those terms are given by multiplying one of $-1, t^{- \ep_i}, 1-t^{- \ep_i}$ as in \eqref{eq-method1} with the coefficient $\vartheta_i$ as above. This is the same as summing up all the labels as in Figure~\ref{AN-at-i} as you move across the arc $x_j$. Of course, the $\lambda_i$ term corresponds to the Alexander numbering of the arcs.

\begin{figure}[h]
\begin{center}
\includegraphics[scale=0.85]{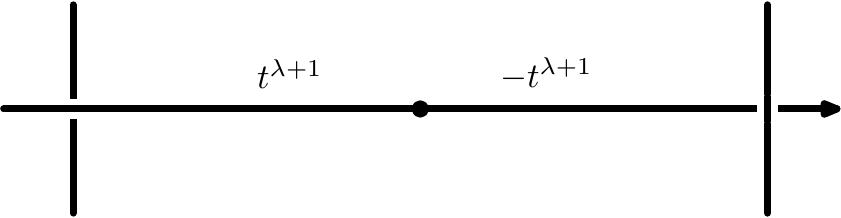}
\end{center}
\caption{Two consecutive Alexander numbers.} \label{cons-AN}
\end{figure}

Now, on a given arc $x_j$, it turns out that consecutive labels cancel and this shows why the sum $\sum_{i=1}^{n} \vartheta_i A_{i,*}=0$ is zero.
Recall that $x_j$ goes from the $(j-1)$-st under-crossing to the $j$-th under-crossing, so it is an incoming underarc and outgoing underarc exactly once. However, it can be an overarc for multiple crossings. Since the $\vartheta_i$'s correspond to the Alexander numbers, if we have a short arc contributing two terms to the sum, each term must have the same power of $t$ (see Figure~\ref{cons-AN}).

The reason is that Alexander numberings are assigned to the short arcs but the labels (as in Figure~\ref{AN-at-i}) are assigned to half of a short arc. They are still related to the Alexander numbering on the arc, and any two terms on the same short arc must have the same Alexander number. On the other hand, the two terms are of opposite sign (since one will be ingoing and one outgoing), and that is why the sum is zero. This completes the proof of the claim. 

%
%
%
%

The claim provides a linear dependence among the rows of $A$, and the fundamental identity of Fox derivatives shows that the sum of the columns of $A$ is zero. This fact, together with the observation that each $\vartheta_i$ is a unit in $\ZZ[t^{\pm 1}],$ shows that any $(n-1) \times (n-1)$ minor is a generator of $\cE_1$. In particular, $\cE_1$ is principal, and this completes the proof.
\end{proof}

As a consequence of the proof, it follows that for an almost classical knot $K$ with Jacobian matrix $A$ constructed as in Definition~\ref{method-1}, its Alexander polynomial $\Delta_K(t)$ is given by taking the determinant of the $(n-1) \times (n-1)$ matrix obtained by removing any row and any column from $A$. 


\subsection{Virtual braids}
In this section, we introduce virtual braids and recall the virtual analogue of Alexander's theorem (stated later in this section), which shows that every virtual knot or link can be realized as the closure of a virtual braid. 

The \emph{virtual braid group on $k$ strands}, denoted $\VB_{k}$,
is the group generated by $\sigma_{1}, \ldots, \sigma_{k-1}, \tau_1,\ldots, \tau_{k-1}$ subject to the relations in equations \eqref{classical-rel}, \eqref{virtual-rel}, \eqref{mixed-rel}. Here, $\sigma_{i}$ represents a classical crossing  and $\tau_i$ represents a virtual crossing involving the $i$-th and $(i+1)$-st strands as in Figure~\ref{generators}. Virtual braids are drawn from top to bottom, the group operation is given by stacking the diagrams, and the closure of a virtual braid represents a virtual link. 
 
\begin{figure}[ht]
\centering
\includegraphics[scale=0.85]{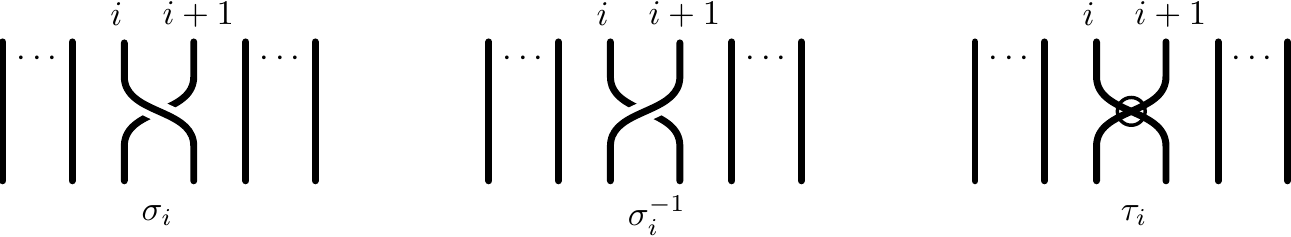}
\vspace{-2mm}
\caption{Generators of $\VB_{k}$.}
\label{generators}
\end{figure}  


Note that the virtual generators $\tau_1, \ldots, \tau_{k-1}$ generate a finite subgroup of $\VB_k$ isomorphic to the symmetric group $S_k$ on $k$ letters; the element $\tau_i$ swaps the $i$-th and $(i+1)$-st strands and corresponds to the transposition $(i, i+1) \in S_n$.

\begin{equation} \label{classical-rel}
\begin{array}{rcl}
\sigma_{i}\sigma_{j}&=&\sigma_{j}\sigma_{i} \hspace{1.5cm} \text{ if $|i-j|>1$,} \\
\sigma_{i}\sigma_{i+1}\sigma_{i}&=&\sigma_{i+1}\sigma_{i}\sigma_{i+1},
\end{array}
\end{equation}
\begin{equation} \label{virtual-rel}
\begin{array}{rcl}
\tau_{i}\tau_{j}&=&\tau_{j}\tau_{i} \hspace{1.5cm} \text{ if $|i-j|>1$,} \\
\tau_{i}\tau_{i+1}\tau_{i} &=& \tau_{i+1}\tau_{i}\tau_{i+1}, \\
\tau_{i}^{2}&=&1,
\end{array}
\end{equation} 
\begin{equation}\label{mixed-rel}
\begin{array}{rcl}
\sigma_{i}\tau_{j}&=&\tau_{j}\sigma_{i} \hspace{1.5cm} \text{ if $|i-j|>1$,} \\
\tau_i \sigma_{i+1}\tau_i & = & \tau_{i+1} \sigma_i \tau_{i+1}.
\end{array}
\end{equation} 

The next result is Alexander's theorem for virtual knots and links, and it was first proved by Kamada in \cite{kamada} via a braiding process. Interestingly, the statement of Alexander's theorem is stronger in the virtual setting because, as we shall see, the virtual braid faithfully reproduces the Gauss code of the virtual knot diagram, cf. \cite[Theorem 2.1]{birman}. We provide a proof, and later in Theorem~\ref{mperiod} we will establish an equivariant version of Alexander's theorem for periodic virtual knots.

\begin{theorem} \label{Thm-Alex}
Every virtual knot diagram can be realized as the closure of a virtual braid.
\end{theorem}

\begin{proof}
Let $K$ be a virtual knot diagram with $n$ real crossings and Gauss code $C_K$. We will show how to construct a virtual braid $\beta$ on $2n$ strands whose closure is a virtual knot with Gauss code identical with $C_K$, and it follows that $\widehat{\beta}$ and $K$ are equivalent (as virtual knot diagrams) up to a sequence of detour moves and planar isotopies.

First, we draw the $n$ real crossings side by side pointing downwards according to their sign (see Figure~\ref{local-writhe}). For example,
for the Gauss code $C_K=\hbox{\tt O1-U2+O3-U1-O2+U3-}$, we draw three crossings as in Figure~\ref{braid-algo}.

\begin{figure} 
\begin{center}
\includegraphics[scale=0.80]{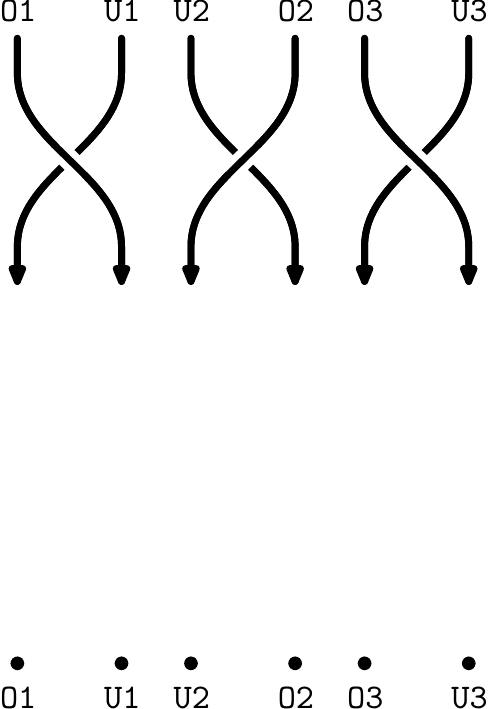}
\qquad \qquad \qquad
\includegraphics[scale=0.80]{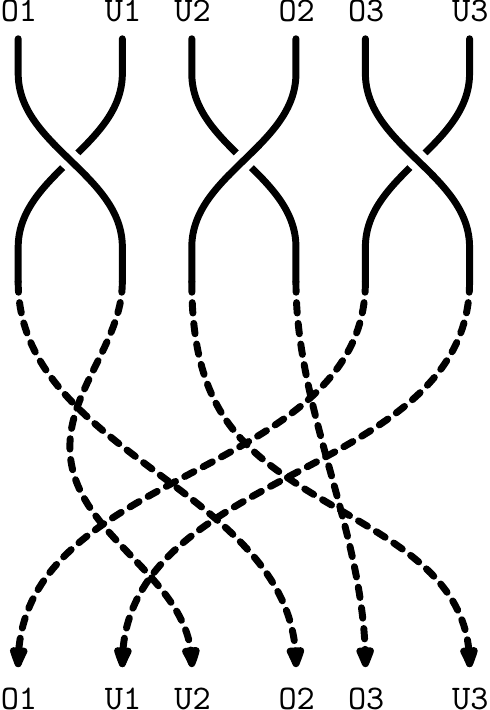}
\end{center}
\caption{The braiding algorithm for $C_K=\hbox{\tt O1-U2+O3-U1-O2+U3-}$.}\label{braid-algo}
\end{figure}

We label the $2n$ arcs across the top with {\tt O1, U1, U2, O2, ...} appropriately (that is, {\tt Oi,Ui} for a negative crossing, and  {\tt Ui,Oi} for a positive crossing), and we draw $2n$ points directly underneath, which we label {\tt O1, U1, U2, O2, ...} in exactly the same order as on top. (This is illustrated on the left of Figure~\ref{braid-algo}.)

The Gauss word $C_K$ tells us how to connect the outgoing arcs from each of the $n$ crossings to the corresponding points at the bottom. For instance, the first part of the Gauss word {\tt O1-U2+O3-U1-O2+U3-} tells us to connect the outgoing overarc of the first crossing ({\tt O1}) to the point labelled {\tt U2} below, and next it tells us to connect the outgoing underarc of the second crossing ({\tt U2}) to the point labelled {\tt O3}. Continuing in this way, we connect all the arcs to points, with the last crossing in the Gauss word connected back to the first entry. In the example, it tells us to connect the outgoing underarc of the third crossing ({\tt U3}) to the point labelled {\tt O1}.
The outcome is a virtual braid as depicted on the right of Figure~\ref{braid-algo}. 

The connecting arcs are drawn monotonically decreasing, and every new crossing that is created is drawn as a virtual crossing. Basically, the connecting lines, which appear as dashed lines in Figure~\ref{braid-algo}, determine an element in $S_{2n}$, the symmetric group. Because the virtual generators of $\VB_{2n}$ generate a subgroup  isomorphic to $S_{2n}$, we can always write this element of $S_{2n}$ as a word in the $\tau_1,\ldots, \tau_{2n-1}$. In the example above, we get the word $\tau_1 \tau_4 \tau_3 \tau_5 \tau_4 \tau_5 \tau_2 \tau_1 \tau_3\tau_2$.

The resulting diagram will be a virtual braid $\beta$ 
on $2n$ strands whose closure is equivalent to the given knot $K$. In fact, as one can easily verify, the Gauss code of $\widehat{\beta}$ is equal to $C_K$.
\end{proof}

We conclude this section by defining almost classical braids and introducing an invariant for them.
\begin{definition} A braid is called
\emph{almost classical} if it admits an Alexander numbering (that is, if one can number the arcs of $\beta$ such that, at each crossing the conditions of Figure~\ref{Alexander-Numbering-Def-2} are satisfied) and so that the numbers along the bottom of $\beta$ coincide with the numbers at the top.
\end{definition}
 
Note that a braid is almost classical if and only if its closure $\widehat{\beta}$ is an almost classical diagram. For example, if $\beta \in B_k$ is a classical braid, then taking $\lambda_i=i$ on the $i$-th strand at the top extends to an Alexander numbering of $\beta$ such that the $i$-th strand on the bottom also has the number $i$, thus any classical braid $\beta$ is almost classical. Note that taking $\lambda_i=i$ gives a valid Alexander numbering since, for a classical braid, there are no virtual crossings, and each strand's Alexander number must increase by one as we move from left to right across the strands, because of the conditions of Figure~\ref{Alexander-Numbering-Def-2}. So, if we start on the left with $\lambda_1 = 1$, then the second strand will need to have $\lambda_2 = 2$, and so on. 

If $\beta \in \VB_k$ is an almost classical braid, then consider the polynomial $f(t) = \sum_{i=1}^k t^{\lambda_i}$, where $\lambda_i$ refers to the Alexander number on the $i$-th strand at a horizontal cross-section of $\beta$. Notice that this polynomial is independent of where along the braid the cross-section is taken.  
When passing a classical crossing, the Alexander numbers on the two strands swap positions, but $f(t)$ remains unchanged. When passing a virtual crossing, the Alexander numbers do not change.
Taken up to multiples of $t^\ell$, $f(t)$ gives a well-defined invariant of almost classical braids, which is also independent of the choice of Alexander numbering provided $\beta$ is not a split braid.  


\subsection{Alexander invariants (reprise)}

Let $K$ be a virtual knot diagram, which has been realized as the closure  $\widehat{\beta}$ for a virtual braid $\beta \in \VB_k$. We use the braid realization to give an alternative presentation matrix for the Alexander invariants. The main difference from Definition~\ref{method-1} is that we have generators $x_1, \ldots,  x_k$ for the strands on the top of $\beta$ and generators $z_1, \ldots, z_k$ for the strands on the bottom. This approach is especially convenient in deriving formulas for the Alexander invariants of periodic virtual knots $K$ which are represented as the closures of periodic virtual braids.
  

\begin{definition} \label{method-2} 
Suppose $K$ is a virtual knot diagram with $n$ crossings, and apply Theorem~\ref{Thm-Alex} to write $K= \widehat{\beta}$, where $\beta$ is a virtual braid on $k$ strands. We
label the arcs on top of $\beta$ by $x_1,  \ldots, x_k$ and the arcs on the bottom of $\beta$ by $z_1, \ldots, z_k$, and we use $y_1,\ldots, y_r$ to label the internal arcs $\beta$, which are the arcs that do not start or end at the top or bottom of $\beta$.
(Note that we will typically have $r = n-k$, unless $n<k$ or some strands of $\beta$ pass over all the other strands.)

This gives a presentation for the knot group 
$$G_K = \langle x_1,\ldots, x_k, y_1, \ldots, y_r, z_1, \ldots, z_k \mid R_1,\ldots, R_n,S_1, \ldots, S_k\rangle, $$ 
where the $R_i$ are the usual Wirtinger relations coming from the crossings of $\beta$ and the $S_i$ are the relations which correspond to setting $x_i = z_i$ for the closure $\widehat{\beta}.$

The \emph{Jacobian matrix} $B$ associated to this presentation of $G_K$ is the $(n+k) \times (n+k)$ matrix 
with rows ordered by the relations $R_1,\ldots, R_n, S_1, \ldots S_k$  and columns ordered by the generators $x_1,\ldots, x_k, y_1, \ldots, y_r, z_1, \ldots, z_k$ and $(i,j)$ entry given by the Fox differentiating the $i$-th relation with respect to the $j$-th generator and applying  the abelianization map $G_K \to \ZZ$ sending each of the generators to $t$. 
\end{definition}

\begin{figure}
\begin{center}
\includegraphics[scale=0.75]{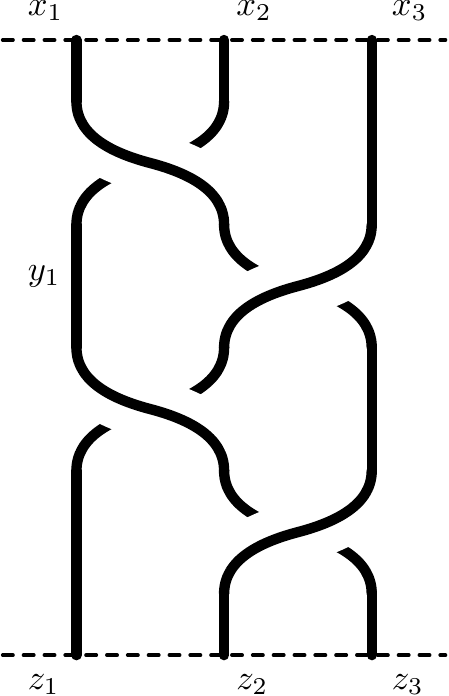}
\end{center} 
\caption{The classical braid $\beta = (\sigma_1 \sigma_2)^2$.} \label{braid-ex}
\end{figure}

For example, consider the classical braid $\beta = (\sigma_1 \sigma_2)^2$ in Figure~\ref{braid-ex}. The knot group $G_K$ of  $K=\widehat{\beta}$ has a presentation with generators $x_1, x_2, x_3, y_1, z_1, z_2, z_3$ and relations
 
\begin{eqnarray*}
&R_1= x_1 x_2 x_1^{-1} y_1^{-1}, \;
R_2 = x_3^{-1} x_1 x_3 z_2^{-1}, \;
R_3 = y_1 x_3 y_1^{-1} z_1^{-1}, \;
R_4 = z_2^{-1} y_1 z_2 z_3^{-1},& \\
&S_1 = z_1 x_1^{-1}, \; 
S_2 = z_2 x_2^{-1}, \;
S_3 = z_3 x_3^{-1}. \qquad \qquad& \\ 
\end{eqnarray*}

We write the Jacobian matrix of $G_K$ with rows ordered by the relations $R_1, \ldots, R_4, S_1, \ldots, S_3$ and columns ordered by the generators $x_1, x_2, x_3, y_1, z_1, z_2, z_3$;  it is the square matrix

$$
B=\begin{bmatrix}
1-t & t & 0 & -1 & 0 & 0 & 0   \\ 
t^{-1} & 0 & 1-t^{-1} & 0 & 0 & -1 & 0\\ 
0 & 0 & t & 1-t & -1 & 0 & 0   \\ 
0 & 0 & 0 & t^{-1} & 0 & 1-t^{-1} & -1 \\ 
-1 & 0 & 0 & 0 & 1 & 0 & 0  \\ 
0 & -1 & 0 & 0 & 0 & 1 & 0 \\ 
0 & 0 & -1 & 0 & 0 & 0 & 1  
\end{bmatrix}.
$$

Note that the Wirtinger presentation for the virtual knot is
$$G_K = \langle x_1,x_2,x_3, y_1 \mid \tilde{R}_1, \tilde{R}_2, \tilde{R}_3, \tilde{R}_4 \rangle,$$
where $\tilde{R}_i$ is obtained from $R_i$ above by substituting $x_j$ for $z_j$. Applying Definition~\ref{method-1} to this presentation of $G_K$ gives the Jacobian 
$$A=\begin{bmatrix}
1-t & t & 0 & -1     \\ 
t^{-1} & -1 & 1-t^{-1} & 0 \\ 
-1 & 0 & t & 1-t   \\ 
0 & 1-t^{-1} & -1 & t^{-1}  
\end{bmatrix}.$$
Notice that this matrix can also be obtained from the upper left $4 \times 4$ block of $B$ by combining (that is, adding) the $x_j$ and $z_j$ columns. 

In general, the matrix derived from Definition~\ref{method-2} will be given by
$$
B =   
\left[
\begin{array}{c|c|c}
\partial R_i/\partial x_j & \partial R_i/\partial y_j & \partial R_i/\partial z_j\\ \hline
\partial S_i/\partial x_j & \partial S_i/\partial y_j & \partial S_i/\partial z_j
\end{array} \right] = \left[
\begin{array}{c|c|c}
x_* & y_* & z_* \\ \hline
-I_k & 0 & I_k  
\end{array} \right],
$$
and it is related to the matrix derived from Definition~\ref{method-1}, which is
$$A = 
 \left[
\begin{array}{c|c}
x_* +z_*& y_*  \\ \hline
0 & 0 \\
\end{array} \right]
\cong
\left[
\begin{array}{c|c}
x_* +z_*& y_*  \\
\end{array} \right].
$$

We now state and prove a result analogous to Proposition~\ref{linearcombo} for the Jacobian matrix $B$ from Definition~\ref{method-2} for an almost classical knot $K$. We will assume that $K$ has been realized as the closure of an almost classical braid $\beta \in \VB_k$. 

\begin{lemma}\label{LC-method-2}
Suppose $K$ is an almost classical knot diagram with $K=\widehat{\beta}$ for $\beta \in \VB_k$ an almost classical braid with $n$ crossings. 
Let $B$ be the Jacobian matrix (constructed as in Definition~\ref{method-2}) with relations $R_1, \ldots, R_n$ from the crossings and $S_1, \ldots, S_k$ from the identities $z_i x_i^{-1}$ as above. Note that $B$ is an $(n+k) \times (n+k)$ matrix. For $i=1,\ldots, n,$ let $\lambda^R_i$ be the Alexander number of the $i$-th crossing of $\beta$, as in Figure~\ref{crossing-AC-num-braid}, and for $i=1,\ldots, k$, let $\lambda^S_{i}$ be the Alexander number of the $i$-th strand  at the top of $\beta$. 

Then ${\displaystyle \sum_{i=1}^{n+k}} \omega_i B_{i,*} = 0$, where
$$\omega_i = 
\begin{cases}
t^{\lambda^R_i} &\text{for $1\le i \le n$ and $\ep_i = -1$,} \\
t^{\lambda^R_i + 1} &\text{for $1\le i \le n$  and $\ep_i = +1$,}  \\
t^{\lambda^S_{i-n}} & \text{for $n+1 \le i \le n+k$.}
\end{cases}
$$
\end{lemma}

\begin{figure} 
\begin{center}
\includegraphics[scale=0.80]{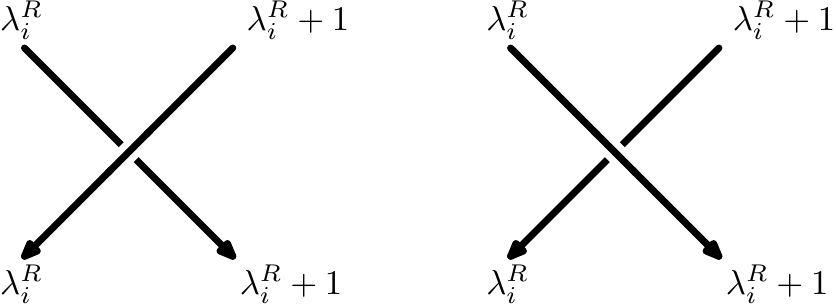}
\end{center}
\caption{The Alexander numbers at the $i$-th crossing for $\varepsilon_i=1$ on the left and $\varepsilon_i=-1$ on the right.}\label{crossing-AC-num-braid}
\end{figure}

\begin{proof}
From Claim~\ref{LC-method-1},  if we have a matrix $A$ constructed from Definition~\ref{method-1}, then $\sum_{i=1}^n \vartheta_i A_{i,*} = 0$, where
$$\vartheta_i = 
\begin{cases}
t^{\lambda^R_i} & \text{if $\ep_i = -1$ for $1 \le i \le n,$} \\
t^{\lambda^R_{i} + 1} & \text{if $\ep_i = +1$ for $1\le i \le n.$} 
\end{cases}
$$
Recall from the proof of Claim~\ref{LC-method-1}, that $\sum_{i=1}^n \omega_i A_{i,j}$ corresponds to the sum of all the $\pm t^k$ labels on the arc corresponding to the $j$-th column (so if column $j$ was corresponding to an $x_k$ label at the top, then it would be the sum of the labels on the arc that is labeled $x_k$), where the labels were assigned as in Figure~\ref{AN-at-i}. When we label the arcs as in Definition~\ref{method-2} for a virtual braid, the $y_i$ arcs are not affected, and the only difference is that the arcs that were previously labelled $x_i$ are now cut in half (on the unknotted part of the braid), and we have both an $x_i$ and $z_i$ arc. Previously, all the labels on $x_i$ cancelled in pairs, so we had a sum of $0$. But now that we have split $x_i$ into two pieces, there is a pair that gets separated. The label at the top of the braid will be a $t^{\lambda^S_i}$ (positive since it is an ingoing arc), which will sit on the $x_i$ arc. At the bottom of the braid, there will be a $-t^{\lambda^S_i}$ for the outgoing arc, which will sit on the $z_i$ arc. Thus, before considering the $S_i$ relations, we will have a sum of $t^{\lambda^S_i}$  on the $x_i$ arc (all other labels within the braid on $x_i$ will cancel as before), and a sum of $-t^{\lambda^S_i}$  on the $z_i$ arc (all other labels within the braid on $z_i$ will cancel as before). It remains to look at the $S_i$ relations that contribute to the linear combination. Recall that these correspond to $x_i = z_i$, which will give a $-1$ in the $x_i$ column and a $+1$ in the $z_i$ column, and zeros everywhere else. Hence, if we multiply this row by $t^{\lambda^S_i}$, we will get a sum of zero for the $x_i$ and $z_i$ columns. As stated above, the $y_i$ columns remain unchanged from the previous method, and since the $S_i$ relations do not involve $y_i$'s, the sum in their columns remains zero as well. Now the row corresponding to $S_i$ is actually the $(n+i)$-th row of $B$, so for $n+1 \le i \le n+k$, we take $\omega_i = t^{\lambda^S_{i-n}}$.
\end{proof}


\section{Periodic Virtual Knots} 

 

\subsection{Basic Definitions} \label{sec-3-1}
In this section, we recall the definition of periodicity for virtual knot diagrams, Gauss codes and Gauss diagrams. We write out Wirtinger presentations for the knot groups $G_K$ and $G_{K_*}$ of a periodic virtual knot and its quotient, and we show how the Jacobian matrix of $K$ is related to that of $K_*$ in terms of circulant block matrices.

\begin{definition}
A virtual knot diagram $K$ is called \emph{periodic} with \emph{period $q$} if it misses the origin and is invariant under a rotation in the plane by an angle of $2 \pi / q$ about the origin. 
 
Given a periodic virtual knot diagram $K$, its \emph{quotient knot} $K_*$ is the knot obtained by closing up one fundamental domain of $K$. To be specific, take one fundamental domain of $K$, which is a pie-shaped region centered at the origin with angle $2 \pi/q$, and connect the arcs along the upper and lower edges with concentric circular arcs.
\end{definition}

Figure~\ref{PVK} illustrates a periodic virtual knot on the left and its quotient on the right. In that picture, $\tau$ is a virtual tangle diagram.

\begin{figure}
\begin{center}
\includegraphics[scale=0.75]{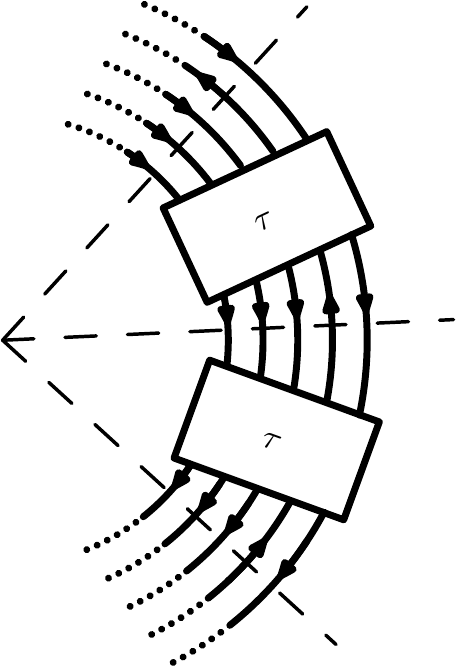} \qquad \qquad
\includegraphics[scale=0.80]{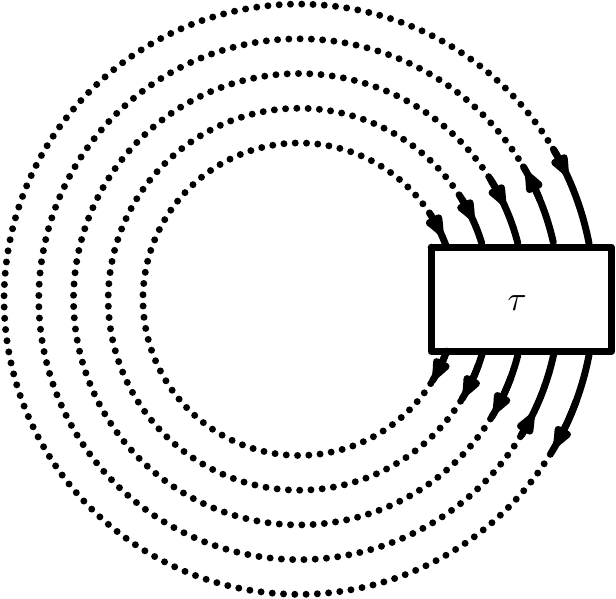}
\end{center}
\caption{A periodic virtual knot and its quotient knot.}
\label{PVK}
\end{figure}

\begin{definition}
The  \emph{linking number} $k$ of a periodic virtual knot diagram $K$ is the absolute value of the intersection number of a ray $R$ emanating from the origin with the virtual knot diagram. As usual, we sum up intersection points, and they count positively if they come from an arc of $K$ that winds counter-clockwise around the origin, otherwise they count negatively.  
\end{definition}

For example, the periodic virtual knot in Figure~\ref{PVK} has linking number 3. A well known example is the $(3,3,3)$-pretzel knot $9_{35}$, and Figure~\ref{pretzel-1} depicts a 3-periodic diagram for $9_{35}$ with linking number $k=2$ and quotient $K_*$ a diagram of the unknot.

\begin{figure}
\begin{center}
\includegraphics[scale=0.65]{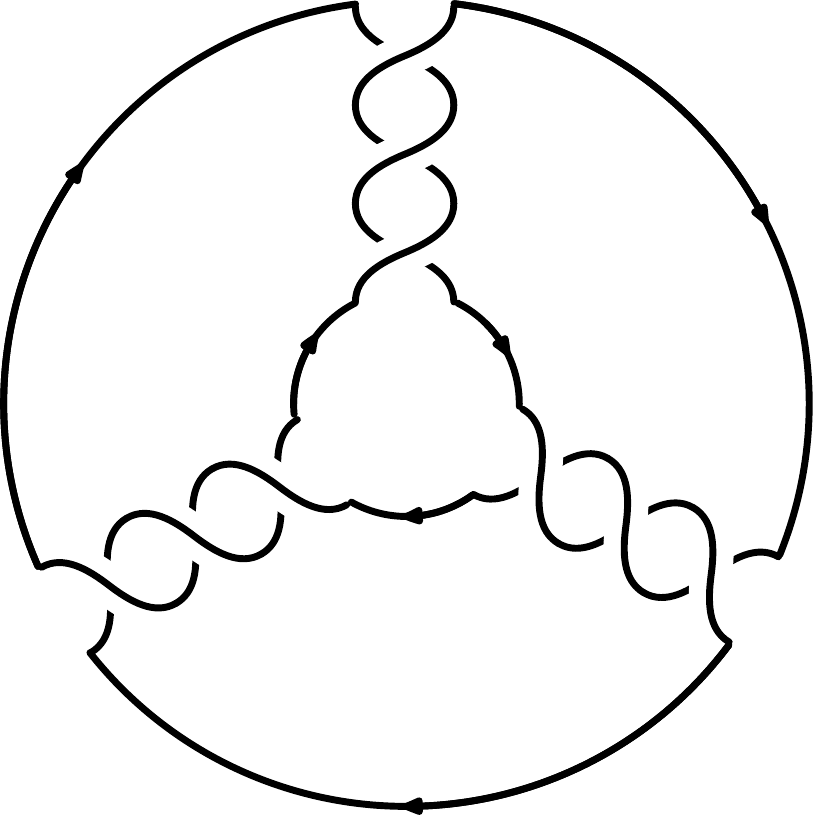}  
\end{center}
\vspace{-2mm}
\caption{A 3-periodic knot diagram for the pretzel knot $9_{35}$.}
\label{pretzel-1}
\end{figure}

Periodicity of a virtual knot diagram is reflected in its Gauss code. For instance, the diagram for the pretzel knot $9_{35}$ has underlying Gauss code
$$C=\hbox{\tt U1-O2-U3-O6-U5-O4-U7-O8-U9-O3-U2-O1-U4-O5-U6-O9-U8-O7-}.$$ 

In general, we say a Gauss code is $q$-periodic if the {\tt O/U} and {\tt +/-} patterns repeat with period $q$, and whenever {\tt Oi} goes to {\tt Oj} in the next period, then {\tt Ui} goes to {\tt Uj}. (In other words, the periodic transformation is a well-defined map on the crossings of $K$.)

For instance, the Gauss code for $9_{35}$ is a signed word of length 18, and we will write it in the following way to emphasize its 3-periodicity.
\begin{eqnarray*}
C=&\hbox{\tt U1-O2-U3-O6-U5-O4-}& \\
&\hbox{\tt U7-O8-U9-O3-U2-O1-}&\\
&\hbox{\tt U4-O5-U6-O9-U8-O7-}.&
\end{eqnarray*}
Thus, the periodic transformation of $C$ is the map (of ordered sets)
$$
{\tt \{1,2,3\} \mapsto \{7,8,9\} \mapsto \{4,5,6\} \mapsto \{1,2,3\}.}$$
(Note, a more succinct description of this map is $i \mapsto i +6 \mod 9$.) Applying this transformation to $C$, the new Gauss code is easily seen to be equivalent to the original one under a cyclic permutation of $C$. 

A Gauss diagram is said to be $q$-periodic if it is invariant under a rotation of an angle of $2\pi/q$. This is equivalent to the condition of $q$-periodicity for the associated Gauss code, but it is  easier to visualize. 

Clearly, if $K$ is a $q$-periodic virtual knot diagram, then its Gauss code and Gauss diagram are also both $q$-periodic. In that case, the Wirtinger presentation associated to the periodic diagram as in Equation \eqref{gp-pres-2} is symmetric, as we explain. 

Suppose $K$ has $qn$ crossings. Pick a basepoint and label the chords 
$$c_{1,0}, \ldots, c_{n,0}, c_{1,1}, \ldots, c_{n,1}, 
\ldots, c_{1,q-1}, \ldots, c_{n,q-1}$$ of the Gauss diagram $D_K$ 
in the order in which their arrowheads are encountered as one travels around the knot. Because $K$ is $q$-periodic, we can assemble them 

$$\begin{array}{ccc}
c_{1,0} & \cdots & c_{n,0} \\
\vdots & & \vdots \\
c_{1,q-1} & \cdots & c_{n,q-1}
\end{array}
$$ 
so that the periodic action is the vertical shift sending $c_{i,j}$ to $c_{i,j+1}$ for $j = 0,\ldots, q-1$, with $j+1$ taken $\mod q,$ which is to say that if $j=q-1$, then $j+1$ equals 0.

We can label the arcs
$$
\begin{array}{ccc}
a_{1,0} & \cdots & a_{k,0} \\
\vdots & & \vdots \\
a_{1,q-1} & \cdots & a_{n,q-1}
\end{array}
$$ 
accordingly, so that, for $i=2, \ldots, n$, the arc $a_{i,j}$ starts at the arrowhead of $c_{i-1,j}$ and ends at $c_{i,j}$; see Figure~\ref{GD-periodic}. When $i=1$, the arc $a_{1,j}$ starts at $c_{n,j-1}$ and ends at $c_{1,j}$.

\begin{figure}
\centering
 \includegraphics[scale=0.9]{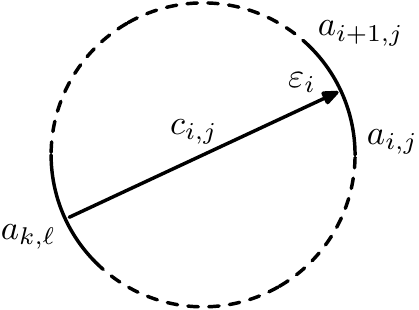} \qquad \qquad \includegraphics[scale=0.9]{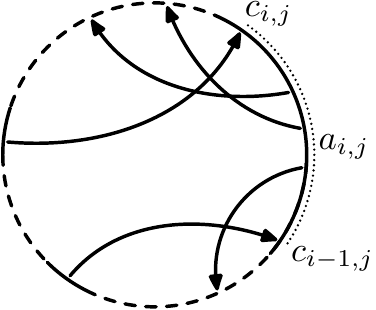}  
\caption{On left, $c_{i,j}$ separates $a_{i,j}$ and $a_{i+1,j}$ and has arrowfoot on $a_{k,\ell}$. \qquad
On right, $a_{i,j}$ is the dotted arc between the arrowheads of $c_{i-1,j}$ and $c_{i,j}$.}
\label{GD-periodic}
\end{figure}

If $\ep_{i,j} = \pm 1$ is the sign of the chord $c_{i,j}$, then periodicity implies that $\ep_{i,j} = \ep_{i,j+1}$, so we will simply write $\ep_i$. If $a_{k,\ell}$ denotes the arc on which the arrowfoot of $c_{i,j}$ lies, then by periodicity  $a_{k,\ell+1}$ is the arc on which the arrowfoot of $c_{i,j+1}$ lies. (Here, $\ell +1$ and $k+1$ are taken $\mod q.$)

With these assumptions, the Wirtinger relation of the crossing $c_{i,j}$ is given by
$$r_{i,j} = a_{k,\ell}^{-\ep_{i}} \, a_{i,j}^{ } \, a_{k,\ell}^{\ep_{i}} \, a_{i+1,j}^{-1} $$
for $i=1,\ldots, n-1$ and $j=0,\ldots, q-1.$
When $i=n$, we get the relation $r_{n,j} = a_{k,\ell}^{-\ep_{n}} \, a_{n,j}^{ } \,a_{k,\ell}^{\ep_{n}}a_{1,j}^{-1} .$

The resulting Wirtinger presentation of the $q$-periodic virtual knot $K$ is then
\begin{equation} \label{gp-pres-3}
G_K = \langle a_{i,j}  \mid r_{i,j} \rangle,
\end{equation} 
where $1 \leq i \leq n, \; 0 \leq j \leq q-1$ in \eqref{gp-pres-3}.
This presentation admits a $\ZZ/q$ symmetry, and the Wirtinger presentation for $K_*$ by obtained as the quotient by adding the relations $a_{i,0} = a_{i,1} = \cdots = a_{i,q-1}$ for $1 \leq i \leq n$, which gives the presentation
\begin{equation} \label{gp-pres-4}
G_{K_*} = \langle a_1, \ldots, a_n \mid r_{1}, \ldots, r_n \rangle,
\end{equation}  
where $a_i$ refers to the equivalence class $\{a_{i,0}, \ldots, a_{i,q-1}\}$ of generators and $r_i$ is the relation $a_{i+1}^{-1} a_{k}^{-\ep_{i}} \,a_{i}^{}\, a_{k}^{\ep_{i}}$ with $i+1$ taken $\mod n$.

\begin{theorem} \label{thm:jaco}
Let $K$ be a virtual knot diagram with period $q$, and let $K_*$ be its quotient knot. If $A$ and $B$ are the Jacobian matrices of the Wirtinger presentations \eqref{gp-pres-4} and \eqref{gp-pres-3} of $G_{K_*}$ and $G_K$, respectively, then 

\begin{equation} \label{eq:cm}
B=
\left[\begin{array}{cccc}
A_0 & A_1 & \cdots & A_{q-1} \\ 
A_{q-1} & A_0 & \cdots & A_{q-2} \\ 
\vdots & \ddots & \ddots & \vdots\\ 
A_1 & \cdots & A_{q-1}& A_0 
\end{array}\right]
\end{equation}
is a block circulant matrix, where $A_0, A_1,\ldots, A_{q-1}$ are square matrices satisfying $A_0 + A_1 + \cdots +A_{q-1} = A$.
\end{theorem}

\begin{proof}
Suppose $K$ is $q$-periodic and label the chords $c_{i,j}$ and arcs $a_{i,j}$ as above. In constructing the Jacobian matrix $B$ associated to
\eqref{gp-pres-3}, we order the rows to correspond with the chords $$c_{1,0}, \ldots, c_{n,0}, c_{1,1}, \ldots, c_{n,1} \ldots, c_{1,q-1} , \ldots, c_{n,q-1}$$ and the columns to correspond with the arcs 
$$a_{1,0}, \ldots, a_{n,0}, a_{1,1}, \ldots, a_{n,1}, \ldots, a_{1,q-1}, \ldots, a_{n,q-1}.$$ 
Then the entry of $B$ in the $(i,j)$-th row and $(i',j')$-th column, which is the entry in the row corresponding to $c_{i,j}$ and column
corresponding to $a_{i',j'}$, is given by
$$B{\left((i,j),(i',j')\right)}=\begin{cases}
t^{-\ep_i} & \text{if $(i',j')=(i,j)$,} \\
-1 & \text{if $(i',j')=(i+1,j)$, or if $i'=1, i=n,$ and $j'=j+1$,} \\
1-t^{-\ep_i} & \text{if $(i',j')=(k,\ell)$,}   \\
0  & \text{otherwise.}
\end{cases}$$
(Recall that $a_{k,\ell}$ is the arc on which the arrowfoot of $c_{i,j}$ lies.)

Notice that for $c_{i,j}$, the $t^{-\ep_i}$ entry will always be in the $j$-th column block of the $i$-th block row; the $-1$ term will always be in the $j$-th column block of the $i$-th block row, unless $i=n$, and then the $-1$ will sit in the $(j+1)$-st column block of the $i$-th block row. The $1-t^{-\ep_i}$ entry, on the other hand, can be in any column block of the $i$-th block row.
Also notice that if the foot of $c_{i,j}$ lies on $a_{k,\ell}$, then periodicity implies that the foot of $c_{i,j+1}$ lies on $a_{k,\ell+1}$ (with $j+1$, $\ell+1$ taken $\mod q$). 
For example, in the $i$-th block row of $B$, we will have
$$
\begin{blockarray}{ccccccccccccccccc}
& a_{i,0} & a_{i+1,0} & & a_{k,0} & a_{i,1} & a_{i+1,1} & & a_{k,1} & 
 a_{i,2} & a_{i+1,2} & & a_{k,2} &  \cdots    \\
\begin{block}{c[cccc|cccc|cccc|cccc]}
& & & & & & & & & & & & & & & & &    \\
c_{i,1} & t^{-\ep_i} & -1 & & & & & &  1-t^{-\ep_i} & & & & &   \\
& & & & & & & & & & & & & & & & &   \\ 
\cline{2-17}
& & & & & & & & & & & & & & & & &   \\
c_{i,2} & & & & & t^{-\ep_i} & -1 & & & & & & 1-t^{-\ep_i}   \\
& & & & & & & & & & & & & & & & &    \\ 
\cline{2-17}
& & & & & & & & & & & & & & & & &   \\
\vdots & & \cdots & & & & \cdots & & & & \cdots & &   \cdots \\ 
& & & & & & & & & & & & & & & & &   \\
\end{block}
\end{blockarray}
$$Notice that the matrix $B$ satisfies $B_{i,j} = B_{i+n,j+n}$, and therefore it is block circulant with $n \times n $ blocks of the desired form:
$$B=
\left[\begin{array}{cccc}
A_0 & A_1 & \cdots & A_{q-1} \\ 
A_{q-1} & A_0 & \cdots & A_{q-2} \\ 
\vdots & \ddots & \ddots & \vdots\\ 
A_1 & \cdots & A_{q-1}& A_0 
\end{array}\right].$$

Next we will show that $\sum_{k=0}^{q-1}A_k = A$, where $A$ is the Jacobian matrix for the quotient knot $K_*$. Notice that in the Wirtinger presentation \eqref{gp-pres-4} for $G_{K_*},$ the relations $r_i$ are obtained from the relations $r_{i,j}$ of \eqref{gp-pres-3} under setting $a_{i,0} = a_{i,1} = \cdots = a_{i,q-1}$ for $1 \leq i \leq n$. The Jacobian matrix for $G_{K_*}$ has $i,j$ entry

$$A(i,j) =\begin{cases}
 t^{-\ep_i} & \text{if $j=i$,} \\
-1 & \text{if $j=i+1 \mod n$,}\\
1-t^{-\ep_i} & \text{if $j=k$,} \\
0 & \text{otherwise.}
\end{cases}$$

Notice that the $i$-th row of $A$ is equal to the sum of the $i$-th rows of block matrices appearing in $B$, which is equivalent to the statement that $\sum_{k=0}^{q-1}A_k = A$. This completes the proof.
\end{proof}

\subsection{Periodicity and almost classical knots} \label{section-3-2}
Both periodicity and almost classicality are defined for virtual knots in terms of their representative diagrams, and it remains to show that we can find virtual knot diagrams that exhibit both properties at the same time. In this section, we use Manturov projection to show that if $K$ is a $q$-periodic virtual knot diagram representing an almost classical knot, then one can find a $q$-periodic almost classical diagram equivalent to $K$.
 
We begin with a few useful lemmas.

\begin{lemma} \label{Lemma-3-2-1}
Suppose $K$ is a virtual knot diagram representing an almost classical knot. Then its image $P_f(K)$ under Manturov projection is virtually isotopic to $K$.
\end{lemma}
\begin{proof}
Since $K$ represents an almost classical knot (but may not be Alexander numberable itself), by the definition of almost classical, we have a virtual knot diagram $K'$ which is Alexander numberable and virtually isotopic to $K$. Equivalently, the Gauss diagram $D'$ corresponding to $K'$ has all of its chords of index 0. Applying Proposition~\ref{Prop-Projection}, it follows that $P_f(K)$ is virtually isotopic to $P_f(K') = K'$, which is virtually isotopic to $K$.
\end{proof}

Note that in the above lemma, $P_f(K)$ need not be an almost classical diagram. In fact, even if $K$ represents an almost classical knot, $K$ and $P_f(K)$ may fail to be almost classical diagrams. On the other hand, for any virtual knot diagram $K$, its image $P_f^\infty(K)$ under stable projection is an almost classical diagram.

\begin{lemma} \label{Lemma-3-2-2}
Let $K$ be a $q$-periodic virtual knot diagram with quotient $K_*$. For $j=0,\ldots, q-1,$ let  $c_{1,j}, \ldots, c_{n,j}$ be the chords in the $j$-th period of the Gauss diagram $D_K$ for $K$, and let $c_{1}, \ldots, c_{n}$ be the corresponding chords in the Gauss diagram $D_{K_*}$ for $K_*$. Then the index satisfies $I(c_{i,j}) = I(c_i)$ for $i=1,\ldots, n$ and $j =0,\ldots, q-1.$ In particular, the index $I(c_{i,j})$ of a chord is independent of its period $j =0,\ldots, q-1.$  
\end{lemma}
\begin{proof}
Let $\pi: D_K \to D_{K_*}$ be the mapping of Gauss diagrams. It is a covering map of oriented trivalent graphs preserving the signs.

According to Definition~\ref{chord-index}, the index of $c_{i}$ is given by counting the arrowheads and arrowtails with sign along the arc $\alpha_i$ of $D_{K_*}$ from the arrowtail of $c_{i}$ to its arrowhead. One can perform this computation upstairs in $D_K$ after lifting $\alpha_i$ under $\pi$. If $\widetilde{\alpha}_i$ denotes the lift starting at the arrowtail of $c_{i,j}$, then it will end at the arrowhead of $c_{i,k}$ for some $k=0,\ldots, q-1.$ The index along $\widetilde{\alpha}_i$ differs from the index of $c_{i,j}$ by a similar count along an arc $\beta$ of $D_K$ from the arrowhead of $c_{i,k}$ to the arrowhead of $c_{i,j}.$ Taking its image $\pi(\beta)$ under $\pi$, we obtain an arc that winds around $D_{K_*}$ $|j-k|$ times (because of the periodicity), and consequently the index along $\beta$ is necessarily zero. (The index around an entire diagram will always be zero because the all arrowheads will cancel with their arrowtails in the sum). It follows that $I(c_i) = I(c_{i,j})$ for $j=0,\ldots, q-1,$ and this completes the proof.
\end{proof}

\begin{corollary}
\label{ACquotient}
If $K$ is a $q$-periodic almost classical diagram with quotient $K_*$, then $K_*$ is also almost classical.  \end{corollary}

\begin{lemma} \label{Lemma-3-2-3}
Suppose $K$ is a $q$-periodic virtual knot diagram. Then its image $P_f(K)$ under Manturov projection is also $q$-periodic.
\end{lemma}
\begin{proof}
 At the level of the virtual knot diagram, Manturov projection is the process of replacing all of the odd (real) crossings with virtual crossings. 
 By the previous lemma, if $K$ is $q$-periodic and has an odd crossing, then so is every other crossing in its $\ZZ/q$-orbit.  This fact ensures that if $K$ is $q$-periodic, then so is $P_f(K).$   
\end{proof}

\begin{theorem}
\label{ACrep}
If $K$ is a $q$-periodic virtual knot diagram which represents an almost classical knot, then $\bar{K}=P_f^\infty(K)$ is a $q$-periodic almost classical diagram representing the same virtual knot. 
\end{theorem}

\begin{proof}
Since $K$ represents an almost classical knot, repeated application of Lemma~\ref{Lemma-3-2-1} implies that $P_f^\infty(K)$ is virtually isotopic to $K$. On the other hand, since $K$ is $q$-periodic, repeated application of Lemma~\ref{Lemma-3-2-3} ensures that $P_f^\infty(K)$ is also $q$-periodic. That completes the proof of the theorem.
\end{proof}

The next proposition is a slightly  more general result along the same lines.

\begin{proposition}  \label{prop-ManProj}
Suppose $K$ is a $q$-periodic virtual knot diagram. Then its image  $P_{f}^\infty(K)$ under stable Manturov projection is an almost classical $q$-periodic knot diagram. 
\end{proposition}
\begin{proof}
This follows by repeated application of Lemma~\ref{Lemma-3-2-3}, together with the fact that $P_{f}^\infty(K')$ is an almost classical diagram for any virtual knot diagram $K'$.   
\end{proof}

%

Even though Manturov projection $P_f$ is defined at the level of virtual knot diagrams, Proposition~\ref{Prop-Projection} ensures that it is a well-defined operation on the level of virtual knots.  The next corollary will allow us to eliminate periods for a general virtual knot $K$ by applying the Murasugi conditions to the almost classical knot $\bar{K} = P_f^\infty(K)$ obtained by stable projection.

\begin{corollary} \label{Cor-not-per}
Let $K$ be a virtual knot, and $\bar{K} = P_f^\infty(K)$ be the associated almost classical knot obtained by stable Manturov projection. If $\bar{K}$ does not admit a $q$-periodic diagram, then neither does $K$.
\end{corollary}

We derive a formula for the writhe polynomial for periodic virtual knots.
Let $K$ be a virtual knot with Gauss diagram $D$,
and let $w_n(D) = \sum_{I(c)=n} \ep(c)$ be the $n$-writhe of $D$, which is an invariant of the virtual knot $K$ for $n \neq 0$  (see \cite{Satoh-Taniguchi-2014}). The writhe polynomial is defined by setting
$$W_K(t) = \sum_{n \in \ZZ} w_n(D)\, t^n -{\rm Wr}(D),$$
where ${\rm Wr}(D)$ is the total writhe of $D$. In \cite{Cheng-Gao}, Cheng and Gao show that the writhe polynomial $W_K(t)$ is an invariant of the virtual knot $K$, and in \cite{Bae-Lee}, Bae and Lee give a formula for $W_K(t)$ for periodic virtual knots. The next result recovers their formula as an immediate consequence of Lemma \ref{Lemma-3-2-2}.
\begin{proposition}  
If $K$ is a q-periodic virtual knot with quotient $K_*$, then its writhe polynomial satisfies $W_K(t) =  q \cdot W_{K_*}(t).$ 
\end{proposition}


\subsection{Periodic virtual braids}

In this section, we will show that every periodic virtual knot diagram $K$ can be realized as the closure of a periodic braid; in other words, $K = \widehat{\beta^q}$ for some virtual braid $\beta$. It can be viewed as an equivariant version of Alexander's theorem, and it is proved via an equivariant braiding process. Whenever we have $K = \widehat{\beta^q}$, it is clear that the linking number equals the braid index and that the quotient knot is given by $K_* = \widehat{\beta}$.

\begin{theorem}\label{mperiod}
(i) A virtual knot diagram $K$ is $q$-periodic if and only if there exists a $q$-periodic Gauss code representing it. \\
(ii) A virtual knot diagram $K$ is $q$-periodic if and only if 
it can be realized as the closure of the $q$-periodic braid, that is, $K=\widehat{\beta^q}$ for some $\beta \in \VB_k$.
\end{theorem}
 
\begin{proof}
For both (i) and (ii), one direction is clear. For instance, if $K$ is a $q$-periodic virtual knot diagram, then its corresponding Gauss code  is obviously $q$-periodic. Likewise, if $K =\widehat{\beta^q}$ is the closure of a periodic braid, then obviously $K$ is itself a $q$-periodic virtual knot diagram.

To show the other directions, we will construct a  periodic virtual knot diagram $K$ from a $q$-periodic Gauss code $C$. In the construction, we will further arrange that the arcs wind monotonically around the origin, thus it will follow that the virtual knot diagram we construct is in fact the closure of a periodic virtual braid.

Assume then that $C$ is a $q$-periodic Gauss code, so its Gauss diagram will then have $qn$ chords, which we list  
$$c_{1,0}, \ldots, c_{n,0}, \; c_{1,1}, \ldots, c_{n,1}, 
\ldots, c_{1,q-1}, \ldots, c_{n,q-1}$$ 
in the order in which their overcrossings are encountered in $C$. Because $C$ is $q$-periodic, we can assemble them 

$$\begin{array}{ccc}
c_{1,0} & \cdots & c_{n,0} \\
\vdots & & \vdots \\
c_{1,q-1} & \cdots & c_{n,q-1}
\end{array}
$$ 
so that the periodic action is the vertical shift sending $c_{i,j}$ to $c_{i,j+1}$ for $j = 0,\ldots, q-1$, with $j+1$ taken $\mod q$ (so if $j=q-1$, then $j+1$ equals 0).

To draw the periodic virtual knot diagram, we draw the crossings $c_{i,j}$ in the plane according to the sign $\ep_{i}$ (which recall by periodicity is independent of $j$) and such that $c_{i,j}$ goes to $c_{i,j+1}$ under a $2\pi/q$ rotation of the plane.
  
To achieve that, draw the crossings $c_{1,0}, c_{1,1}, \ldots, c_{1,q-1}$ equally spaced around a circle, making sure the crossings are all right-handed if $\ep_1=1$ and left-handed if $\ep_1=-1$. Drawing them equally spaced will ensure that each $c_{1,j}$  is sent to $c_{1,j+1}$ under the rotation of $2\pi/q$.


We then do the same for $c_{2,0}, \ldots, c_{2,q-1}$, making sure they are equally spaced, then for $c_{3,0}, \ldots, c_{3,q-1}$, and so on. To ensure that the virtual knot we construct is the closure of a virtual braid, we draw each crossing so that its arcs are oriented clockwise with respect to the origin. This is easy to arrange, for instance by drawing all of $c_{1,0}, \ldots, c_{n,q}$ oriented downwards to the right of the origin, and rotating by an angle of $2\pi j /q$ before drawing the other crossings $c_{i,j}$, see Figure~\ref{Fig-Thm3_4}.

\begin{figure}[h]
\begin{center}
\includegraphics[scale=0.8]{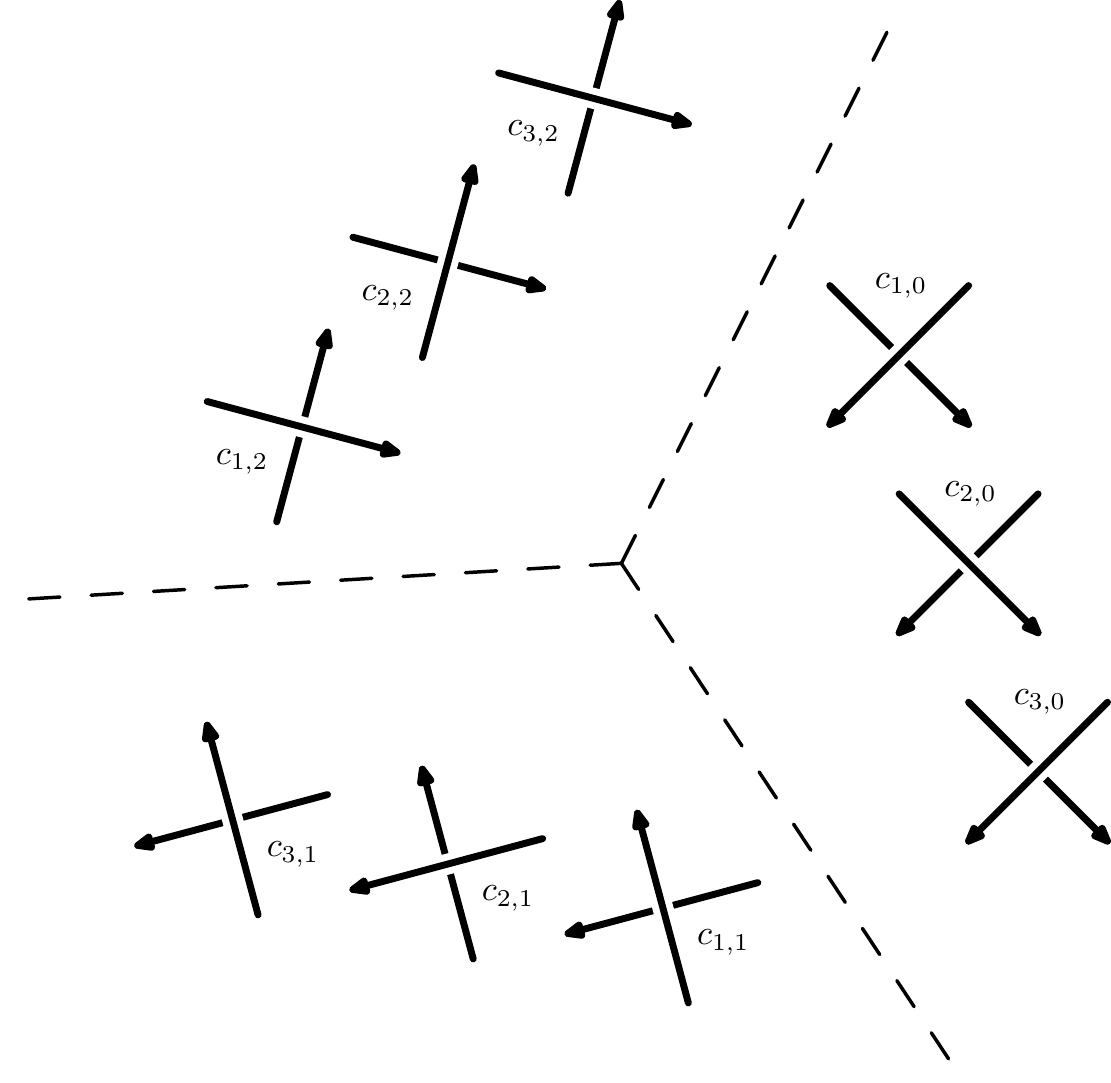}
\end{center} \label{Fig-Thm3_4}
\end{figure}

The result is that we have drawn all $qn$ crossings $$c_{1,0}, \ldots, c_{n,0}, c_{1,1}, \ldots, c_{n,1}, 
\ldots, c_{1,q-1}, \ldots, c_{n,q-1}$$ symmetrically, and we complete the diagram using an equivariant braiding process. For instance, reading the first segment of the Gauss code tells us to connect either the over or under-crossing arc from the first crossing $c_{1,0}$ to one of the arcs of another crossing, say $c_{i,j}$. By periodicity, the same arc of $c_{1,\ell}$ will be connected to the corresponding arc of $c_{i,j+\ell}$, with $j+\ell$ taken $\mod q.$ Thus, in total there will be $q$ connecting arcs, and we draw them equivariantly with respect to the $\ZZ/q$ action. This guarantees the resulting virtual knot diagram will be $q$-periodic.

To ensure we end up with the closure of a virtual braid, we draw the connecting arcs so they wind monotonically around the origin. This process will typically produce a large number of additional crossings, all of which are taken to be virtual crossings of the resulting periodic virtual knot diagram. See Example~\ref{ex-pvb} to see this process carried out for the pretzel knot $9_{35}.$
\end{proof}

The next result is a consequence of Theorems~\ref{ACrep} and~\ref{mperiod}.
\begin{corollary}
\label{ACbraidrep}
If the virtual knot $K$ is $q$-periodic and almost classical, then it can be represented as $K=\widehat{\beta^q},$ the closure of a $q$-periodic braid $\beta \in \VB_k$ that admits an Alexander numbering.
\end{corollary}

\begin{example} \label{ex-pvb}
Consider the classical pretzel knot $9_{35}$, which admits a 3-periodic classical diagram. It is a consequence of the theorem of Edmonds \cite{Edmonds-1984} that $9_{35}$ does not admit a classical $q$-periodic diagram for any $q>3$. Since $9_{35}$ is a genus one knot, this follows from the general bound $q \leq 2g(K)+1$ on the possible periods of a classical knot diagram $K$, where $g(K)$ denotes the Seifert genus of $K$.

On the other hand, the knot $K=9_{35}$ has Alexander polynomial $\Delta_K(t)= 7t^2 - 13t +7,$ which satisfies Murasugi's conditions (Theorem~\ref{classicalmur}) for $q=3$ and $k=2,$ see Figure~\ref{pretzel-1}. However, it is impossible to realize $K$ as the closure of a 3-periodic classical braid $\beta^3$, since $\beta$ would necessarily be a braid on two strands, and for any braid in $\beta \in B_2$, the closure $\widehat{\beta^3}$ is necessarily a $(2,n)$ torus knot or link. (Any braid on two strands would only have ${\sigma_i}^{\pm1}$ terms, so it would reduce to either ${\sigma_1}^n$, which is the $(2,n)$ torus knot; or ${\sigma_1s}^{-n}$, which is the $(2,-n)$ torus knot).

However, Theorem~\ref{mperiod} tells us that $K$ can be realized as the closure of a 3-periodic virtual braid. On the left of Figure~\ref{pretzel-2} is a 3-periodic tangle diagram for $K$ as a classical knot that attempts to wind monotonically around the origin. It is not  a braid because monotonicity fails along the six dashed arcs in that figure.

\begin{figure}[h]
\begin{center}
\includegraphics[scale=0.38]{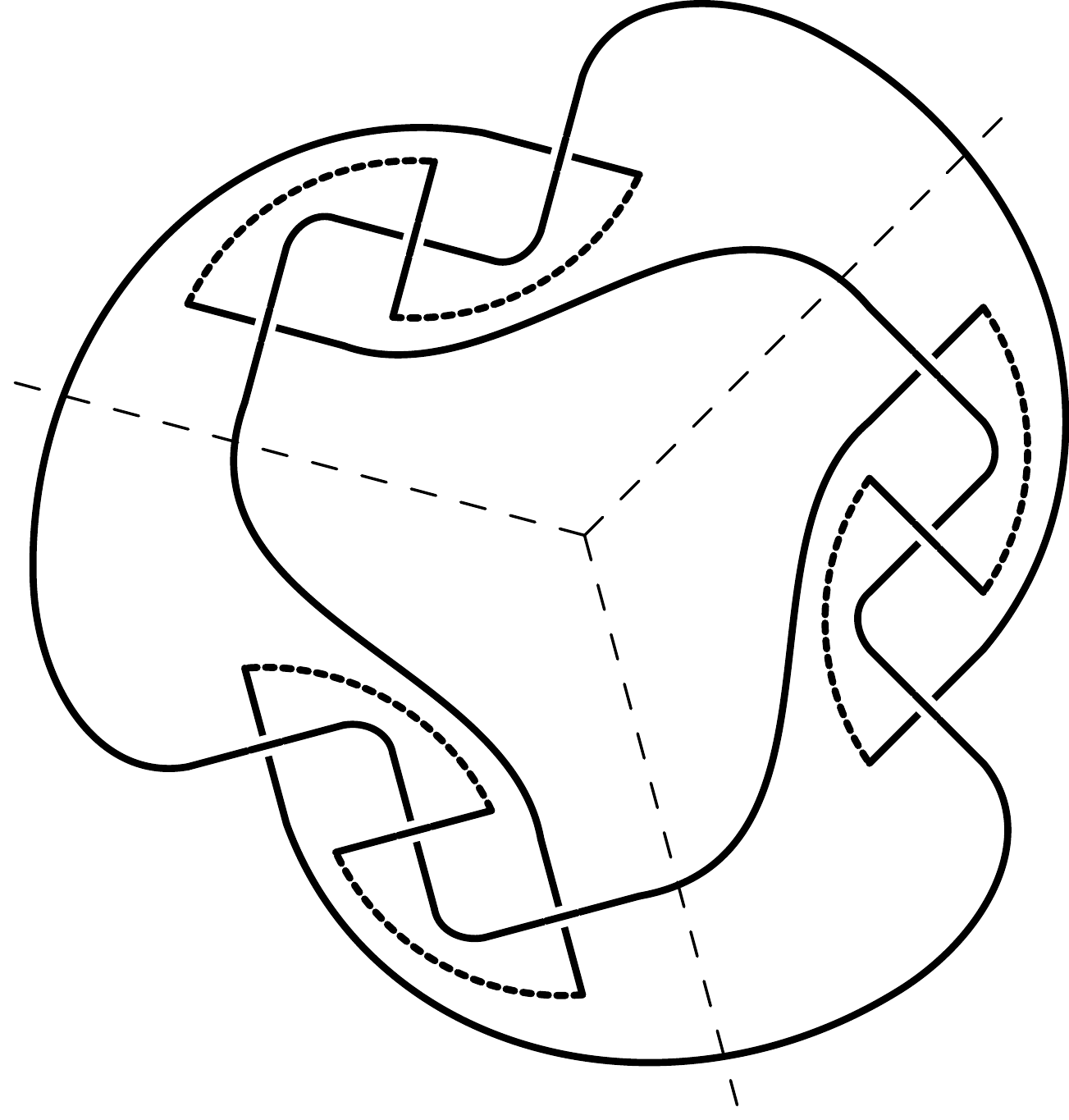}  
\qquad \qquad
\includegraphics[scale=0.38]{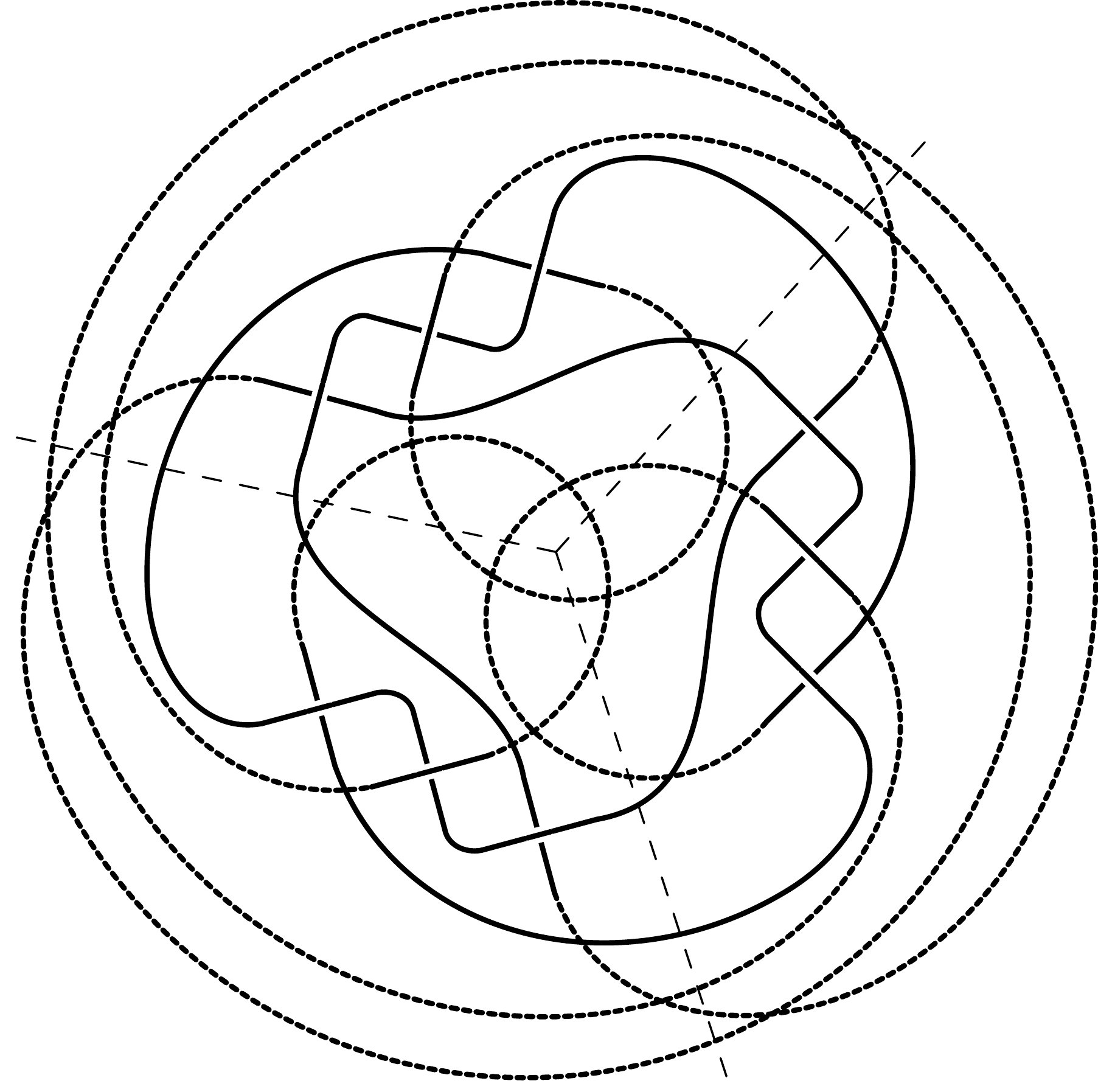}  
\end{center}

\caption{Converting a periodic tangle into a periodic virtual braid  for $9_{35}$.}
\label{pretzel-2}
\end{figure}

On the right of Figure~\ref{pretzel-2} is the result of replacing these six arcs with arcs that wind monotonically around the origin. This creates many new crossings, and all of them are virtual.  As in the proof of Theorem~\ref{mperiod}, these arcs are added so as to preserve the periodicity of the diagram, and the result is a 3-periodic virtual braid diagram for $9_{35}.$
One can check that the resulting braid $\beta \in \VB_8$ is given by the braid word 
$$\beta = \tau_2 \tau_5 \sigma_4 \tau_1 \tau_3 \sigma_4 \tau_5 \tau_7 \sigma_4 \tau_3 \tau_6.$$
\end{example}

Table \ref{tab:periodic-braids} lists the known periods of almost classical knots up to six crossing, along with periodic virtual braids closing up to the given knot.
The numbering of the virtual knots comes from Green's virtual knot table \cite{Green}.
\begin{example}
\label{two_bridge_example}
The almost classical knot 6.90227 coincides with the classical \hbox{2-bridge} knot $6_1 = K(9,7)$ (Schubert normal form).
Every 2-bridge knot is known to be 2-periodic, but this particular one is not fibered since its Alexander polynomial is $\Delta_K(t) \doteq 2t^2-5t+2$, which is not monic. Hence \cite[Corollary 3.4]{Lee-Park-1997} implies that $6_1$ cannot be written as the closure of a 2-periodic classical braid. 
Nevertheless, Table \ref{tab:periodic-braids} shows that it is the closure of the 2-periodic virtual braid
$\beta^2$ for $\beta = \tau_3\tau_2\tau_1 \sigma_2 \tau_2 \tau_3 \sigma_2^{-1} \sigma_4$.
\end{example}

In Theorem~\ref{thm:jaco}, we applied the construction of Definition~\ref{method-1} to determine the Jacobian of any periodic virtual knot $K$. We now show how to apply Definition~\ref{method-2} to give a formula for the Jacobian matrix for a periodic virtual knot $K$ that has been realized as the closure of a periodic virtual braid. As before, the matrix we obtain will be a circulant block matrix, and so completely determined by its first block row. The advantage of using Definition~\ref{method-2} here is that, as we shall see, the block matrices $A_i$ vanish for $i \geq 2$.

We label the arcs of $K$ using labels $x_i^j, y_i^j, z_i^j$ as before, with the $j=0,\ldots, q-1$ indicating the period. In particular, in the $j$-th period, the strands at the top of the braid are labelled $x_1^j,\ldots, x_{k}^j$, and the strands at the bottom are labelled $z_1^j, \ldots, z_{k}^j$. The internal arcs are labelled $y_1^j, \ldots, y_r^j$. We assume that there are $n$ crossings in each period, and so we obtain the relations $R_1^j,\ldots, R_n^j$ for the internal crossings in the $j$-th period, and the relations $S_1^j,\ldots, S_k^j$ corresponding to setting $z_i^j = x_i^{j+1}$. Notice that $n=k+r$ (since we have an $(n+k) \times (n+k)$ matrix).

The Jacobian $B$ is determined by the first block row, which is obtained by differentiating the relations $R_1^0,\ldots, R_n^0$ and $S_1^0,\ldots, S_k^0$ from the $0$-th period.   It follows that this determines the rest of the matrix since  the relations $R^j_i$ and $S^j_i$ are obtained from $R^0_i$ and $S^0_i$ by adding $j$ to the superscripts of all the occurrences of $x_i^0, y_i^0, z_i^0,$ and $x_i^1$ ($j+1$ is taken $\mod q$ here). Recall that the relations $R_1^0,\ldots, R_n^0$ are written in terms of $x_i^0, y_i^0, z_i^0,$ (for $i \in \{1, \cdots n\}$), and the relation $S_i^0$  is written in terms of $z_i^0$ and $x_i^1$.

\begin{figure}[h]
\begin{center}
\includegraphics[scale=0.8]{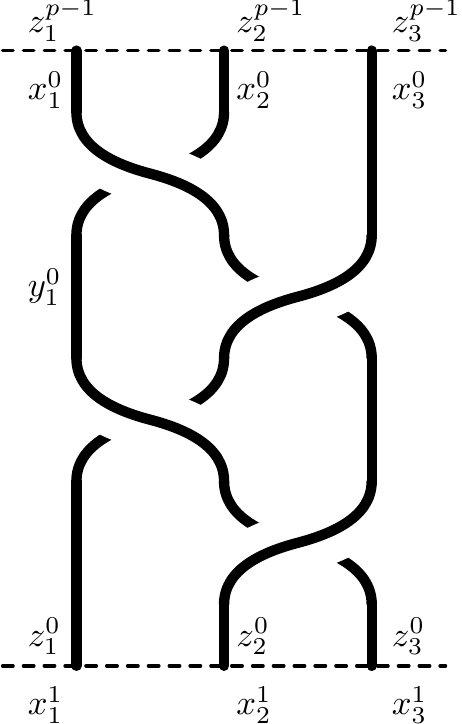}
\end{center}
\end{figure}

Consider for example the $q$-periodic braid with first period given as below (here we assume $q$ is relatively prime to 3). It has relations:
 
\begin{eqnarray*}
R_1^0 &=& x_1^0  \, x_2^0 \, (x_1^0)^{-1}\,(y_1^0)^{-1}, \\ 
R_2^0  &=& (x_3^0)^{-1}\, x_1^0 \, x_3^0\,(z_2^0)^{-1}, \\ 
R_3^0  &=& y_1^0 \, x_3^0\, (y_1^0)^{-1}\, (z_1^0)^{-1}, \\ 
R_4^0 &=& (z_2^0)^{-1}\, y_1^0 \, z_2^0\, (z_3^0)^{-1}, \\ 
S_1^0 &=& z_1^0\, (x_1^1)^{-1}, \\ 
S_2^0 &=& z_2^0\, (x_2^1)^{-1}, \\ 
S_3^0 &=& z_3^0\, (x_3^1)^{-1}. \\ 
\end{eqnarray*}

The first block row of the Jacobian $B$ therefore has the form: 
$$
\begin{blockarray}{c c c c c c c c c c c c c c c c c}
   & x_1^0 & x_2^0 & x_3^0 & y_1^0 & z_1^0 & z_2^0 & z_3^0 & x_1^1 & x_2^1 & x_3^1 & y_1^1 & z_1^1 & z_2^1 & z_3^1 & x_1^2 & \hdots \\
\begin{block}{c [c c c c c c c|c c c c c c c|c c]}
R_1^0 & 1-t & t & 0 & -1 & 0 & 0 & 0  & 0 & 0 & 0 & 0 & 0 & 0 & 0  & 0 & \hdots \\ 
R_2^0 & t^{-1} & 0 & 1-t^{-1} & 0 & 0 & -1 & 0  & 0 & 0 & 0 & 0 & 0 & 0 & 0  & 0 & \hdots \\ 
R_3^0 & 0 & 0 & t & 1-t & -1 & 0 & 0  & 0 & 0 & 0 & 0 & 0 & 0 & 0  & 0& \hdots \\ 
R_4^0 & 0 & 0 & 0 & t^{-1} & 0 & 1-t^{-1} & -1  & 0 & 0 & 0 & 0 & 0 & 0 & 0  & 0 & \hdots \\ 
\cline{2-17}
S_1^0 & 0 & 0 & 0 & 0 & 1 & 0 & 0  & -1 & 0 & 0 & 0 & 0 & 0 & 0  & 0 & \hdots \\ 
S_2^0 & 0 & 0 & 0 & 0 & 0 & 1 & 0  & 0 & -1 & 0 & 0 & 0 & 0 & 0  & 0 & \hdots \\ 
S_3^0 & 0 & 0 & 0 & 0 & 0 & 0 & 1  & 0 & 0 & -1 & 0 & 0 & 0 & 0  & 0 & \hdots \\
\end{block}
\end{blockarray}
$$

In general, we will have $A_2 = A_3 = \cdots = A_{q-1} = [0]$, and
\begin{equation} \label{A0andA1}
A_0 = \left[
\begin{array}{c|c}
* & * \\ \hline
0_{k \times n} & I_{k} \\
\end{array} \right]
 \quad \text{ and } \quad
 A_1 = \left[
\begin{array}{c|c}
0_{n \times k} & 0_{n \times n} \\ \hline
-I_{k} & 0_{k \times n} \\
\end{array} \right].
\end{equation}
The Jacobian matrix of the quotient knot $K_*$ is the matrix   
\begin{eqnarray*}
A = A_0 + A_1 
&=& 
 \left[
\begin{array}{c|c}
* & * \\ \hline
0_{k \times n} & I_{k} \\
\end{array} \right]
+
\left[
\begin{array}{c|c}
0_{k \times k} & 0_{k \times n} \\ \hline
-I_{k} & 0_{k \times n} \\
\end{array} \right] \\
&=&
 \left[
\begin{array}{c|c|c}
x_* & y_* & z_* \\ \hline
-I_{k} & 0_{k \times r} & I_{k} \\
\end{array} \right],
\end{eqnarray*}
where $x_*,y_*$ and $z_*$ represent the matrix entries in the columns corresponding to the differentiating the relations $R_1,\ldots, R_n$ with respect to the $x_i,y_i$ and  $z_i$ generators, respectively. 

\smallskip

\subsection{Periodicity, elementary ideals, and the virtual Alexander polynomial}
\label{Alexander_ideals}
In this subsection we discuss some consequences of  Theorem \ref{braid-rep} as applied to the Alexander module of a periodic virtual knot
(see Corollaries \ref{divisible-three}  and \ref{divisible-four})  and also to the ``virtual Alexander module'' of a periodic virtual knot
(see Propositions \ref{prop:KLS-2014} and \ref{vcrossing_bound}, and Corollary \ref{strong_vcrossing_bound}).

Let $A$ be an $m \times n$ matrix over a commutative ring with unit $R$.
For a non-negative integer $\ell$, the {\it $\ell$-th elementary ideal of $A$}, denoted $\cE_\ell(A)$, is:
the ideal generated by all $(n-\ell) \times (n-\ell)$ minors of $A$ if $0 < n - \ell \leq m$, 
the zero ideal  if $n - \ell > m$, and
$R$ if $n - \ell \leq 0$.
If $M$ is a finitely presented $R$-module and $A$ is a presentation matrix for $M$ then the ideal $\cE_\ell(A)$ is independent of the
choice of the presentation matrix $A$ and so the ideal $\cE_\ell(M)$ is well defined by $\cE_\ell(A)$.
Observe that $\cE_j(M) \subset \cE_{j+1}(M)$ for $j \geq 0$.

Recall that a {\it greatest common divisor domain}, abbreviated GCD domain, is an integral domain $R$ such that any two non-zero elements have a greatest common divisor.
Any unique factorization domain, for example the Laurent polynomial ring in any number of variables over the integers or a field,  is a GCD domain.
A finitely generated ideal $I= (a_1, \ldots, a_m)$ in a GCD domain is contained in a unique smallest principal ideal, namely
the ideal generated by $\operatorname{gcd}(a_1, \ldots, a_m)$.
We write $\operatorname{gcd}(I)$ for $\operatorname{gcd}(a_1, \ldots, a_m)$ and note that $\operatorname{gcd}(I)$ is well defined up to
multiplication by a unit in $R$.
The {\it $i$-th elementary divisor} of a  finitely presented module $M$ over $R$  is  $\Delta^i(M) = \operatorname{gcd}(\cE_i(M))$.
Observe that $\Delta^{i+1}(M)$ divides  $\Delta^i(M)$ for $i\geq 0$.

The divisibility properties of elementary divisors in the following two propositions and their corollaries are useful.

\begin{proposition}
\label{divisible-one}
Let $R$ be a commutative ring with unit.  Let $n$ be a positive integer and
let $M$ be an $R$-module $M$ that has  a presentation matrix of the form $A^n - I$ where
$A$ is a square matrix and $I$ is the identity matrix of the same size.
Let $N$ be the $R$-module with presentation matrix $A - I$.
Then for $\ell \geq 0$,  $\cE_\ell(M) \subset  \cE_\ell(N)$.
Consequently, if $R$ is a GCD domain then for $\ell \geq 0$,
$\Delta^\ell(N)$ divides $\Delta^\ell(M)$.
\end{proposition}

\begin{proof}
Note that $A^n - I= (A -I) \left( \sum^{n-1}_{i=0} A^i\right)$.
By \cite[Theorem 1, Chapter 1]{Northcott} we have that
$\cE_\ell\left( A^n - I \right) \subset \cE_\ell\left( A-I \right) \cE_\ell\left(  \sum^{n-1}_{i=0} A^i\right)$.
Hence
\[
\cE_\ell(M)  =\cE_\ell(A^n - I)  \subset \cE_\ell\left( A-I \right) \cE_\ell\left(  \sum^{n-1}_{i=0} A^i\right)  \subset \cE_\ell\left( A-I \right) = \cE_\ell(N). \qedhere
\]
\end{proof}

For a virtual knot $K$, let $A_K$ denote its Alexander module (see section \ref{knot_and_alex}).
Let $\ell \geq 0$.
The $\ell$-th {\it Alexander ideal of $K$} is $\cE_\ell(A_K)$ and the 
{\it $\ell$-th Alexander polynomial of $K$} is $\Delta^\ell_K = \Delta^\ell(A_K)$.

\begin{corollary}
\label{divisible-three}
Let $n$ be a positive integer and
let $K$ be a $n$-periodic virtual knot diagram with quotient knot $K_*$.
Then for $\ell \geq 0$,  $\cE_\ell(A_K) \subset  \cE_\ell(A_{K_*})$.
Consequently,  for $\ell \geq 0$, 
$\Delta^\ell_{K_*}$  divides $\Delta^\ell_K$.
\end{corollary}
 
\begin{proof}
By Theorem \ref{mperiod}(ii),
$K$ can be realized as the closure of a periodic virtual braid,
$K=\widehat{\beta^{n}}$.
Hence the Alexander module of $K$ has a presentation matrix of the form
$A^{n} - I$, where $A$ is a square matrix and $I$ is the identity matrix of the same size;
furthermore, $A  - I$ is a presentation matrix for the Alexander module of $K_*$,  see Remark \ref{special_presentation}.
The conclusion of the Corollary follows from Proposition \ref{divisible-one}.
\end{proof}

\begin{proposition}
\label{divisible-two}
Let $R$ be a commutative ring with unit.  Assume that $R$ has prime characteristic $p>0$.
Let $M$ be an $R$-module $M$ that has  a presentation matrix of the form $A^{p^r} - I$ where
$A$ is a square matrix, $I$ is the identity matrix of the same size and $r \geq 1$.
Let $N$ be the $R$-module with presentation matrix $A - I$.
Then for $\ell \geq 0$,  $\cE_\ell(M) \subset  \cE_\ell(N)^{\, p^r}$.
Consequently, if $R$ is a GCD domain then for $\ell \geq 0$,
$\Delta^\ell(N)^{\, p^r}$ divides $\Delta^\ell(M)$.
\end{proposition}

\begin{proof}
Note that $(A-I)^{p^r} = \sum^{p^r}_{i=0} \, (-1)^i \binom{p^r}{i} A^{p^r -i} = A^{p^r} - I$
since $p$ divides the binomial coefficient $\binom{p^r}{i}$ for $0 < i < p^r$.
By \cite[Theorem 1, Chapter 1]{Northcott} we have that
$\cE_\ell\left( (A-I)^{p^r}\right) \subset \left( \cE_\ell\left( A-I \right) \right)^{p^r}$.
Hence
\[
\cE_\ell(M)  =\cE_\ell(A^{p^r} - I) = \cE_\ell\left( (A-I)^{p^r}\right)  \subset \left( \cE_\ell\left( A-I \right) \right)^{p^r} =\cE_\ell(N)^{\, p^r}. \qedhere
\]
\end{proof}

Note that $\F_p \otimes A_K$ is a module over 
$\F_p[t^{\pm1}]$ where $\F_p$ is the field of integers modulo a prime $p$.

\begin{corollary}
\label{divisible-four}
Let $p$ be a prime and
let $K$ be a $p^r$-periodic virtual knot diagram with quotient knot $K_*$.
Then for $\ell \geq 0$,  $\cE_\ell(\F_p \otimes A_K) \subset  \cE_\ell(\F_p \otimes A_{K_*})^{\, p^r}$.
Consequently,  for $\ell \geq 0$, 
$
\left( \Delta^\ell_{K_*}\right)^{p^r} \text{ \rm mod } p  \text{ divides  } \Delta^\ell_K \text{ \rm mod } p.
$
\end{corollary}

\begin{proof}
By Theorem \ref{mperiod}(ii),
$K$ can be realized as the closure of a periodic virtual braid,
$K=\widehat{\beta^{p^r}}$.
Hence the Alexander module of $K$ has a presentation matrix of the form
$A^{p^r} - I$ where $A$ is a square matrix and $I$ is the identity matrix of the same size;
furthermore, $A  - I$ is a presentation matrix for the Alexander module of $K_*$, see Remark \ref{special_presentation}.
The conclusion of the Corollary follows from Proposition \ref{divisible-two}.
\end{proof}

\medskip

The {\it virtual Alexander polynomial} of a virtual knot $K$ (\cite[Definition 3.1]{BDGGHN}), denoted $H_K(s,t,q)$,
is the $0$-th elementary divisor of the 
{\it virtual Alexander module of $K$} associated to the {\it virtual knot group}
$\operatorname{{\it VG}}_K$
of $K$ as in  \cite[\S 3]{BDGGHN}.
(Note that the knot group  $G_K$  (Definition \ref{knot_group}) is a
quotient of $\operatorname{{\it VG}}_K$, and the precise relationship between $G_K$  and $\operatorname{{\it VG}}_K$  as well as various other groups
associated to $K$ is explained in \cite{AC}.)
The integral Laurent polynomial $H_K(s,t,q)$ in the variables $s,t,q$ is related to Sawollek's {\it generalized Alexander polynomial} of $K$ (\cite{Sawollek}), denoted  $G_K(s,t)$,
via the formula $H_K(s, t, q) = G_K(sq^{-1}, tq)$  (up to multiplication by $\pm s^a t^b q^c$) as shown in \cite[Corollary 4.8]{BDGGHN}. 
By \cite[Theorem 3.4]{BDGGHN}, we have  
\begin{equation}
\label{virtual_crossing_reln}
q\mhyphen\fatness(H_K(s,t,q) )  \leq  2v(K),
\end{equation}
where $q\mhyphen\fatness(H_K(s,t,q) )$ is the span, also known as the {\it width}, of $H_K(s,t,q)$ as a Laurent polynomial in the variable $q$.
Moreover, by \cite[Proposition 4.10]{BDGGHN}, if $H_K(s,t,q)$ is nontrivial, then
$q\mhyphen\fatness(H_K(s,t,q) )\geq 2.$
Note that if $K$ is almost classical then $H_K(s, t, q) = 0$ by \cite[Corollary 5.4]{AC} and so does not yield information on such knots.

The {\it virtual Burau representation} (\cite[Definition 4.2]{BDGGHN}) is a homomorphism  $\Psi \colon  \VB_k \rightarrow GL_k(\ZZ[s^{\pm 1}, t^{\pm 1}, q^{\pm 1}])$.
If  $K$  is the closure of the virtual braid $\beta \in \VB_k$ then, 
by \cite[Theorem 4.4]{BDGGHN}, the matrix $\Psi(\beta) - I_k$, where $I_k$ is the $k \times k$ identity matrix,
is  a  presentation matrix for the virtual Alexander module of $K$.
Hence $H_K(s,t,q) = \det\left( \Psi(\beta) - I_k \right)$.

\begin{remark}
\label{special_presentation}
Let ${\bar \Psi} \colon  \VB_k \rightarrow GL_k(\ZZ[t^{\pm 1}])$ be the homomorphism obtained from $\Psi$
by evaluation at $s=1$ and $q=1$.
If the virtual knot $K$  is the closure of the virtual braid $\beta \in \VB_k$ then ${\bar \Psi}(\beta) - I_k$
is  a  presentation matrix for its Alexander module $A_K$. 
\end{remark}

A normalization of $H_K(s,t,q)$, denoted ${\widehat H}_K(s,t,q)$, for a virtual knot or link $K$ was defined in \cite[Definition 5.4]{BDGGHN} as follows.
Let $\beta \in \VB_k$ be a virtual braid whose closure is $K$. 
Then   ${\widehat H}_K(s,t,q) = (-1)^{\operatorname{writhe}(\beta) + v(\beta)}\det\left( \Psi(\beta) - I_k \right)$
where $v(\beta)$ is the virtual crossing number of the closure of $\beta$.
The invariant  ${\widehat H}_K(s,t,q)$  is  defined up to powers of $st$ and in particular,  the lowest and highest exponents of $q$ occurring
in ${\widehat H}_K(s,t,q)$ are well defined.
This can be used to give a stronger version of \eqref{virtual_crossing_reln},  see  \cite[Theorem 5.6]{BDGGHN}.

Our  approach yields a small enhancement of \cite[Theorem 3.1]{Kim-Lee-Seo-2014}, as follows.

\begin{proposition} \label{prop:KLS-2014}
Let $p$ be a prime and 
let $K$ be a $p^r$-periodic virtual knot diagram with quotient $K_*$. Then
${\widehat H}_K(s,t,q)  = \left[{\widehat H}_{K_*}(s,t,q)\right]^{p^r} \text{ mod $p$}$,  up to multiplication by a power of $st$.
\end{proposition}

\begin{proof}
By Theorem \ref{mperiod}(ii), $K$ can be realized as the closure of a periodic virtual braid, $K=\widehat{\beta^{p^r}}$.
Note that $K_*=\widehat{\beta}$. 
For a virtual braid $\eta$,  let $\mu(\eta)  = \operatorname{writhe}(\eta) + v(\eta)$.
Note that $(-1)^{\mu\left(\beta^{p^r}\right)}  = (-1)^{p^r \mu(\beta)}$.
Also,
\[
{\widehat H}_K(s,t,q) =  (-1)^{\mu\left(\beta^{p^r}\right)} \det\left( \Psi\left(\beta^{p^r}\right) - I  \right) =  (-1)^{\mu\left(\beta^{p^r}\right)} \det\left( \Psi(\beta)^{\, p^r} - I\right).
\]
As in the proof of Proposition \ref{divisible-two},   $\left( \Psi(\beta)- I \right)^{p^r} =   \Psi(\beta)^{\, p^r} - I \text{ mod $p$}$ and so
\begin{align*}
		{\widehat H}_K(s,t,q) &= (-1)^{\mu\left(\beta^{p^r}\right)}  \det \left( \left( \Psi(\beta)- I \right)^{p^r} \right) \text{ mod $p$} \\
                                  &= (-1)^{p^r \mu(\beta)} \det  \left( \Psi(\beta)- I  \right)^{\, p^r} \text{ mod $p$}  \\
                                  &= \left[{\widehat H}_{K_*}(s,t,q)\right]^{p^r} \text{ mod $p$.} \qedhere
\end{align*}
\end{proof}

We use Proposition \ref{prop:KLS-2014} to give a bound on the prime power periods of $K$ in terms of its virtual crossing number $v(K)$. 
This applies provided $H_K(s,t,q)$ is  non-zero mod $p$ in which case
$H_{K_*}(s,t,q)$ is also non-zero mod $p$.
By \cite[Proposition 4.10]{BDGGHN} we have that $(1-tq) (q-s)$ divides $H_{K_*}(s,t,q)$.
This holds over $\F_p$ as well, and since $H_{K_*}(s,t,q)$ is non-zero modulo $p$, we conclude that $\left((1-tq) (q-s)\right)^{p^r}$ divides $H_K(s,t,q)$.  Hence
\begin{equation}
\label{virtual_crossing_inequality}
q\mhyphen\fatness(H_K(s,t,q) )  \geq  q\mhyphen\fatness_p(H_K(s,t,q) )   \geq  {p^r}  \left(q\mhyphen\fatness_p(H_{K_*}(s,t,q) ) \right) \geq 2{p^r}.
\end{equation}
Here,  $q\mhyphen\fatness_p(H_K(s,t,q))$ is the span in the variable $q$ of $H_K(s,t,q)$ reduced modulo $p$, 
see section \ref{sec:periods} for a discussion of $\fatness$ and $\fatness_p$.

Combining \eqref{virtual_crossing_inequality} with \eqref{virtual_crossing_reln} yields the following result.

\begin{proposition}
\label{vcrossing_bound}
If $K$ is a ${p^r}$-periodic virtual knot such that $H_K(s,t,q)$ is  non-zero modulo $p$ 
then $p^r \leq v(K)$, where $v(K)$ is the virtual crossing number of $K$.\qed
\end{proposition}

We obtain the following upper bound for the periods of $K$.

\begin{corollary}
\label{strong_vcrossing_bound}
Assume $K$ is a virtual knot with $H_K(s,t,q)$ non-zero modulo $p$ for all primes $p$.
Then  $e^{v(K)^{1.3841}}$ is an upper bounded for a period of $K$.
\end{corollary}
\begin{proof}
For an integer $n$, 
let $\omega(n)$ be the number distinct primes dividing $n$.
Assume $n$ is a period of $K$ and write
$n = \prod^{\omega(n)}_{j=1}  p_j^{r_j}$ where the $p_j$'s are distinct primes.
By Proposition \ref{vcrossing_bound},  $p^{r_j} \leq v(K)$ for each $j$.
Hence $n \leq v(K)^{\omega(n)}$ and so $\ln(n) \leq {\omega(n)} \ln(v(K))$.
Robin showed that $\omega(n) \leq 1.3841 \ln(n)/ \ln(\ln(n))$ for $n \geq 3$,  \cite[Theorem 11]{robin}.
Hence $\ln(\ln(n)) \leq 1.3841 \ln(v(K))$ from which the conclusion follows.
\end{proof}

For example, the hypothesis of Corollary \ref{strong_vcrossing_bound} is satisfied whenever $\pm 1$ appears a coefficient of $H_K(s,t,q).$ 
However, this hypothesis obviously fails when $H_K(s,t,q)$ vanishes as in the case when $K$ is an almost classical knot. 
In addition, there are virtual knots with $H_K(s,t,q)\neq 0$ where the hypothesis of Corollary \ref{strong_vcrossing_bound} fail, for instance the knot $K= 4.43$ (from the table of virtual knots in\cite{Green}) has $$H_K(s,t,q) = 2 s^3 t^3 - 2 s t + 2 q s t^2 + 2 s^2 t q^{-1} - 2 q s^2 t^3 - 2 s^3 t^2 q^{-1},$$ which reduces to zero modulo 2.


\section{Circulant matrices in positive characteristic}
\label{sec:circulant}

In this section we study circulant matrices since they arise naturally in our approach to the computation of the Alexander invariants of periodic knots.
While circulant matrices over field of characteristic $0$ have been extensively examined in the literature,
we are mainly interested in circulant matrices over a field of characteristic $p > 0$.
Our two key results are Theorems~\ref{xinvpx} and~\ref{xinbx_upper}, which may be also of independent interest to algebraists.

Let $R$ be a commutative ring with unit and let $M$ be a left $R$-module.
Since $R$ is assumed to be commutative, $M$ is also a $R$-$R$-bimodule where the right $R$-action is given
by $x r = rx$ for $r \in R$ and $x \in M$.
For a positive integer $n$, 
let $\operatorname{Mat}(n,R)$ denote the $R$-algebra of $n \times n$ matrices over $R$ and $\operatorname{Mat}(n,M)$
the $R$-$R$-bimodule of $n \times n$ matrices over $M$.
The $R$-$R$-bimodule structure on $M$ induces a
$\operatorname{Mat}(n,R)$-$\operatorname{Mat}(n,R)$-bimodule structure on $\operatorname{Mat}(n,M)$ via
left and right matrix multiplication.

We will be primarily interested in the case $R = \F_p$, the field of integers modulo a prime $p$,
and $M = \operatorname{Mat}(\ell,S)$  over an $\F_p$-algebra $S$.
In that case, $n \times n$ circulant matrices over $M$ (see Definition \ref{def:circulant} below) are also called \emph{block circulant matrices},
as the elements of $M$ are themselves matrices.
Forgetting the block structure yields an $\ell n \times \ell n$ matrix over $S$.

\begin{definition}
\label{def:circulant}
Let $R$ be a commutative ring with unit and let $M$ be an $R$-module.
For elements $A_0, \ldots , A_{n-1} \in M$ the corresponding {\it circulant matrix} is the $n \times n$ matrix over $M$ given by
\[
C(A_0, \ldots, A_{n-1})=\left[
\begin{array}{ccccc}
A_0 & A_1  & A_2 & \cdots & A_{n-1} \\
A_{n-1}  & A_0 & A_1 & \cdots & A_{n-2} \\
A_{n-2}  & A_{n-1} & A_0 & \cdots & A_{n-3} \\ 
\vdots & \vdots & \vdots & \ddots & \vdots \\ 
A_1 & A_2 & A_3 & \cdots & A_0
\end{array}\right].
\]
\end{definition}

Let $P$ be the $n \times n$ matrix over $R$ given by
\begin{equation*}
P = C(0, 1, 0, \ldots, 0)=
\left[
\begin{array}{cccc}
 0 &1 & \cdots& 0 \\
  \vdots&  \ddots& \ddots& \vdots\\
   0& \cdots  & & 1\\
 1 & 0 & \cdots  & 0 
\end{array} 
 \right],
\end{equation*}
where $1 \in R$ is the unit element for $R$.

Observe that for $0 \leq j \leq n-1$, we have $P^j = C(0,\ldots, 0, 1, 0, \ldots, 0),$ where $1$ appears in the $(j+1)$-st slot.
Hence the circulant matrix $C(A_0, \ldots, A_{n-1})$ over $M$ can be written as 
\begin{equation}
\label{circformula}
C(A_0, \ldots, A_{n-1}) = \sum^{n-1}_{i=0}  A_i P^i.
\end{equation}
(Note that $P^0 = I$, the $n \times n$ identity matrix.)

\medskip

\noindent
{\it Conventions for binomial coefficients.}   Let $a, b$ be non-negative integers.  We define
\[
\binom{a}{b} = \begin{cases}
\frac{a!}{b! \, (a -b)!} & \text{if $a \geq b$,} \\
 0  & \text{if $a < b$}.
\end{cases}
\]
It is also convenient to define
\[
\binom{-1}{b} =(-1)^b.
\]

The following identities involving binomial coefficients modulo a prime $p$ will be useful.

\smallskip

\begin{lemma}
\label{binomlemma}
Let $p$ be a prime and $r$ a positive integer.   
\begin{enumerate}
\item  If $0 \leq \ell \leq p^r-1$, then
\[
\binom{p^r -1}{\ell} = (-1)^\ell \mod p.
\]

\item If $1 \leq \ell \leq p^r -1$, then
\[
\binom{p^r + \ell -1}{p^r-1} = 0 \mod p.
\]
\end{enumerate}
\end{lemma}
\begin{proof}
Proof of (1).
The standard binomial coefficient identity $\binom{p^r-1}{\ell} + \binom{p^r-1}{\ell-1} = \binom{p^r}{\ell}$
gives $\binom{p^r-1}{\ell} = - \binom{p^r-1}{\ell-1} \mod p$, valid for $0 < \ell \leq p^r-1$,
because
$\binom{p^r}{\ell}$ is divisible by $p$ for $0 < \ell < p^r-1$
and $\binom{p^r-1}{(p^r-1)-1}  = p^r - 1 = -1 \mod p$.
The formula $\binom{p^r -1}{\ell} = (-1)^\ell \mod p$ now easily follows by induction on $\ell$.

Proof of (2).
We apply Lucas's Theorem
\cite[Theorem 1]{Fine}
which asserts that if $M$ an $N$ are non-negative integers
written in base $p$ as
$M = \sum^k_{i=0} M_i p^i$ and $N = \sum^k_{i=0} N_i p^i$,
with $0 \leq M_i,  N_i < p$, then
$
\binom{M}{N} = \prod^k_{i=0} \binom{M_i}{N_i} \mod p
$.
Since $1 \leq \ell \leq p^r -1$, 
we have $p^r + \ell -1 = \sum^r_{i=0} M_i p^i$
where $M_r =1$ and at least one of the numbers $M_i$ for $0 \leq i \leq  r-1$ is strictly less than $p-1$.
Also, $p^r -1 = \sum^{r-1}_{i=0} (p-1) p^i$.
Hence at least one the numbers $\binom{M_i}{p-1}$ is $0$ and so by Lucas's Theorem,
$\binom{p^r + \ell -1}{p^r-1} = 0 \mod p$.
\end{proof}

\smallskip

We will show that the matrix $P = C(0,1,0,\ldots,0)$  over $\F_p$ of size $p^r \times p^r$ is
conjugate, via an explicitly given matrix over $\F_p$,  to an elementary Jordan matrix.
Before proving that, we establish a useful lemma.

\begin{lemma}
\label{lem-xinvpx}
Let $X$ be the matrix  of size $p^r \times p^r$ over $\F_p$ with
\[
X_{i,j} = {\binom{i-2}{j-1}}, \ 1 \leq i,j \leq p^r.
\]
Then $X$ is invertible and 
 \[
(X^{-1})_{i,j} = { \binom{p^r-j+1}{p^r-i}}, \ 1 \leq i,j \leq p^r.
\]
\end{lemma}

\begin{proof}
Let $Y$ be the matrix over $\F_p$ given  by
$
Y_{i,j} = { \binom{p^r-j+1}{p^r-i}}, \ 1 \leq i,j \leq p^r.
$
We will show that $XY =I$, from which the lemma will follow.  

\[
(XY)_{i,j} = \sum_{k=1}^{p^r} X_{i,k}Y_{k,j} = \sum_{k=1}^{p^r} { \binom{i-2}{k-1} } { \binom{p^r-j+1}{p^r-k} } = \sum_{\ell=0}^{p^r-1} { \binom{i-2}{\ell} } { \binom{p^r-j+1}{p^r-1 - \ell} }.
\]

The well-known {\it Vandermonde Convolution formula} asserts that for non-negative integers $m, n, q$,
\[
\sum^q_{\ell=0} \binom{m}{\ell}  \binom{n}{q-\ell}  =   \binom{m+n}{q}.
\]
This formula is also valid for $m = -1$ with our convention  $\binom{-1}{\ell} = (-1)^\ell$.
Applying Vandermonde Convolution to the above expression for $(XY)_{i,j}$ (with $m=i-2,$ $q=p^r-1,$ and $n=p^r-j+1$) yields
\[
(XY)_{i,j} = \binom{p^r+ i-j-1}{p^r-1}.
\]
We need that $XY = I$, so we need to show that $(XY)_{i,i} = 1$ and $(XY)_{i,j} = 0$ for $i \neq j$.
If $i < j$ then $p^r + i -j -1 < p^r -1$ and so $(XY)_{i,j} = \binom{p^r+ i-j-1}{p^r-1} =0$ in this case.
We have $(XY)_{i,i} = \binom{p^r-1}{p^r-1} =1$.
If $i > j$ then Lemma~\ref{binomlemma}(2) applies (as $1 \le i-j \le p^r$), and we have $(XY)_{i,j} = \binom{p^r+ i-j-1}{p^r-1} =0 \mod p$ in this case.
Hence $XY = I$.
\end{proof}

\begin{theorem}
\label{xinvpx}
Let $P = C(0,1,0,\ldots,0)$ be the circulant matrix  of size $p^r \times p^r$ over $\F_p$ and let $X$ be the matrix  of size $p^r \times p^r$ over $\F_p$ with
\[
X_{i,j} = {\binom{i-2}{j-1}}, \ 1 \leq i,j \leq p^r
\]
and
 \[
(X^{-1})_{i,j} = { \binom{p^r-j+1}{p^r-i}}, \ 1 \leq i,j \leq p^r.
\] Then
$$
X^{-1}PX = \left[
\begin{array}{c c c c}
1 & 1 &  \cdots & 0 \\ 
0 & \ddots & \ddots  &  \vdots \\ 
\vdots & \ddots & \ddots & 1 \\ 
0 & \cdots &   0 & 1
\end{array}\right] = J.
$$
\end{theorem}
\begin{proof}
By Lemma~\ref{lem-xinvpx}, we know that $X$ and $X^{-1}$ are indeed inverses, and the theorem follows once we verify that $PX = XJ$.

Since $P_{k,k+1} =1$ for $k = 1, \ldots p^r-1$ and $P_{p^r,1} =1$ and $P_{i,j} =0$ otherwise, 
\[
(PX)_{i,j} = \sum^{p^r}_{k=1} P_{i,k} X_{k,j} = 
\begin{cases}
X_{i+1,j} = \binom{i-1}{j-1}  & \text{if  $i \neq p^r$,}\\
X_{1,j} = \binom{-1}{j-1}  = (-1)^{j-1} & \text{if  $i = p^r$.}
\end{cases}
\]
Note that this is because if $i=p^r$, $P_{p^r,1} =1$ and $P_{i,j} =0$ otherwise, so our sum becomes $1\cdot X_{1,j} = X_{1,j}$. If $i \neq p^r$, then $P_{i,i+1} =1$ and $P_{i,j} =0$ otherwise, so our sum becomes $X_{i+1,j}$.
Since $J_{k,k+1} =1$ for $k= 1, \ldots p^r-1$ and $J_{k,k} =1$ for $k = 1, \ldots p^r$ and $J_{i,j} =0$ otherwise,
\[
(XJ)_{i,j} = \sum^{p^r}_{k=1} X_{i,k} J_{k,j} = 
\begin{cases}
X_{i,j-1} + X_{i,j} = \binom{i-2}{j-2} + \binom{i-2}{j-1}  = \binom{i-1}{j-1}& \text{if  $j \neq 1$,}\\
X_{i,1} = \binom{i-2}{0} = 1 & \text{if  $j=1$.}
\end{cases}
\]

If $j=1$, $J_{k,1} = 1$ for $k=1$ and $J_{k,1} = 0$ otherwise, so we have the sum equalling $X_{i,1}J_{1,1} = X_{i,1}$. If $j \neq 1$, $J_{k,j} = 1$ for $k\in \{j-1,j\}$ and $J_{k,j} = 0$ otherwise, so we have the sum equalling $X_{i,j}+ X_{i,j-1}$.

It follows immediately that $(PX)_{i,j} = (XJ)_{i,j}$ for $i \neq p^r$, $j \neq 1$.   
For the $j=1$ case, $(PX)_{i,1} = \binom{i-1}{0} = 1$. For the $i=p^r$ case, we have
\[
(XJ)_{p^r, \,j} =  \binom{p^r -1}{j-1}  = (-1)^{j-1} \mod p.
\]
with the last equality coming from Lemma~\ref{binomlemma}(1).
Since $(PX)_{p^r, \,j} = (-1)^{j-1}$, we obtain $(PX)_{p^r, \,j} =(XJ)_{p^r, \,j}  \mod p$,
completing the proof that  $PX = XJ$.
\end{proof}

We remark that
numbers of the form ${\binom{p^r+m-1}{p^r -1}}$,  $1 \leq m \leq p^r -1$, while divisible by $p$, need not be divisible by higher powers of $p$.
For example, if $p=3$, $r=2$ and $m = 6$ then
\[
\binom{p^r+m-1}{p^r -1} = \binom{14}{8} = 3 \cdot 7 \cdot 11 \cdot 13
\]
which is divisible by 3 but not by $3^2$.
This observation obstructs a version of Theorem~\ref{xinvpx} where $\F_p$ would conceivably be replaced by the ring of integers modulo $p^r$.

\begin{theorem}
\label{xinbx_upper}
Let $M$ be vector space over  $\F_p$ and let $B=C(A_0, \ldots, A_{p^r-1})$ be a $p^r \times p^r$ circulant matrix over $M$.
Then $X^{-1} B X  = \sum^{p^r-1}_{i=0} A_i J^i$, where 
$X$ is the matrix  of size $p^r \times p^r$ over $\F_p$ with
\[
X_{i,j} = {\binom{i-2}{j-1}}, \ 1 \leq i,j \leq p^r,
\]
 \[
(X^{-1})_{i,j} = { \binom{p^r-j+1}{p^r-i}}, \ 1 \leq i,j \leq p^r,
\]
and
$$
J = X^{-1}PX = \left[
\begin{array}{c c c c}
1 & 1 &  \cdots & 0 \\ 
0 & \ddots & \ddots  &  \vdots \\ 
\vdots & \ddots & \ddots & 1 \\ 
0 & \cdots &   0 & 1
\end{array}\right].
$$
In particular, $X^{-1} B X$ is upper triangular as a matrix over $M$.
\end{theorem}

\begin{proof}
By \eqref{circformula},  $B  = \sum^{p^r-1}_{i=0} A_i P^i$ where $P = C(0,1,0,\ldots,0)$.
Hence
	\begin{align*}
		X^{-1} B X  &= X^{-1} \left( \sum^{p^r-1}_{i=0} A_i P^i  \right) X\\
		& =  \sum^{p^r-1}_{i=0}  X^{-1} A_i P^i  X   =  \sum^{p^r-1}_{i=0} A_i ( X^{-1} P  X)^i  \\
                 &  = \sum^{p^r-1}_{i=0} A_i J^i \qquad \text{by Theorem~\ref{xinvpx}.}
	\end{align*}
Since each $A_i J^i$ is upper triangular, so is $ \sum^{p^r-1}_{i=0} A_i J^i$.
\end{proof}

\begin{corollary}\label{xinbx}
Let $B$ and $X$ be as Theorem~\ref{xinbx_upper}.
Then
\[
(X^{-1}BX)_{j, \, j + \ell} =  \sum_{i=\ell}^{p^r-1} \binom{i}{\ell} A_i.
\]
\end{corollary}

\begin{proof}
Let $Q$ be the $p^r \times p^r$ matrix over $\F_p$ given by
$$
Q = \left[
\begin{array}{c c c c}
0 & 1 &  \cdots & 0 \\ 
& \ddots & \ddots  &  \vdots \\ 
\vdots &  & \ddots & 1 \\ 
0 & \cdots &   & 0
\end{array}\right].
$$

Then $J = I +Q$ and so
$J^i = (I+Q)^i = \sum_{\ell=0}^{i}{\binom{i}{\ell}} Q^\ell$.  By Theorem~\ref{xinbx_upper},
\[ 
X^{-1} B X  = \sum^{p^r-1}_{i=0} A_i J^i
                  = \sum^{p^r-1}_{i=0} A_i \left( \sum_{\ell=0}^{i}{\binom{i}{\ell}} Q^\ell \right)
                  = \sum^{p^r-1}_{\ell=0} \left( \sum_{i=\ell}^{p^r-1}{\binom{i}{\ell}} A_i \right) Q^\ell 
\]
Then $(X^{-1} B X)_{j,j+\ell} = \sum^{p^r-1}_{\ell=0} \left( \sum_{i=\ell}^{p^r-1}{\binom{i}{\ell}} A_i \right) Q^\ell $. Note that $Q_{j,j+1} = 1$ and in general, ${Q^{\ell}}_{j,j+\ell} = 1$ and ${Q^k}_{j,j+\ell} = 0$, for $k \neq \ell$. So we then get $(X^{-1} B X)_{j,j+\ell} = \sum_{i=\ell}^{p^r-1} \binom{i}{\ell} A_i$.
\end{proof}
The next result is obtained by evaluating the formula in Corollary~\ref{xinbx} for $\ell = 0,1.$
\begin{corollary}\label{xinbx-2}
For $B$ and $X$ as Theorem~\ref{xinbx_upper}, we have
$$
X^{-1}BX =  \begin{bmatrix}
A & D &    \cdots& * \\
 & \ddots & \ddots&  \vdots \\
\vdots & & \ddots &  D \\
0 & \cdots &  &   A
\end{bmatrix},
$$
where $A =\sum_{k=0}^{p^r-1} A_k$ and $D = \sum_{k=1}^{p^r-1} kA_k$.
\end{corollary}


\section{Murasugi's Theorem for Almost Classical Knots}
\label{sec:virtualMurasugi}

In this section, we prove Theorem~\ref{virtmur-0} from the Introduction, which is the analogue of Murasugi's Theorem~\ref{classicalmur} for periodic almost classical knots. We begin by restating the result. First, recall that for a $q$-periodic almost classical knot $K$, Corollary~\ref{ACbraidrep} allows us to write $K = \widehat{{\beta}^{q}}$ for some $k$-strand virtual braid $\beta$ that admits an Alexander numbering.

\begin{theorem}\label{virtmur-1}
Let $K = \widehat{{\beta}^{q}}$ be a $q$-periodic almost classical knot diagram, where $\beta$ a $k$-strand virtual braid
that admits an Alexander numbering, and $q=p^r$ a prime power.
Then $K_* = \widehat{\beta}$, and
\begin{enumerate}
\item $\Delta_{K_*}(t)$ divides $\Delta_{K}(t)$ in $\ZZ[t^{\pm 1}],$ and
\item $\Delta_{K}(t) ~\doteq~ \left(\Delta_{K_*}(t)\right)^{q} \left(f(t) \right)^{q - 1} \mod p,$ 
where $f(t) = {\textstyle \sum_{i=1}^{k}} t^{\lambda_i}$ and $\lambda_i$ is the Alexander number on the $i$-th strand of $\beta$.
\end{enumerate}
\end{theorem}

\begin{proof}

Part 1 follows from Corollary \ref{divisible-three}.
We divide the proof of part 2 into several claims.

\begin{claim}\label{claim-1}
${\displaystyle \Delta_K(t) \doteq (\Delta_{K_{*}}(t))^{q}(f(t))^{q - 1} \mod p}$ for some $f(t) \in \F_p[t^{\pm 1}]$. 
\end{claim}

The proof of this claim requires extensive matrix manipulation which we now present.
Let $B$ be the block circulant Jacobian matrix for $K$ constructed in Definition~\ref{method-1} and written out as in Equation \eqref{eq:cm}. We will assume it has $n \times n$ blocks, or equivalently that there are $n$ crossings in each period of $K$. Then Theorem~\ref{thm:jaco} and Corollary~\ref{xinbx-2}  
show that 

$$
B \cong 
X^{-1}BX =  
\begin{bmatrix}
A & D &    \cdots& * \\
 & \ddots & \ddots&  \vdots \\
\vdots & & \ddots &  D \\
0 & \cdots &  &   A
\end{bmatrix} \mod p,
$$
where $A=\sum_{k=0}^{q-1} A_k$ is the Jacobian of $K_*$ and $D = \sum_{k=1}^{q-1} kA_k$.

Set
$$C =\begin{bmatrix} 
A & D &    \cdots& * \\
 & \ddots & \ddots&  \vdots \\
\vdots & & \ddots &  D \\
0 & \cdots &  &   A
\end{bmatrix},$$
which is an $nq \times nq$ matrix written in block form with $n\times n$ blocks. 

Let $\bar{A}$ and $\bar{C}$ denote the matrices  $A$ and $C$ with their last row and first column removed. So $\bar{A}$ is an $(n-1)\times (n-1)$ matrix and $\bar{C}$ is an $(nq-1) \times (nq-1)$ matrix. Then modulo $p$,
the  Alexander polynomials of $K$ and $K_*$ are given by 
$$\Delta_K(t) \doteq \det (\bar{C}) \mod p \quad \text{ and } \quad  \Delta_{K_*}(t)  \doteq \det (\bar{A}) \mod p.$$

Let ${A'}$  be the $(n-1) \times n$ matrix obtained by removing the last row from $A$, and let ${A''}$ be the $n \times (n-1)$ matrix obtained by removing the first column of $A$. Also let $0_n$ be the $n \times n$ matrix of zeroes, $0_n'$ the $(n-1) \times n$ matrix of zeroes, and $0_n''$ the $n \times (n-1)$ matrix of zeroes. 
Note that $\bar{0}_n = 0_{n-1}$.

Using these to rewrite $\bar{C}$, we get

$$
\bar{C} = \left[
\begin{array}{c c c c c c}
A'' & D & * & \hdots & \hdots & * \\
0_n'' & A & D & \ddots & & \vdots  \\
\vdots & {0_n} & \ddots & \ddots & \ddots & \vdots    \\
\vdots &   & \ddots  & \ddots & \ddots & *   \\
\vdots &  &  & {0_n} & A & D \\
\bar{0}_n & \hdots & \hdots & 0_n' & 0_n' & A'  \\
\end{array}\right].
$$

Notice that  $\bar{A}$ is a submatrix of each of the $q$  block terms $A$, $A'$, and $A''$ appearing on the block diagonal of $\bar{C},$ and our goal is to extract those terms using row and column operations on $\bar{C}$ to reduce it to an upper block triangular matrix with $\bar{A}$ blocks on the diagonal. 

For that, we require some additional notation.
Let $r_A=[A_{n,2}, \ldots, A_{n,n}]$ be the last row of $A$ minus the first entry, let $c_{A}=[A_{1,1}, \ldots, A_{n-1,1}]^t$ be the first column of $A$ minus the last entry, and let $u_A=A_{1,n}$ be the bottom left corner entry for $A$.  We can now rewrite $A, A', A''$ in terms of $c_A, r_A,$ and $u_A$ as

$$
A = \left[
\begin{array}{c | c c c}
 & & & \\
c_{A} & & \bar{A} & \\
 & & & \\ \hline \\
u_A & & r_A & \\
\end{array}\right],
\quad
A' = \left[
\begin{array}{c | c c c}
 & & & \\
c_{A} & & \bar{A} & \\
 & & & \\ 
\end{array}\right],
\quad \text{ and } \quad
A'' = \left[
\begin{array}{c c c}
 & &  \\
 & \bar{A} & \\
  & & \\ \hline \\
  & r_A & \\
\end{array}\right].
$$

Further, let $r_D, c_D,$ and $u_D$ be the corresponding row, column, and bottom left corner entry of $D$, which we use to write
$$
D = \left[
\begin{array}{c | c c c}
 & & & \\
c_{D} & & \bar{D} & \\
 & & & \\ \hline
u_D & & r_D & \\
\end{array}\right],
$$
where $\bar{D}$ is the $(n-1) \times (n-1)$ matrix  $D$ with its last row and first column removed.

Using these matrices, we can rewrite $\bar{C}$ as:

$$
\bar{C} = \left[
\begin{array}{c c c c c c c c}
\bar{A} & c_D & \bar{D} & * & * & \hdots & \hdots & * \\
r_A & u_D & r_D & * & * & \hdots & \hdots & * \\
0_{n-1} & c_{A} & \bar{A} & c_D & \bar{D} & \ddots & & \vdots \\
0_{1 \times (n-1)} & u_A & r_A & u_D & r_D & \ddots & \ddots & \vdots \\
\vdots & \ddots &  &  & \ddots &  & \ddots & \vdots \\
\vdots & & \ddots &  & c_{A} & \bar{A} & c_D & \bar{D} \\
\vdots & &  & \ddots & u_A & r_A & u_D & r_D \\
0_{n-1} & \hdots & \hdots & \hdots & 0_{(n-1) \times 1} \!\! & 0_{n-1} & c_{A} & \bar{A} 
\end{array}\right].
$$

Since $A$ is the Jacobian of the quotient knot $K_*$, we know that the sum of its columns equals zero, in other words, $\sum_{j=1}^n A_{*,j} =0$. (Here, $A_{*,j}$ denotes the $j$-th column of $A$.)  Further, since $K_*$ is almost classical, Proposition~\ref{linearcombo} shows that there is a linearly dependence among the rows. More specifically, we have units $\vartheta_1,\ldots, \vartheta_n \in \ZZ[t^{\pm 1}]$ such that $\sum_{i=1}^n \vartheta_i A_{i,*} =0$. (Here, $A_{i,*}$ denotes the $i$-th row of $A$.)
Thus, replacing the first column in $A$ by the sum of all its columns, and replacing the last row by the linear combination $\sum_{i=1}^n \vartheta_i A_{i,*}$, we obtain the matrix
$$\widetilde{A} = \left[
\begin{array}{c | c c c}
 & & & \\
0'' & & \bar{A} & \\
 & & & \\ \hline \\
0 & & 0' & \\
\end{array}\right],$$
where $0'=0_{1 \times (n-1)}$ is a row of zeros and $0'' = 0_{(n-1) \times 1}$ is a column of zeros.

Performing the same row and column operations to $D$ gives the matrix
$$
\widetilde{D} = \left[
\begin{array}{c | c c c}
 & & & \\
\widetilde{c}_{D} & & \bar{D} & \\
 & & & \\ \hline \\
\widetilde{u}_D & & \widetilde{r}_D & \\
\end{array}\right],$$
where $\widetilde{c}_D$ is the sum of the columns in $D$ minus the last entry,
$\widetilde{r}_D = \sum_{i=1}^n \vartheta_i D_{i,*}$ is the linear combination of the rows in $D$ minus the first entry, and
$$\widetilde{u}_D =\sum_{i,j=1}^n \vartheta_i D_{i,j}.$$ 
(This is proved in Claim~\ref{claim-2} below.) 
The result of performing these operations on each of the blocks of $\bar{C}$ gives the matrix

$$
\widetilde{C} = \left[
\begin{array}{c c c c c c c c c c}
\tikzmark{left}\bar{A} & \widetilde{c}_D & \bar{D} &  &  &  &  &  &  &  \\
{0'} & \widetilde{u}_D \tikzmark{right} & \widetilde{r}_D &  & &  &    &   & \text{\huge{*}} &  \\ 
\DrawBox[thick]
 & 0'' & \tikzmark{left}\bar{A} & \widetilde{c}_D & \bar{D} &  &  &  &  &  \\
 & 0 & 0' & \widetilde{u}_D\tikzmark{right} & \widetilde{r}_D & \ddots &  &  &  &  \\
\DrawBox[thick]
 &  &  &  &  \ddots&    & \ddots &  &  &  \\
 & & & & & \ddots&  & \ddots & & \\
 & & & & & & \ddots&   & \ddots & \\
 &  &  &  &  &  &  0'' & \tikzmark{left}\bar{A} & \widetilde{c}_D & \bar{D} \\
 & \text{\huge0} &  &  &  &  & 0 &0'  & \widetilde{u}_D\tikzmark{right} & \widetilde{r}_D \\
\DrawBox[thick]
  &  &  &  &  &  &  &  & 0'' & \tikzmark{left}\bar{A}\tikzmark{right}  
\end{array}\right],
\DrawBox[thick]
$$
where $0'=0_{1 \times (n-1)}$ and $0'' = 0_{(n-1) \times 1}$.


Expanding along the block diagonal of $\widetilde{C}$, we compute that:
\begin{eqnarray*}
\det (\widetilde{C}) &=& \det(\bar{A}) \left(\det 
\begin{bmatrix} \bar{A} & \tilde{c}_D \\
0 & \tilde{u}_D
\end{bmatrix}\right)^{q-1} \mod p \\
&\doteq& \det (\bar{A})(\tilde{u}_D \det (\bar{A}))^{q-1} \mod p \\
&\doteq& (\det (\bar{A}))^{q}(\tilde{u}_D)^{q-1} \mod p.
\end{eqnarray*}
Since $\Delta_K(t) \doteq \det (\bar{C}) \mod p$ and $\Delta_{K_*}(t) = \det (\bar{A})$, the equations above imply that 
$$\Delta_K(t) \doteq (\Delta_{K_{*}}(t))^{q}(f(t))^{q - 1} \mod p,$$
provided we take $f(t) = \widetilde{u}_D$. This completes the proof of Claim~\ref{claim-1}. 

The last step in proving  Theorem~\ref{virtmur-1} is to show that
$$f(t) = \sum_{i=1}^{k}t^{\lambda_i},$$
where $\lambda_i$ is the Alexander number on the $i$-th strand at the top of $\beta$. 

Before doing that, we shall prove the following
\begin{claim}\label{claim-2}
${\displaystyle \widetilde{u}_D = \sum_{i=1}^n\sum_{j=1}^n \vartheta_i D_{i,j}}$
\end{claim}

The element $\widetilde{u}_D$ is the bottom left corner entry of $\widetilde{D}$, the matrix obtained by performing row and column operations to $D$. Specifically, $\widetilde{D}$ is obtained in two steps. The first step is to replace the first column of $D$ by the sum $\sum_{j=1}^n D_{*,j}$ of all its columns. Here, $D_{*,j}$ denotes the $j$-th column of $D$. 

Let $D'$ be the matrix obtained after the first step. The second step is to replace the last row of $D'$ by the linear combination $\sum_{i=1}^n \vartheta_i D'_{i,*}$. Here, $D'_{i,*}$ denotes the $i$-th row of $D'$, and $\vartheta_i$ is the unit in $\ZZ[t^{\pm 1}]$ whose existence is guaranteed by Proposition~\ref{linearcombo}.  Therefore, the bottom left entry in the resulting matrix $\widetilde{D}$ is given by 
$$\widetilde{u}_D =  \sum_{i=1}^n \vartheta_i D'_{i,1} = \sum_{i=1}^n \sum_{j=1}^n \vartheta_i D_{i,j}.$$
Note that $ D'_{i,1} = \sum_{j=1}^n D_{i,j}$ since $D'_{i,j} = D_{i,j}$ for $j \neq 1$.
This completes the proof of Claim~\ref{claim-2}.

We are now ready to complete the proof of Theorem~\ref{virtmur-1}. To that end, suppose now that $B$ is the block circulant Jacobian matrix obtained by applying Definition~\ref{method-2} to the braid $\beta^q \in \VB_k$. It follows that 
$$B=
\begin{bmatrix}
A_0 & A_1 & 0 & \cdots & 0 \\ 
0 & A_0 & A_1& \cdots & 0 \\ 
\vdots & \ddots & \ddots & \vdots\\ 
A_1 & 0& \cdots & 0& A_0 
\end{bmatrix},$$
That is, $A_\ell$ is the zero matrix for $\ell \geq 2.$ 
Now Corollary~\ref{xinbx-2} implies that 
$$C:=X^{-1}BX= \begin{bmatrix} 
A & D &    \cdots& 0 \\
 & \ddots & \ddots&  \vdots \\
\vdots & & \ddots &  D \\
0 & \cdots &  &   A
\end{bmatrix},$$
where $A=A_0+A_1$ and $D = A_1$. Note that the blocks in $C$ are $(n+k) \times (n+k)$ matrices. After performing the corresponding row and column operations to $\bar{C}$,  Claim~\ref{claim-2} implies that
$$f(t) = \widetilde{u}_D =  \sum_{i=1}^{n+k} \sum_{j=1}^{n+k} {\omega_i} A_1(i,j).$$
Here, the coefficients $\omega_1,\ldots, \omega_{n+k}$ are the units in $\ZZ[t^{\pm 1}]$ whose existence is guaranteed by Lemma~\ref{LC-method-2}.

Equation \eqref{A0andA1} implies that
$$
A_1 = \left[
\begin{array}{c|c}
0_{n \times k} & 0_{n \times n} \\ \hline
-I_k & 0_{k \times n} \\
\end{array} \right],
$$
so $(A_1)(n+i,i) = -1$ for $1 \le i \le k$ and $(A_1)_(i,j) = 0$ otherwise.

Thus, we have
\begin{eqnarray*}
f(t) &=& \sum_{i=1}^{n+k} \sum_{j=1}^{n+k} \omega_{i} A_1(i,j) \\
&=& \sum_{i=1}^{k} - \omega_{n+i} \doteq \sum_{i=1}^{k}  \omega_{n+i}.
\end{eqnarray*}

(The last step holds because $f(t)$ is only defined up to multiplication by $\pm t^\ell$.) 
From Lemma~\ref{LC-method-2}, we have that $\omega_{n+i} = t^{\lambda^S_{i}}$ for $1 \le i \le  k$, where recall that $\lambda^S_i$ is the Alexander number on the $i$-th strand of $\beta.$ (That same Alexander number is denoted simply $\lambda_i$ here.) Therefore, it follows that
$$f(t) \doteq \sum_{i=1}^{k} \omega_{n+i} = \sum_{i=1}^{k}t^{\lambda_i}$$
as desired.
\end{proof}

By constructing almost classical braids $\beta$ with a given polynomial $f(t) = \sum_{i=1}^k t^{\lambda_i}$, one can realize $q$-periodic knots $K$ as the closure of $\beta^q$ for many different prime powers $q=p^r.$
We present an example that illustrates this idea.

\begin{figure}[h]
\begin{center}
\includegraphics[scale=0.7]{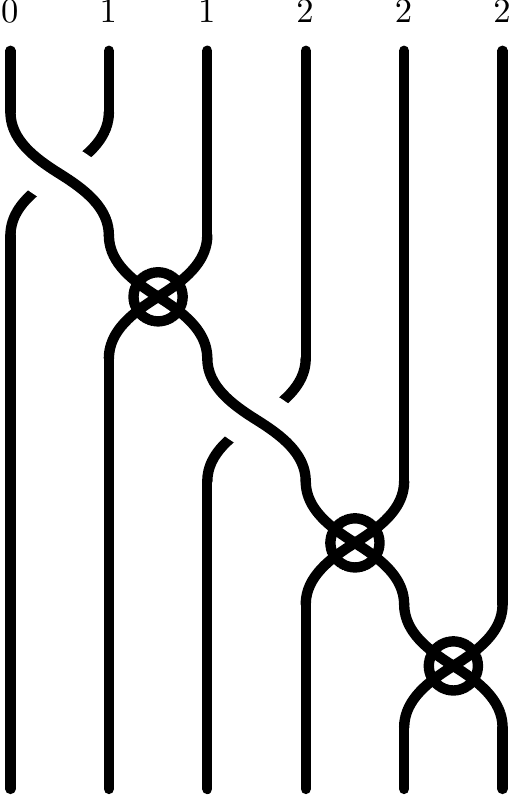}
\end{center} 
\caption{The braid $\beta=\sigma_1 \tau_2\sigma_3\tau_4\tau_5$ is almost classical.} \label{ACbraid-6}
\end{figure}

\begin{example}
Let $f(t) = 1+2t+3t^2$ and set $k = f(1) = 6$ (so $\beta$ will have $6$ strands). We start by labeling the strands at the top of the braid in $\VB_6$ with Alexander numbers,  which are $0,1,1,2,2,2$ going from left to right. (The more general process and reasoning for our choices will be described below). We then form the braid $\beta$ which drags the left-most strand across all the others with either classical or virtual crossings (the choice of classical or virtual is determined by the Alexander numbering). In this case, we get the 
braid $\beta=\sigma_1 \tau_2\sigma_3\tau_4\tau_5,$ see Figure~\ref{ACbraid-6}. Notice that $\beta$ is almost classical and its closure $\widehat{\beta}$ represents the trivial knot. Notice further that $\beta^q$ closes up to a virtual knot (as opposed to a virtual link) for any prime power $q=p^r$ as long as $p \neq 2,3$. For such $q$, the closure $K = \widehat{\beta^q}$ will be a q-periodic almost classical knot diagram with trivial quotient. Applying Theorem~\ref{virtmur-1}, we see that $$
\Delta_K(t) \doteq f(t)^{q-1} = (1+2t+3t^2)^{q-1} \mod p.$$ 
\end{example}

In this way, we can realize many different Alexander polynomials as arising from $q$-periodic almost classical knots. The next result gives a general construction along these lines.

\begin{proposition}
Suppose $f(t) = a_0 + a_1 t + \cdots + a_{n} t^{n} \in \ZZ[t^{\pm 1}]$ is a polynomial satisfying $f(1) \not\equiv 0 \mod p$. Assume also that $a_i > 0$ for $i=0,\ldots, n.$ Then for $q=p^r$, there exists a $q$-periodic almost classical knot $K$ with trivial quotient $K_*$ and
$\Delta_K(t) \doteq (f(t))^{q-1} \mod p$.  
\end{proposition}

\begin{proof}
We will construct an almost classical braid $\beta \in \VB_k$ whose closure $\widehat{\beta}$ is trivial with the property that, for the given polynomial, we have $f(t)  = \sum_{i=1}^k t^{\lambda_i} \mod p,$ where $\lambda_i$ denotes the Alexander number on the $i$-th strand at the top of $\beta$.  The result will then follow by applying Theorem~\ref{virtmur-1} to the closure $K=\widehat{\beta^q}$.

Let $k=f(1)$ and notice that since $f(1) \not\equiv 0 \mod p$, it follows that $k$ is relatively prime to $p$ (otherwise, we could end up with $\widehat{\beta^q}$ being a virtual link, which we explain near the end of the proof).
Let $\beta \in \VB_k$ be a braid with Alexander numbers along the top strands given by
$$\overbrace{0,\ldots, 0}^{a_0},\; \overbrace{1 \ldots, 1}^{a_1}, \ldots, \overbrace{n,\ldots, n}^{a_n}$$
as one goes from left to right.
We form the braid $\beta$ by crossing the left-most strand across all the others, using a negative real crossing whenever the Alexander numbers allow it (In other words, whenever the Alexander number on the left is one less than the Alexander number on the right) and a virtual crossing otherwise (that is, whenever the Alexander numbers on the two strands are equal).
The condition that each coefficient $a_i>0$ is positive ensures that, for every Alexander number $i$ with $0 \leq i \leq n,$ there is at least one strand labeled $i$. (Otherwise, $\widehat{\beta}$ would be a virtual link with two or more components.) 

In any case, we can write $\beta$ as the braid word 
$$\beta = \theta_1 \theta_2 \cdots \theta_{k-1},\quad \text{where} \quad 
\theta_i =\begin{cases} \sigma_i & \text{if $\lambda_{i+1}=\lambda_{i}+1$,}\\
\tau_i & \text{if $\lambda_{i+1}=\lambda_{i}$.}
\end{cases}$$
It follows that $\beta$ is almost classical (since it was Alexander numberable), and the permutation induced by $\beta$ is the $k$-cycle
$(k, k-1, k-2, \ldots, 1)$ (since each $\theta_i$ corresponds to the transposition $(i \ i+1)$), which has order $k$. Notice, as in Figure~\ref{ACbraid-6}, that the $i$-th strand at the top goes to the $(i-1)$-st at the bottom, for $1<i \le k$, and the first strand goes to the $k$-th. This is what the permutation is reflecting. Thus the closure $K_*=\widehat{\beta}$ is a knot, and in fact an easy argument using real and virtual Reidemeister I moves shows that $K_*$ is trivial. If $p$ is a prime which does not divide $k$ (note that the assumptions of our theorem assumed that $k \neq 0 \mod p$, so $p$ cannot divide $k$), then the closure $K =\widehat{\beta^q}$ for $q=p^r$ is also a virtual knot diagram (as opposed to being a virtual link). The number of components in a link are determined by the number of closed cycles when you write the permutation mentioned above as a product of closed cycles. For $\beta^q$, we have $(k, k-1, k-2, \ldots, 1)^q$, which will give us a $k$-cycle as long as $k$ and $q$ are relatively prime, which they are. By construction, it follows that $K$ is $q$-periodic and almost classical, and its quotient $K_* = \widehat{\beta}$ is trivial. Theorem~\ref{virtmur-1} applies to show that
$\Delta_K(t) \doteq (f(t))^{q-1} \mod p$ as claimed.
\end{proof}


\section{Periods of almost classical knots}
\label{sec:periods}

In this section we apply  Theorem~\ref{virtmur-1}
to the problem of determining the possible periods of an almost classical knot.

We recall some facts about Laurent polynomials.
Let $R$ be an integral domain.  A non-zero Laurent polynomial $f \in R[t^{\pm 1}]$, also written as $f(t)$,  can be expressed
uniquely in the form $f(t) =\sum^n_{j=m} a_j t^j$ where $m, n$ are integers with $m \leq n$,   $a_j \in R$
and where $a_m, a_n \neq 0$.
The {\it span} of $f$,   denoted $\fatness f$, is the non-negative integer $n -m$. 
Since $R$ is assumed to have no zero divisors,
$\fatness f  g  = \fatness f + \fatness g $.
We write $f \doteq g $ if there exists a unit $c \in R$ and an integer $j$
such that $c\,  t^j f(t) = g(t)$. 
If $f  \doteq g $ then $\fatness f = \fatness g $.
%
%
If $R$ is a field, then a non-zero $f \in R[t^{\pm 1}]$ has $\fatness f =0$ if and only if $f$ is trivial, that is, $f \doteq 1$.

Given an integral Laurent polynomial $f \in\ZZ[t^{\pm 1}]$, its reduction modulo a prime $p$,  denoted $f \mod p$, is obtained by reducing the coefficients of $f(t)$ modulo $p$.
We denote  the span of $f\mod p$, as an element of $\F_p[t^{\pm 1}]$, by $\fatness_p f$ (assuming $f\mod p$ is not zero).
Observe that $\fatness_p f \leq \fatness f$.

We make use of the following fact about the Alexander polynomial, $\Delta_K$, of an
almost classical knot $K$.
By \cite[Lemma 7.5]{AC},  $\Delta_K(1) =  \pm 1$
and so  $\Delta_K \mod p$ is never the zero polynomial for any prime $p$.

\begin{proposition}\label{Anew}
Let $K$ be an almost classical knot, $p$ a prime, and $r$ a positive integer.
Assume $\Delta_K \mod p$ is not trivial.
If $K$ has period $p^r$, 
then $\fatness_p \Delta_K \geq p^r -1$.
\end{proposition}

\begin{proof}
By Theorem~\ref{virtmur-1}, there are $f, g  \in \ZZ[t^{\pm 1}]$ such that
\[
\Delta_K  \doteq g^{p^r} f^{p^r -1}   \mod p.
\]
It follows that  $\fatness_p \Delta_K = {p^r}\fatness_p g +  (p^r -1) \fatness_p f$. Note that $f$ or $g$ could be trivial.
Since, by assumption, $\Delta_K \mod p$ is not trivial, we have $\fatness_p \Delta_K > 0$  and so one of the numbers  $\fatness_p g$, $\fatness_p f$ is positive. At worst, we have $\fatness_p f = 1$ and $\fatness_p g = 0$. Hence
$\fatness_p \Delta_K \geq p^r -1$.
\end{proof}

An immediate consequence of 
the inequality $\fatness_p \Delta_K \geq p^r -1$ is the following restriction on prime power periods for a given prime $p$.

\begin{corollary} \label{cor:finite-p}
Let $K$ be an almost classical knot and $p$ a prime.
If  $\Delta_K \mod p$ is not trivial, then $p^r$ is a period for $K$ for at most finitely many $r$.
\end{corollary}

\begin{proof}
Since $\fatness_p \Delta_K \le \fatness \Delta_K$ there are finitely many $r$ such that $p^r-1 \le \fatness_p \Delta_K$.
\end{proof}

The next result is a direct consequence of Proposition~\ref{Anew} and Corollary~\ref{cor:finite-p}. 
\begin{corollary}[Finitely many periods]
\label{cor:finitely-many-periods}
If $K$ is an almost classical knot such that \linebreak
$\fatness_p \Delta_K > 0$ for all primes $p$, then $K$ admits only finitely many periods.   
 \end{corollary}
Note that assuming $\fatness_p \Delta_K >0$ is equivalent to assuming that $\Delta_K\mod p$ is non-trivial.
Corollary~\ref{cor:finitely-many-periods} applies in many circumstances and gives a positive answer to Question~\ref{question-1} from the Introduction. For instance, it applies to any classical fibered knot $K$, and shows that such knots are virtually periodic for only finitely many periods. Table~\ref{acks2} lists the Alexander polynomials of the 76 almost classical knots up to 6 crossings, and 
Corollary~\ref{cor:finitely-many-periods} applies in 44 instances to show the given almost classical knot admits only finitely many periods.

The next result applies more generally, but gives a weaker conclusion.

\begin{theorem}[Finitely many prime periods]
\label{thm:finitely-many-periods}
Let $K$ be an almost classical knot such that $\Delta_K$ is not trivial.
If $K$ has prime period $p$,  then $p$ divides a non-zero coefficient of  $\Delta_K$, 
or $\fatness \Delta_K  \geq  p-1$.
In particular, there are at most finitely many primes $p$ for which $K$ has period $p$.
\end{theorem}

\begin{proof}
Let $S$ be the (possibly empty) set of  primes that divide some non-zero coefficient of $\Delta_K$.
Since $\Delta_K$ has only finitely many coefficients, $S$ is finite.
For a prime $p \notin S$, we want to show that, if $K$ has prime period $p$, $\fatness_p K \ge p-1$. Since $p \notin S$, $p$ does not divide any coefficients of $\Delta_K$, and we have $\fatness \Delta_K = \fatness_p \Delta_K$. 
By assumption, $\fatness \Delta_K >0$  and hence for such $p$ we have that $\Delta_K \mod p$ is non-trivial.
Let $p$ be a prime for which $K$ has period $p$, and assume $p \notin S$.
By Proposition~\ref{Anew},  $\fatness \Delta_K = \fatness_p \Delta_K \geq  p-1$.
Since $\fatness \Delta_K \ge p-1$, there are at most finitely many such $p$.
\end{proof}


\smallskip

The following refinement of Proposition~\ref{Anew} will be useful.

\begin{proposition}\label{Bnew}
Let $K$ be an almost classical knot, $p$ a prime, and $r$ a positive integer.
Assume that $\Delta_K \mod p$ is divisible by distinct irreducibles $u_1, \ldots, u_s$
where $s \geq 1$.
If $K$ has period $p^r$, 
then $\fatness_p \Delta_K \geq (p^r -1)\sum^s_{j=1} \fatness_p u_j$.
\end{proposition}

\begin{proof}
By Theorem~\ref{virtmur-1}, there are $f, g  \in \ZZ[t^{\pm 1}]$ such that
\[
\Delta_K  \doteq g^{p^r} f^{p^r -1}   \mod p.
\]
Since  $u_j$ divides $\Delta_K \mod p$ and is irreducible, then $u_j$ divides  either $g$ or $f$ and
hence must appear with multiplicity at least $p^r-1$. Therefore, $(u_1 \cdots u_s)^{p^r-1}$ divides $\Delta_K \mod p$, from which the conclusion follows.
\end{proof}

\begin{corollary}\label{Cnew}
Let $K$ be an almost classical knot, $p$ a prime, and $r$ a positive integer.
Assume that $\Delta_K = u_1 \cdots u_s \mod p$ where the $u_j$'s are distinct irreducible factors
and $s \geq 1$.
If $K$ has period $p^r$ 
then $p=2$ and $r=1$.
\end{corollary}

\begin{proof}
The hypothesis on  $\Delta_K$ implies that $\fatness_p \Delta_K >0$.
By Proposition~\ref{Bnew},  $\fatness_p \Delta_K \geq (p^r -1)\fatness_p \Delta_K$, and so $p^r -1 =1$.
  It follows that $p=2$ and $r=1$.
\end{proof}

\begin{corollary}\label{cor-irr}
Let $K$ be an almost classical knot, $p$ a prime, and $r$ a positive integer.
Assume that $\Delta_K = u_1  \mod p$ ,where $u_1$ is irreducible. Then $K$ cannot have period $p^r$ for $p \neq 2$.
\end{corollary}

Theorem~\ref{virtmur-1} is effective for the analysis of the (virtual) periods of  torus knots.

\begin{theorem}
Let $m,n$ be relatively prime integers with $m, n\geq 2$ and
let $K_{m,n}$ be the classical $(m,n)$-torus knot.
If $p^r$ is a (virtual) period for $K_{m,n}$ then $p$ divides $m$ or $n$, or 
possibly $p=2$ and $r=1$.
\end{theorem}

\begin{proof}
The Alexander polynomial of $K_{m,n}$ is
\[
\Delta_{K_{m,n}} (t) = \frac{(t-1)(t^{mn}-1)}{(t^m -1)(t^n-1)},
\]
see \cite[page 119]{Lick}.
Note that  $\Delta_{K_{m,n}}$ has highest order term $t^{(m-1)(n-1)}$ and
lowest order term $1$.
Hence for any prime $p$,  $\fatness \Delta_{K_{m,n}} = \fatness_p \Delta_{K_{m,n}} = (m-1)(n-1)$,  a positive number since we have assumed $m,n\geq2$.
Observe that $\Delta_{K_{m,n}}$ divides $t^{mn}-1$.
Let $p$ be a prime that does not divide $m$ or $n$.
We have $\tfrac{d~}{dt} \left( t^{mn} - 1 \right) = mn t^{mn-1}$
which is not $0$ modulo $p$ because $p \nmid mn$.
Hence $t^{mn} - 1$ and $mn t^{mn-1}$ are relatively prime as polynomials over $\F_p$
and so $t^{mn} - 1$ does not have multiple roots. The only root of $mnt^{mn-1}$ is $t=0$ (with multiplicity $mn-1$), and this is not a root of $t^{mn}-1$.
Since $\Delta_{K_{m,n}} $ divides $t^{mn} - 1$, it follows that $\Delta_{K_{m,n}} $ does
not have multiple roots as a polynomial over $\F_p$.
If $p^r$ is a virtual period for $K_{m,n}$, then by Theorem~\ref{virtmur-1}
\[
\Delta_{K_{m,n}}   = \left( \Delta_{(K_{m,n})_*}   \right)^{p^r} f^{p^r-1}  \mod p.
\]
If $r > 1$ or if $r = 1$ and $p \neq 2$ then the right hand side has multiple roots, a contradiction.
This shows that if $p^r$ is a virtual period for $K_{m,n}$ then $p$ divides $m$ or $n$ or 
possibly $p=2$ and $r=1$.
\end{proof}


We can also use Theorem~\ref{virtmur-1} to exclude certain composite periods, as illustrated by the example below.

\begin{example} 
\label{threesix}
Let $K$ be any of the knots: 3.6 (the classical trefoil), 5.2160, 6.72938, 6.73053, 6.76479, 6.77833, 6.77844, 6.77985 in Table~\ref{tab:periods}.
The knot $K$ does not have period 6.
\end{example}

\begin{proof}
For $K$ in the given list of knots $\Delta_{K} = t^2 - t +1$. 
Suppose that $K$ has period $6=2 \cdot 3$.
Then by Theorem~\ref{ACrep},
there exists an almost classical braid $\beta$ such that $K =\widehat{\beta^6}$.
Note that $K$ also has period 2 with quotient $\widehat{\beta^3}$,  and period 3  with quotient $\widehat{\beta^2}$.
Since the period 3 and period 2 diagrams are the same (coming from the same period 6 diagram), the polynomial $f(t)$
appearing in Theorem~\ref{virtmur-1}
will be the same, whether we  regard $K$ as having period 2 or 3. (This comes from the definition of $f(t)$ as $\sum^k_{i=1} t^{\lambda_{i}}$ where $k$ is the number of strands in $\beta$.)
By Theorem~\ref{virtmur-1}, we have:
\[
\Delta_{K} = t^2 + t + 1 \doteq (\Delta_{\widehat{\beta^3}})^2f(t) \mod 2.
\]
Note that  $t^2+t+1$ is irreducible modulo 2.  Hence
$$ f(t)  \doteq t^2+t+1  \mod 2$$
and $\Delta_{\widehat{\beta^3}} \doteq 1 \mod 2$.
Again by Theorem~\ref{virtmur-1}, we also have:
$$\Delta_{K} = t^2+2t+1 \doteq (\Delta_{\widehat{\beta^2}})^3(f(t))^2 \mod 3$$
$$\text{ implies }~~ (t+1)^2 \doteq (f(t))^2 \mod 3$$
$$\text{ implies }~~ t+1 \doteq f(t) \mod 3$$
and $\Delta_{\widehat{\beta^2}} \doteq 1 \mod 3$.
Since  $\Delta_{K} = t^2-t+1$ is irreducible over $\ZZ$ and $\Delta_{\widehat{\beta^2}} ~|~ \Delta_{K}$,
we have  $\Delta_{\widehat{\beta^2}} \doteq 1$ or $\Delta_{\widehat{\beta^2}} \doteq \Delta_{K}$.
The latter possibility is excluded because $\Delta_{\widehat{\beta^2}} \doteq 1 \mod 3$ and $\Delta_K \neq 1 \mod 3$,
 and so  $\Delta_{\widehat{\beta^2}} \doteq 1$.
Applying  Theorem~\ref{virtmur-1} again to $\Delta_{\widehat{{\beta}^2}}$ as a period $2$ knot,
\[
1 \doteq \Delta_{\widehat{\beta^2}}  \doteq ( \Delta_{\widehat{\beta}})^2 \, f(t) \mod 2.
\]
Hence $f(t) \doteq 1 \mod 2$,  a contradiction since $f(t)  \doteq t^2+t+1  \mod 2$.
Thus $K$ cannot have period 6.
\end{proof}

Generalizing Example~\ref{threesix}, we give some criteria for the exclusion of composite periods of
the form $2p$, where $p$ is an odd prime.

Let $K$ be an almost classical knot with period $2p$ where $p$ is an odd prime.
By Theorem~\ref{ACrep},
there exists an almost classical braid $\beta$ such that $K =\widehat{\beta^{2p}}$.
Note that $K$ also has period $p$ with quotient $\widehat{\beta^2}$,  and period 2  with quotient $\widehat{\beta^p}$.
The knot $\widehat \beta$ is a quotient of both $\widehat{\beta^2}$  and  $\widehat{\beta^p}$.
By Theorem~\ref{virtmur-1} there is a polynomial $f(t) \in \ZZ[t^{\pm 1}]$ such that

\begin{align}
\Delta_K = \Delta_{\widehat{{\beta}^{2p}}} &  \doteq (\Delta_{\widehat{\beta^p}})^2 f \mod 2 \label{Meq1}\\
\Delta_K = \Delta_{\widehat{{\beta}^{2p}}} & \doteq (\Delta_{\widehat{\beta^2}})^p f^{p-1} \mod p \label{Meq2}\\
\Delta_{\widehat{{\beta}^{2}}} & \doteq (\Delta_{\widehat{\beta}})^2f  \mod 2 \label{Meq3}\\
\Delta_{\widehat{{\beta}^{p}}} & \doteq (\Delta_{\widehat{\beta}})^pf^{p-1} \mod p \label{Meq4}
\end{align}


\begin{proposition}\label{M_1}
Let $K$ be an almost classical knot.
Assume that $\Delta_K \mod p$ is not trivial, where $p$ is an odd prime.
If $\Delta_K$ is irreducible over $\ZZ$ and $\Delta_K \doteq g^2 \mod 2$ where $g$ is irreducible modulo $2$,  then $K$ cannot have period $2p$.
\end{proposition}

\begin{proof}
Suppose that $K$ has period $2p$.
By assumption, $\Delta_K \doteq g^2 \mod 2$ for some polynomial $g$, which is irreducible  modulo 2.
By \eqref{Meq1},
$$\Delta_K \doteq g^2 \doteq (\Delta_{\widehat{\beta^p}})^2 f \mod 2$$

\noindent
Case $1$: $f \doteq 1 \mod 2 $ and
$$\Delta_K \doteq g^2 \doteq (\Delta_{\widehat{\beta^p}})^2 \mod 2$$
Since $\Delta_K$ is irreducible over $\ZZ$, and $\Delta_{\widehat{\beta^p}} ~|~ \Delta_K$, then $\Delta_{\widehat{\beta^p}} = 1 \text{ or } \Delta_{\widehat{\beta^p}} = \Delta_K$. If $\Delta_{\widehat{\beta^p}} = 1$, then we would have 
$$\Delta_K \doteq g^2 \doteq (\Delta_{\widehat{\beta^p}})^2 = 1 \mod 2,$$
which is a contradiction, since $\Delta_K \neq 1 \mod 2$. Therefore, we must have $\Delta_{\widehat{\beta^p}} \doteq \Delta_K$. This gives
$$\Delta_K \doteq g^2 \doteq (\Delta_{\widehat{\beta^p}})^2  = (\Delta_K)^2\mod 2,$$
which is also a contradiction, as $\Delta_K \neq (\Delta_K)^2 \mod 2$ for non-trivial $\Delta_K$. Thus  Case $1$ cannot occur.

\smallskip

\noindent
Case 2: $f \doteq g^2 \mod 2$ and $(\Delta_{\widehat{\beta^p}})^2 \doteq 1 \mod 2$. Again, we have that $\Delta_{\widehat{\beta^p}} = 1 \text{ or } \Delta_{\widehat{\beta^p}} = \Delta_K$. Since $(\Delta_{\widehat{\beta^p}})^2 \doteq 1 \mod 2$, we must have that $\Delta_{\widehat{\beta^p}} = 1$ over $\ZZ$ (as $\Delta_K \neq 1 \mod 2$). Now we use that ${\widehat{\beta^p}}$ is a period $p$ knot. Applying  \eqref{Meq4}, we get
$$1 = \Delta_{\widehat{\beta^p}} \doteq (\Delta_{\widehat{\beta}})^pf^{p-1} \mod p,$$
so that $f \doteq 1 \mod p$.
Regarding $K$ as a period $p$ knot,  \eqref{Meq2} yields
$$\Delta_K \doteq (\Delta_{\widehat{\beta^2}})^pf^{p-1} \mod p$$
$$\text{ implies }~ \Delta_K \doteq (\Delta_{\widehat{\beta^2}})^p \mod p,$$
since  $f \doteq 1 \mod p$. Since $\Delta_K$ is irreducible over $\ZZ$, and $\Delta_{\widehat{\beta^2}} ~|~ \Delta_K$, then $\Delta_{\widehat{\beta^2}} = 1 \text{ or } \Delta_{\widehat{\beta^2}} = \Delta_K$. But $\Delta_K \neq 1 \mod p$, so $\Delta_{\widehat{\beta^2}} \neq 1$, and $\Delta_{\widehat{\beta^2}} \neq \Delta_K$ since $\Delta_K \neq (\Delta_K)^p \mod p$. Thus, Case 2 cannot hold.
Therefore, since neither case holds, $K$ cannot have period $2p$.
\end{proof}

\begin{proposition}\label{M_2}
Let $K$ be an almost classical knot.
Assume that $\Delta_K \mod p$ is not trivial, where $p$ is an odd prime.
If $\Delta_K \doteq g^2$ over $\ZZ$ for some polynomial $g$ which is irreducible modulo $2$, then $K$ cannot have period $2p$.
\end{proposition}

\begin{proof}
Suppose that $K$ has period $2p$. 
\eqref{Meq1} gives
$$\Delta_K \doteq g^2  \doteq (\Delta_{\widehat{\beta^p}})^2 f \mod 2$$

\noindent
Case $1$: $f \doteq 1 \mod 2 $ and
$$\Delta_K \doteq g^2 \doteq (\Delta_{\widehat{\beta^p}})^2 \mod 2$$
Since $\Delta_K = g^2$ over $\ZZ$, and $\Delta_{\widehat{\beta^p}} ~|~ \Delta_K$, then $\Delta_{\widehat{\beta^p}} = 1, g, \text{ or } g^2 = \Delta_K$. If $\Delta_{\widehat{\beta^p}} = 1$, then we would have 
$$\Delta_K \doteq g^2 \doteq (\Delta_{\widehat{\beta^p}})^2 = 1 \mod 2,$$
which is a contradiction, since $\Delta_K \neq 1 \mod 2$. If $\Delta_{\widehat{\beta^p}} \doteq \Delta_K$, this gives
$$\Delta_K \doteq g^2 \doteq (\Delta_{\widehat{\beta^p}})^2  = (\Delta_K)^2\mod 2,$$
which is also a contradiction, as $\Delta_K \neq (\Delta_K)^2 \mod 2$ for non-trivial $\Delta_K$. If $\Delta_{\widehat{\beta^p}} \doteq g$, this gives
$$\Delta_K \doteq g^2 \doteq (\Delta_{\widehat{\beta^p}})^2 = g^2 \mod 2,$$
which holds. Then, since $\widehat{\beta^p}$ is a period $p$ knot,  \eqref{Meq4} gives
$$g = \Delta_{\widehat{\beta^p}} \doteq (\Delta_{\widehat{\beta}})^pf^{p-1} \mod p.$$
But $g$ is irreducible, so this cannot hold by Corollary~\ref{cor-irr}.
Thus Case $1$ cannot hold.

\noindent
Case 2: $f \doteq g^2 \mod 2$ and $(\Delta_{\widehat{\beta^p}})^2 \doteq 1 \mod 2$. Again, we have that $\Delta_{\widehat{\beta^p}} = 1, g, \text{ or } g^2 = \Delta_K$. Since $(\Delta_{\widehat{\beta^p}})^2 \doteq 1 \mod 2$, we must have that $\Delta_{\widehat{\beta^p}} = 1$ over $\ZZ$. Now we use that ${\widehat{\beta^p}}$ is a period $p$ knot. Applying \eqref{Meq4}, we get
$$1 = \Delta_{\widehat{\beta^p}} \doteq (\Delta_{\widehat{\beta}})^pf^{p-1} \mod p,$$
so that $f \doteq 1 \mod p$. Now, looking at $K$ as a period $p$ knot, \eqref{Meq2} gives
$$\Delta_K \doteq (\Delta_{\widehat{\beta^2}})^pf^{p-1} \mod p$$
$$\text{ implies }~ \Delta_K \doteq (\Delta_{\widehat{\beta^2}})^p \mod p,$$
since  $f \doteq 1 \mod p$. Since $\Delta_K = g^2$  over $\ZZ$,
and $\Delta_{\widehat{\beta^2}} ~|~ \Delta_K$, then $\Delta_{\widehat{\beta^2}} = 1,g, \text{ or } g^2 = \Delta_K$. But $\Delta_K \neq 1 \mod p$, so $\Delta_{\widehat{\beta^2}} \neq 1$, and $\Delta_{\widehat{\beta^2}} \neq \Delta_K$ since $\Delta_K \neq (\Delta_K)^p \mod p$.
The last choice is that $\Delta_{\widehat{\beta^2}} = g$, but then we would have $g^2 = \Delta_K \doteq (\Delta_{\widehat{\beta^2}})^p  = g^p \mod p$, and $g^2 \neq g^p \mod p$.
Thus Case 2 cannot hold.
Therefore, since neither case holds, $K$ cannot have period $2p$.
\end{proof}

We provide some examples of how the above results are used to eliminate periods, and a full list of all known and excluded periods for almost classical knots up to 6 crossings is given in Table~\ref{tab:periods}.
Proposition~\ref{Anew} is particularly useful in eliminating  many possible periods for a given knot.
However, if $\Delta_K \doteq 1 \mod p$ for some prime $p$,
then we cannot eliminate any prime power periods of the form $p^r$, 
as  Murasugi's condition  holds trivially. In particular, for knots $K$ with $\Delta_K \doteq 1$ trivial, we are unable to exclude any periods using Murasugi's conditions. This affects the knots 5.2012, 5.2025, 5.2080, 6.72507, 6.72557, 6.72692, 6.72695, 6.72975, 6,73007, and 6.73583.

The remaining knots all have $\Delta_K$ nontrivial, and we make a few general observations about their possible periods. Notice that they all
 have $\fatness \Delta_K \leq 4$, and in each case $\Delta_K \mod p$ is nontrivial for $p \geq 7.$ It follows from Proposition~\ref{Anew} that their only possible prime periods are 2, 3 and 5, though they may admit larger periods, either as prime powers of $2,3,5$, or as composite numbers.

\begin{example}[Knot $4.99$]
$\Delta_K \doteq 2t-1$.

This knot admits a 2-periodic diagram; in fact $K =  \widehat{\beta^2}$ for $\beta = \sigma_1 \tau_2 \sigma_2^{-1} \tau_2.$ 
Proposition~\ref{Anew} implies that $K$ cannot have a prime period for any $p \geq 3$, and
since $\Delta_K = 1 \mod 2$,  we cannot eliminate prime power periods of the form $2^r, r>1$.
This gives a partial answer to Question A in \cite{Kim-Lee-Seo-2014} for 4.99.

Similar considerations apply to almost classical knots $K$ with $\Delta_K(t) \doteq 2t-1$ to show that they have only $2^r, r\ge 1$ as possible periods. This applies to the knots 5.2133, 6.72944, 6.75341, 6.75348, and 6.89815 in Table~\ref{tab:periods}.
\end{example}

\begin{example}[Knot $4.105$]
$\Delta_K \doteq 2t^2-2t+1$.

This knot admits a 4-periodic diagram; in fact  $K=\widehat{\beta^4}$ for $\beta = \sigma_3 \tau_4 \tau_2 \tau_1 \tau_3 \tau_2.$
Proposition~\ref{Anew} implies that $K$ cannot have a prime period $p \ge 5$, and period 3 is eliminated by Corollary~\ref{cor-irr}, as $\Delta_K = 2t^2+t+1 \mod 3$ is irreducible.
Since $\Delta_K = 1 \mod 2$, so we cannot exclude prime power periods of the form $2^r, r>2$.
This gives a partial answer to Question A in \cite{Kim-Lee-Seo-2014} for 4.105.

Similar considerations apply to almost classical knots $K$ with $\Delta_K(t) \doteq 2t^2-2t+1$ to show that they have only $2^r, r\ge 1$ as possible periods. This applies to the knots 6.77908, 6.85613, and 6.89623 in Table~\ref{tab:periods}.
\end{example}

\begin{example}[Knot $4.108$]
$\Delta_K \doteq t^2-3t+1$.

This is the classical knot $4_1,$  which admits a classical diagram with period $p=2$.  
Proposition~\ref{Anew} implies that $K$ cannot have a prime period for $p \ge 5$, and period 3 is excluded by
Corollary~\ref{Cnew}, as $\Delta_K = t^2+1 \mod 3$ is irreducible. As well, the period $4 = 2^2$ is eliminated by Proposition~\ref{Anew}.
Thus $4.108$ has only 2 as a period, and virtual knot diagrams of $4_1$ do not introduce non-classical periods. This answers Question A in \cite{Kim-Lee-Seo-2014} for the figure eight knot.

Similar considerations apply to almost classical knots $K$ with $\Delta_K(t) \doteq t^2-3t+1$ to show that they have only 2 as a possible period. This applies to the knots 6.77905, 6.78358, and 6.79342  in Table~\ref{tab:periods}.
\end{example}

\begin{example}[Knot $5.2426$]
$\Delta_K \doteq (t^2-t+1)^2$.

This knot has no known periods, but notice that $\Delta_K \neq 1 \mod p$ for any prime $p$. 
Note that $t^2 -t +1 \mod p$ is an irreducible factor of $\Delta_K \mod p$ provided $p \neq 3$.
Using Proposition~\ref{Bnew}, with $s=1$ and $u_1=t^2 -t +1 $, if $\fatness_p \Delta_K < (p^r-1)\fatness_p (t^2 -t +1)$, then $K$ cannot have period $p^r$ (for $p \neq 3$). Now, $\fatness_p \Delta_K = 4$, and for prime $p \ge 5$, $(p-1) \fatness_p (t^2 -t +1)=2(p-1) \ge 8$, so $4 = \fatness_p \Delta_K < (p-1) \fatness_p (t^2 -t +1)$, and $K$ cannot have prime period $p \ge 5$. For $p^r = 2^2$, $4 = \fatness_2 \Delta_K < (2^2-1)\fatness_2 (t^2 -t +1) = 6$, so $K$ cannot have period $4=2^2$.
Prime power periods of the form $3^r, r >1$, are eliminated by Proposition~\ref{Anew} (since $4= \fatness_p \Delta_K < 3^r -1$), and the composite period
$6 = 2 \cdot 3$ is eliminated by Proposition~\ref{M_2} as follows. Notice that $\Delta_K \doteq (t^2-t+1)^2 \doteq g^2$ over $\ZZ$ and $\Delta_K \doteq  (t^2+t+1)^2 \doteq  g^2 \mod 2$, where $g$ is irreducible in both cases. As well, $\Delta_K$ is non-trivial modulo 3, so that Proposition~\ref{M_2} applies to show that $K$ cannot have period $q=6$. The only possible periods for $K=5.2426$ are therefore 2 and 3. 

Similar considerations apply to almost classical knots $K$ with $\Delta_K(t) \doteq (t^2-t+1)^2$ to show they have only $2, 3$ as possible periods. This applies to the knots 6.87262, 6.89187, and 6.89198 in Table~\ref{tab:periods}.
\end{example}

\begin{example}[Knot $6.90099$]
$\Delta_K = t^4-t^2+1$.

This knot admits a 3-periodic diagram; in fact $K=\widehat{\beta^3}$ for $\beta = \tau_3 \sigma_3 \sigma_4 \tau_1 \tau_3 \tau_2.$ Proposition~\ref{Anew} rules out prime periods for $p \ge 7$, and also the prime power $3^2$. Corollary~\ref{Cnew} excludes $p=5$ as a period since $\Delta_K = (t^2+ 2t +4)(t^2+3t+4) \mod 5$, with both factors irreducible modulo $5$.
In addition, Proposition~\ref{Bnew} can be used to exclude $2^2$ as a period, since $\Delta_K = (t^2+t+1)^2 \mod 2$ and $t^2+t+1$ is irreducible modulo 2, so $4 = \fatness_2 \Delta_K < (3)(2)$.
To exclude the composite period $q=6=2\cdot 3$, 
observe that $\Delta_K$ is irreducible over $\ZZ$, and that $\Delta_K \neq 1 \mod 3$ and $\Delta_K = (t^2+t+1)^2 = g^2 \mod 2$, where $g$ is irreducible modulo 2.
Hence, Proposition~\ref{M_1} applies to show $K$ cannot have period $q=6$.
Therefore $K$ has the known period 3 and also possibly period 2, but no others.
\end{example}

\begin{example}[Knot $6.90209$]
$\Delta_K = t^4-3t^3+3t^2-3t+1$.

This is the classical knot $6_2$, which admits a classical diagram with period 2.
By Proposition~\ref{Anew}, $K$ cannot have prime power period $p$ for $p \ge 7$.
Since $\Delta_K = (t^2+t+2)(t^2+2t+2) \mod 3$ with both factors irreducible modulo 3, 
Corollary~\ref{Cnew} excludes $p=3.$ 
Further, since $\Delta_K = (t^2+t+1)^2 \mod 5$, where $t^2+t+1$ is irreducible modulo 5, Proposition~\ref{Bnew} excludes $p=5$ as well, as $\fatness_5 \Delta_K = 4 < (4)(2)$. 
Note that  $\Delta_K = t^4+t^3+t^2+t+1 \mod 2$, which is irreducible modulo 2,
thus Proposition~\ref{Bnew} applies again to show that $K$ cannot have period $2^2$. 
Thus $K$ has only the known period 2 and no others.  In particular, we see that virtual knot diagrams of the classical knot $6_2$ do not introduce any non-classical periods. 
\end{example}


\renewcommand{\arraystretch}{1.00}
\begin{table}[ht] 

\begin{tabular}{cc}
\begin{minipage}{0.50\textwidth}
\begin{tabular}{|l|c|}
\hline
Knot & Alexander polynomial  \\
\hline \hline
{\bf 3.6} & $t^2-t+1$\\ \hline
4.99  & $2t-1$ \\ \hline
4.105 & $2t^2-2t+1$ \\ \hline
{\bf 4.108}  & $t^2-3t+1 $ \\  \hline
5.2012 & $1$ \\  \hline
5.2025 & $1 $ \\  \hline
5.2080 & $1$ \\  \hline
5.2133 & $2t-1$ \\  \hline
5.2160 & $t^2-t+1 $ \\  \hline
5.2331 & $t^3-t+1 $ \\  \hline
5.2426 & $(t^2-t+1)^2$ \\  \hline
5.2433 & \;\; $t^4-2t^3 +4t^2 -3t +1 $ \;\; \\  \hline
{\bf 5.2437} & $2t^2 - 3t + 2 $ \\  \hline
5.2439 & $t^3-2t^2 +3t -1$ \\  \hline
{\bf 5.2445} & $ t^4 - t^3 + t^2 - t + 1$\\  \hline
6.72507 & $1$ \\  \hline
6.72557 &$1$ \\  \hline
6.72692 & $1 $ \\  \hline
6.72695 & $1$ \\  \hline
6.72938 & $t^2-t+1  $ \\  \hline
6.72944 & $2t-1  $ \\  \hline
6.72975 & $1$ \\  \hline
6.73007 & $1$ \\  \hline
6.73053 & $t^2-t+1 $ \\  \hline
6.73583 &$1$ \\  \hline
6.75341 & $2t-1 $ \\  \hline
6.75348 & $t-2$ \\  \hline
6.76479 & $t^2-t+1 $ \\  \hline
6.77833 & $t^2-t+1 $ \\  \hline
6.77844 & $t^2-t+1 $ \\  \hline
6.77905 & $t^2-3t+1$ \\  \hline
6.77908 & $2t^2-2t+1 $ \\  \hline
6.77985 & $t^2-t+1 $ \\  \hline
6.78358 & $t^2-3t+1 $ \\  \hline
6.79342 & $t^2-3t+1$ \\  \hline
6.85091 & $t^2+t-1 $ \\  \hline
6.85103 & $t^3-t^2+2t-1 $ \\  \hline
6.85613\;\; & $t^2-2t+2 $ \\  \hline

\end{tabular}
\end{minipage}

\begin{minipage}{0.50\textwidth}
 \begin{tabular}{|l|c|} 
\hline
Knot  & Alexander polynomial   \\ 
\hline \hline
6.85774 & $t^3 - t^2 + 1 $ \\  \hline
6.87188 & $(2t - 1)(t^2 - t + 1) $ \\  \hline
6.87262 & $(t^2-t+1)^2$ \\  \hline
6.87269 & $(2t-1)^2$ \\  \hline
6.87310 & $t^4 -t^3 +2t^2 -2t +1 $ \\  \hline
6.87319 & $3t^2-3t+1 $ \\  \hline
6.87369 & $t^3-2t^2+3t-1 $ \\  \hline
6.87548 & $t^3-2t^2 -t+1 $ \\  \hline
6.87846 & $t^3-t^2+2t-1  $ \\  \hline
6.87857 & $t^2-4t+2 $ \\  \hline
6.87859 & $3t^2-3t+1 $ \\  \hline
6.87875 & $t^3+t^2-2t+1 $ \\  \hline
6.89156 &  $2t^3-t^2-t+1 $ \\  \hline
{\bf 6.89187} & $(t^2-t+1)^2$ \\  \hline
{\bf 6.89198} & $(t^2-t+1)^2$ \\  \hline
6.89623 & $2t^2-2t+1 $ \\  \hline
6.89812 & $t^3 -2t+2 $ \\  \hline
6.89815 & $2t - 1 $ \\  \hline
6.90099 & $t^4 - t^2 + 1 $ \\  \hline
6.90109 & \;\; $(2t^2-2t+1)(t^2-t+1)$ \;\; \\  \hline
6.90115 & $(t^2 - 3t + 1)(t^2 - t + 1)$ \\  \hline
6.90139 & $(3t^2 - 3t + 1)(t^2 - t + 1)$ \\  \hline
6.90146 & $t^4 - 5t^3 + 9t^2 - 5t + 1$ \\  \hline
6.90147 & $t^4-3t^3+6t^2-5t+2 $ \\  \hline
6.90150 & $t^4 -5t^3+6t^2-4t+1 $ \\  \hline
6.90167 & $t^4-2t^3+4t^2 -4t+2 $ \\  \hline
{\bf 6.90172} & $t^4-3t^3+5t^2-3t+1 $ \\  \hline
6.90185 & $3t^4-6t^3+6t^2-3t+1$ \\  \hline
6.90194 & $t^4-4t^3+8t^2-5t+1$ \\  \hline
6.90195 & $2t^4-3t^3+3t^2-2t+1 $ \\  \hline
{\bf 6.90209} & $t^4-3t^3+3t^2-3t+1 $ \\  \hline
6.90214 & $3t^2-4t+2 $ \\  \hline
6.90217 & $t^3-4t^2+3t-1 $ \\  \hline
6.90219 & $2t^3 - 3t^2 + 3t - 1$ \\  \hline
{\bf 6.90227} & $(t - 2)(2t - 1)$ \\  \hline
6.90228 & $4t^2 - 6t + 3$ \\  \hline
6.90232 & $2 t^3 - 6 t^2 + 4t - 1$ \\  \hline
6.90235 \;\; & $t^4 -3t^3+5t^2-3t+1 $ \\  \hline  
\end{tabular}
\end{minipage}
\end{tabular}
\bigskip
\caption{Alexander polynomials for almost classical knots up to six crossings.
(Boldface is used to indicate a classical knot.)}
\label{acks2}
\end{table}

\renewcommand{\arraystretch}{1.20}
\begin{table}
\begin{center}
\small
\begin{tabular}{|l||c|c|}
\hline
Knot & period & Periodic virtual braid  \\ \hline
$3.6=3_1$ &  2, 3 & $(\sigma_1\sigma_2)^2,$ \; $\sigma_1^3$  \\ \hline
$4.99$ & 2& $(\sigma_1 \tau_2 \sigma_2^{-1} \tau_2)^2$ \\ \hline
$4.105$ &  4 & $(\sigma_2 \tau_1)^4$ \\ \hline
$4.108 =4_1$ &  2 & $(\sigma_1 \sigma_2^{-1})^2$ \\ \hline
$5.2433$ & 5 &  $(\sigma_2 \tau_1 \tau_3)^5$ \\ \hline
$5.2437=5_2$ &2 &  $(\sigma_4 \tau_2\tau_1 \sigma_2 \tau_2 \sigma_3)^2$ \\ \hline
$5.2445=5_1$ & 2, 5 &  $(\sigma_1\sigma_2 \sigma_3 \sigma_4)^{2}$, \;  $\sigma_1^{5}$ \\ \hline
$6.85091$ & 2 & $(\tau_3\tau_4\tau_3  \sigma_2 \sigma_3  \tau_1 \tau_2\sigma_3^{-1} \tau_2 \tau_1)^2$ \\ \hline
$6.87262$ & 3 & $(\tau_1\tau_3\tau_2\sigma_2\tau_2 \tau_1\sigma_2)^3$ \\ \hline
$6.87269$ & 3 & $(\sigma_2\tau_1 \sigma_2^{-1} \tau_1\tau_3)^3$ \\ \hline
$6.87310$ & 3& $(\tau_1\tau_2\tau_3\tau_2\sigma_2\tau_2 \tau_3\sigma_3)^3$\\ \hline
$6.87319$ & 3& $(\sigma_2^{-1}\tau_2 \tau_1\sigma_3 \tau_2)^3$  \\ \hline
$6.87857$ & 2& $(\tau_4\tau_3\tau_2 \tau_1 \tau_2 \sigma_1  \tau_2\sigma_2^{-1} \tau_1 \tau_2  \sigma_3)^2$ \\ \hline
$6.87859$ & 2 & $(\tau_4\tau_3\tau_2 \sigma_2  \tau_2\sigma_1^{-1} \tau_3  \sigma_3^{-1})^2$\\ \hline
$6.89156$ & 2 & $(\tau_4\tau_2 \sigma_2 \sigma_3 \tau_4\sigma_4  \tau_2 \tau_1)^2$  \\ \hline
$6.89187 =3_1\# 3_1$ & 2&  $(\sigma_2^3 \tau_1)^2$ \\ \hline
$6.89812$ & 2 & $(\sigma_3 \tau_2\tau_3  \sigma_3 \tau_1 \tau_2 \tau_3 \sigma_4 \tau_1 \tau_2)^2$ \\ \hline
$6.89815$ & 2 & $(\tau_4\tau_2\tau_3 \tau_2 \sigma_1  \sigma_2 \tau_4  \sigma_4^{-1})^2$ \\ \hline
$6.90099$ & 3 & $(\tau_4\tau_3\tau_2\sigma_2\sigma_1\tau_2)^3$  \\ \hline
$6.90139$ & 6 & $(\sigma_2 \tau_4 \tau_6 \tau_1 \tau_3 \tau_5)^6$ \\ \hline
$6.90146$ & 3 &  $(\tau_2 \sigma_2 \sigma_3^{-1} \tau_2 \tau_1)^3$ \\ \hline
$6.90172 =6_3$ & 2 &  $(\sigma_4  \sigma_2^{-1} \sigma_3^{-1} \sigma_1 \sigma_4^{-1} \sigma_2 \sigma_3 \sigma_1^{-1})^2$ \\ \hline
$6.90185$ & 3 & $(\sigma_3 \tau_2 \sigma_3 \tau_2)^3$ \\ \hline
$6.90194$ & 3 & $(\tau_3 \sigma_1 \sigma_2^{-1})^3$ \\ \hline
$6.90209=6_2$ & 2 & $(\sigma_4^{-1}\sigma_3^{-1}\sigma_2^{-1}\sigma_1)^2$ \\ \hline
$6.90214$ & 2 & $(\tau_3\tau_4\tau_3 \tau_2 \sigma_2 \tau_2 \sigma_1  \sigma_3)^2$  \\ \hline
$6.90219$ & 2 &   $(\tau_4\tau_3\tau_2 \sigma_2 \tau_2 \sigma_1 \tau_3 \sigma_3^{-1})^2$ \\ \hline
$6.90227=6_1$ & 2 &  $(\tau_3\tau_2\tau_1 \sigma_2 \tau_2 \tau_3 \sigma_2^{-1} \sigma_4)^2$ \\ \hline
$6.90228$ & 6 & $(\tau_1 \sigma_3 \tau_5 \tau_2 \tau_4 \tau_6)^6$  \\ \hline
$6.90232$ &  2&  $(\tau_4\tau_3\tau_2 \sigma_1 \tau_2 \sigma_2 \sigma_3^{-1} \tau_3)^2$  \\ \hline
$6.90235$ &  3 &$(\tau_1 \tau_4\tau_6\tau_5 \tau_4 \tau_3\tau_4 \tau_5 \sigma_3 \sigma_2^{-1})^3$ \\ \hline
\end {tabular}
\bigskip
\caption{\label{tab:periodic-braids} Periodic almost classical knots as closures of periodic virtual braids.}
\end{center}
\end{table}

\renewcommand{\arraystretch}{1.00}
\begin{table}
\begin{center}
\small
\begin{tabular}{|l||c|ccc|c|}
\hline
 &   & \multicolumn{3}{c|}{\bf Excluded Periods} & \bf{Non-Excluded}\\  
Knot & \makecell{Known \\ Periods} & Prime & \makecell{Prime \\Powers} & Composite & \makecell{{\bf Periods}\\ {\quad}} \\ \hline
$3.6=3_1$ &  2, 3 & $p \ge 5$ & $2^2$ & $2 \cdot 3$ & none \\ \hline
$4.99$ & 2& $p \ge 3$ & - &  N/A  & $2^k$, $k \ge 2$ \\ \hline
$4.105$ &  4 & $p \ge 3$ & - & N/A & $2^k$, $k \ge 3$ \\ \hline
$4.108 =4_1$ &  2 & $p \ge 3$ & $2^2$ & N/A & none \\ \hline
$5.2012$ & - &  - & - &  - & $q \ge 2$ \\ \hline
$5.2025$ & - &  - & - &  - & $q \ge 2$\\ \hline
$5.2080$ & - &  - & - &  - & $q \ge 2$ \\ \hline
$5.2133$ & - &  $p \ge 3 $ & - & N/A & $2^k$, $k \ge 1$  \\ \hline
$5.2160$ & - &  $p \ge 5$ & $2^2, 3^2$ & $2 \cdot 3$ & $2,3$  \\ \hline
$5.2331$ & - &  $p \ge 3 $ & $2^2$ & N/A  & $2$ \\ \hline
$5.2426$ & - &  $p \ge 5 $ & $2^2, 3^2$ & $2 \cdot 3$ & $2,3$\\ \hline
$5.2433$ & 5 &  
$p \neq 2,5$ & $2^2, 5^2$ & $2 \cdot 5$ & $2$ \\ \hline
$5.2437=5_2$ &2 &  $p \ge 3$ & - & N/A & $2^k$, $k \ge 2$ \\ \hline
$5.2439$ & - &  $p \ge 3$ & $2^2$ & N/A & none \\ \hline
$5.2445=5_1$ & 2, 5 &  
$p \neq 2,5$ & $2^2, 5^2$ & $2 \cdot 5$ & none \\ \hline
$6.72507$ & - &  - & - &  - & $q \ge 2$ \\ \hline
$6.72557$ & - &  - & - &  - & $q \ge 2$ \\ \hline
$6.72692$ & - &  - & - &  - & $q \ge 2$ \\ \hline
$6.72695$ & - &  - & - &  - & $q \ge 2$ \\ \hline
$6.72938$ & - &  $p \ge 5$ & $2^2, 3^2$ & $2 \cdot 3$  & $2,3$ \\ \hline
$6.72944$ & - &  $p \ge 3$ & - & N/A  & $2^k$, $k \ge 1$ \\ \hline
$6.72975$ & - &  - & - &  - & $q \ge 2$ \\ \hline
$6.73007$ & - &  - & - &  - & $q \ge 2$ \\ \hline
$6.73053$ & - &  $p \ge 5$ & $2^2, 3^2$ & $2 \cdot 3$  & $2,3$ \\ \hline
$6.73583$ & - &- &  - &  - & $q \ge 2$ \\ \hline
$6.75341$ & - & $p \ge 3$ & - & N/A & $2^k$, $k \ge 1$ \\ \hline
$6.75348$ & - &  $p \ge 3$ & - & N/A & $2^k$, $k \ge 1$ \\ \hline
$6.76479$ & - &  $p \ge 5$ & $2^2, 3^2$ & $2 \cdot 3$ & $2,3$ \\ \hline
$6.77833$ & - & $p \ge 5$ & $2^2, 3^2$ & $2 \cdot 3$ & $2,3$ \\ \hline
$6.77844$ & - &  $p \ge 5$ & $2^2, 3^2$ & $2 \cdot 3$  & $2,3$ \\ \hline
$6.77905$ & - &  $p \ge 3$ & $2^2$ & N/A & $2$ \\ \hline
$6.77908$ & - &  $p \ge 3$ & - & N/A & $2^k$, $k \ge 1$ \\ \hline
$6.77985$ & - &  $p \ge 5$ & $2^2, 3^2$ & $2 \cdot 3$ & $2,3$ \\ \hline
$6.78358$ & - &  $p \ge 3$ & $2^2$ & N/A & $2$ \\ \hline
$6.79342$ & - &  $p \ge 3$ & $2^2$ & N/A & $2$ \\ \hline
$6.85091$ & 2 & $p \ge 3$ & $2^2$ & N/A & none \\ \hline
$6.85103$ & - &  $p \ge 3$ & $2^2$ & N/A & $2$ \\ \hline
$6.85613$ & - &  $p \ge 3$ & - & N/A  & $2^k$, $k \ge 1$ \\ \hline
$6.85774$ & - &  $p \ge 3$ & $2^2$ & N/A & $2$ \\ \hline
$6.87188$ & - &   $p \ge 3$ & $2^2$ & N/A & $2$ \\ \hline
$6.87262$ & 3 & $p \ge 5$ & $2^2, 3^2$ & $2 \cdot 3$ & $2$ \\ \hline
\end {tabular}
\end{center}
\end{table}

\begin{table}
\begin{center}
\small 
 \begin{tabular}{|l||c|ccc|c|}
\hline
 &   & \multicolumn{3}{c|}{\bf Excluded Periods} & \bf{Non-Excluded}\\  
Knot & \makecell{Known \\ Periods} & Prime & \makecell{Prime \\Powers} & Composite & \makecell{{\bf Periods}\\ {\quad}} \\ \hline
$6.87269$ & 3 & $p \ge 5$ & $3^2$ &  -  & \makecell{$2^k, k\ge 1$ \\ $3 \cdot 2^k$, $k \ge 1$}\\ \hline
$6.87310$ & 3& $p \ge 5$ & $2^2, 3^2$ & $2 \cdot 3$ & $2$\\ \hline
$6.87319$ & 3& $p \ge 5 $ & $2^2$ & - & \makecell{$2, 3^k, k \ge 2$ \\ $2 \cdot 3^k, k \ge 1$}  \\ \hline
$6.87369$ & -& $p \ge 3$ & $2^2$ & N/A  & $2$   \\ \hline
$6.87548$ & - & $p \ge 3$ & $2^2$ & N/A & $2$ \\ \hline
$6.87846$ &  - & $p \ge 3$ & $2^2$ & N/A & $2$ \\ \hline
$6.87857$ & 2& $p \ge 3$ & - & N/A  & $2^k$, $k \ge 2$ \\ \hline
$6.87859$ & 2 & $p \ge 5$  & $2^2$ & - & \makecell{$3^k, k \ge 1$ \\ $2 \cdot 3^k$, $k \ge 1$}\\ \hline
$6.87875$ & - & $p \ge 3$ & $2^2$ & N/A & $2$ \\ \hline
$6.89156$ &  2 & $p \ge 3$ & $2^2$ & N/A & none  \\ \hline
$6.89187 =3_1\# 3_1$ &  2 & $p \ge 5$ & $2^2, 3^2$ & $2 \cdot 3$ & $3$ \\ \hline
$6.89198 =3_1\# 3_1^*$ & - & $p \ge 5$ & $2^2, 3^2$ & $2 \cdot 3$ & $2,3$ \\ \hline
$6.89623$ &  - & $p \ge 3$ & - & N/A & $2^k$, $k \ge 1$ \\ \hline
$6.89812$ & 2 & $p \ge 3$ & - & N/A & $2^k$, $k \ge 2$ \\ \hline
$6.89815$ &  - & $p \ge 3$ & - & N/A  & $2^k$, $k \ge 1$ \\ \hline
$6.90099$ & 3 & $p \ge 5$ & $2^2, 3^2$ & $2 \cdot 3$ & $2$  \\ \hline
$6.90109$ & - & $p \ge 3$ & $2^2$ & N/A & $2$ \\ \hline
$6.90115$ & - &  $p \ge 3$ & $2^2$ & N/A & $2$\\ \hline
$6.90139$ & 6 & $p \ge 5$ & $2^2, 3^2$ & - & none \\ \hline
$6.90146$ &  3 &  $p \ge 5$ & $2^2, 3^2$ & $2 \cdot 3$ & $2$ \\ \hline
$6.90147$ & - & $p \ge 3$ & $2^2$ & N/A & $2$ \\ \hline
$6.90150$ & - & $p \ge 3$ & $2^2$ & N/A & $2$ \\ \hline
$6.90167$ & - & $p \ge 3$ & - & N/A & $2^k$, $k \ge 1$ \\ \hline
$6.90172 =6_3$ & 2 & $p \ge 5$ & $2^2, 3^2$ & $2 \cdot 3$ & $3$ \\ \hline
$6.90185$ & 3 & $p \ge 5$ & $2^2$ & - & \makecell{$2,3^k,k \ge 2$ \\ $2 \cdot 3^k, k \ge 1$} \\ \hline
$6.90194$ & 3 & $p \ge 5$ & $2^2$ & $2 \cdot 3$ & $2$ \\ \hline
$6.90195$ &  - & $p \ge 3$ & $2^2$ & N/A & $2$ \\ \hline
$6.90209=6_2$ & 2 & $p \ge 3$ & $2^2$ &  N/A & none\\ \hline
$6.90214$ & 2 & $p \ge 3$ & - & N/A & $2^k$, $k \ge 2$  \\ \hline
$6.90217$ & - & $p \ge 3$ & $2^2$ & N/A & $2$ \\ \hline
$6.90219$ & 2 &   $p \ge 3$ & $2^2$ & N/A & none \\ \hline
$6.90227=6_1$ & 2 &  $p \ge 5$ & $3^2$ & - & \makecell{$2^k,k \ge 2$ \\ $3 \cdot 2^k, k \ge 1$} \\ \hline
$6.90228$ & 6 & $p \ge 5$ & - & - & \makecell{$2^j,3^k, j,k \ge 2$ \\ $2^j \cdot 3^k$, $j,k \ge 2$}  \\ \hline
$6.90232$ &  2&  $p \ge 3$ & - & N/A & $2^k$, $k \ge 2$  \\ \hline
$6.90235$ &  3 &$p \ge 5$ & $2^2, 3^2$ & $2 \cdot 3$ & $2$ \\ \hline
\end {tabular}
\bigskip
\caption{\label{tab:periods} Known and excluded periods of almost classical knots. Classical knots are indicated by their Rolfsen number.}
\end{center}
\end{table}

\clearpage   

\bibliographystyle{amsalpha}

\end{document}